\documentclass[11pt,final]{article}%
\usepackage{stix2}
\usepackage{ifdraft}
\usepackage{appendix}
\usepackage{amsfonts}
\usepackage{amsmath}
\usepackage{amssymb}
\usepackage{amsxtra,amscd}
\usepackage[amsmath,hyperref,thmmarks]{ntheorem}
\usepackage{makeidx}
\usepackage{graphicx}
\usepackage{cprotect}
\usepackage[colorlinks=true,bookmarksnumbered=true,pdfpagemode=None,final]%
{hyperref}%
\providecommand{\U}[1]{\protect\rule{.1in}{.1in}}
\ifdraft{}{\usepackage{version}}
\ifdraft{}{\excludeversion{nt}}
\theoremnumbering{arabic}
\theoremheaderfont{\scshape}
\RequirePackage{latexsym}
\theorembodyfont{\slshape}
\theoremseparator{.}
\newtheorem{X}{X}[section]

\newtheorem{conjecture}[X]{Conjecture}
\newtheorem{corollary}[X]{Corollary}

\newtheorem{lemma}[X]{Lemma}

\newtheorem{proposition}[X]{Proposition}

\newtheorem{theorem}[X]{Theorem}
\newtheorem{theoremn}{Theorem}

\newtheorem*{conjecturea}{Conjecture (A)}
\newtheorem*{conjectureb}{Conjecture (B)}
\newtheorem*{conjecturec}{Conjecture (C)}
\newtheorem*{conjectured}{Conjecture (D)}
\theorembodyfont{\upshape}
\newtheorem{application}[X]{Application}
\newtheorem{definition}[X]{Definition}
\newtheorem{example}[X]{Example}
\newtheorem{exercise}[X]{Exercise}

\newtheorem{plain}[X]{}

\newtheorem{question}[X]{Question}
\newtheorem{caution}[X]{Caution}
\newtheorem{remark}[X]{Remark}
\newtheorem{summary}[X]{Summary}

\theorembodyfont{\small}
\newtheorem{aside}[X]{Aside}
\newtheorem*{note}{Notes}
\theorembodyfont{\normalsize}
\theoremstyle{nonumberplain}
\theoremsymbol{\ensuremath{_\Box}}
\newtheorem{proof}{Proof}

\qedsymbol{\ensuremath{_\Box}}

\newcommand{\badjustwidth}{\begin{adjustwidth}{5em}{0em}}
\newcommand{\eadjustwidth}{\end{adjustwidth}}

\newcommand{\df}{\smash{\lower.12em\hbox{\textup{\tiny def}}}}
\newcommand{\lsimeq}{\smash{\lower.24em\hbox{$\scriptstyle\simeq$}}}
\newcommand{\lsim}{\smash{\lower.30em\hbox{$\scriptstyle\sim$}}}
\newcommand{\lapprox}{\smash{\lower.30em\hbox{$\scriptstyle\approx$}}}
\newcommand{\CH}{\mathrm{CH}}

\def\mapsto{\DOTSB\mathchar"39AD }

\usepackage{wrapfig}

\usepackage{changepage}

\let\quoteOld\quote
\let\endquoteOld\endquote
\renewenvironment{quotation}{\small\quoteOld}{\endquoteOld}

\usepackage[overload]{textcase}

\usepackage{xcolor}
\definecolor{bblue}{rgb}{0.0, 0.0, 0.6}

\usepackage{array}

\usepackage{microtype}

\usepackage{tikz}
\usepackage{tikz-cd}
\usetikzlibrary{matrix,arrows,positioning,decorations.pathmorphing}
\usetikzlibrary{decorations.markings,shapes.geometric}
\tikzset{commutative diagrams/column sep/Huge/.initial=12ex}
\tikzcdset{arrow style=math font}

\usepackage[shortlabels]{enumitem}
\setitemize[1]{label=$\diamond$\hspace{0.07in}}
\setlist{topsep=0.1em,itemsep=0.1em,parsep=0.1em}

\setdescription{font=\normalfont}

\usepackage{url}
\usepackage{verbatim}

\usepackage[notcite,notref]{showkeys}

\usepackage{mathtools}

\usepackage{titletoc,titlesec}
\contentsmargin{2.0em}
\dottedcontents{section}[2.3em]{}{1.8em}{1pc}
\usepackage[nottoc]{tocbibind}

\titleformat*{\section}{\LARGE\bfseries}
\titleformat*{\subsection}{\Large\itshape}
\titleformat*{\subsubsection}{\large\scshape}
\titleformat*{\paragraph}{\scshape}

\usepackage{natbib}
\let\cite\citealt

\hypersetup{linkcolor=bblue,anchorcolor=bblue,citecolor=bblue,filecolor=bblue,urlcolor=bblue}
\hypersetup{pdftitle={Shimura Varieties},pdfauthor={J.S. Milne}}
\hypersetup{pdfsubject={none},pdfkeywords={none}}

\usepackage[papersize={6.5in,10.0in},margin=0.5in,pdftex]{geometry} 

\newcommand\babstract{\begin{abstract}}\newcommand\eabstract{\end{abstract}}
\newcommand{\bcomment}{}
\newcommand{\bfootnotesize}{\begin{footnotesize}}\newcommand\efootnotesize{\end{footnotesize}}
\newcommand{\bquote}{\begin{quote}}\newcommand\equote{\end{quote}}
\newcommand{\bsmall}{\begin{small}}\newcommand\esmall{\end{small}}
\newcommand{\btable}{\begin{table}}\newcommand{\etable}{\end{table}}
\newcommand{\dstyle}{\displaystyle}
\newcommand{\edocument}{\end{document}}

\newcommand{\tableoc}{\tableofcontents}

\renewcommand{\Gamma}{\varGamma}
\renewcommand{\Pi}{\varPi}
\renewcommand{\Sigma}{\varSigma}

\def\1{{1\mkern-7mu1}}

\DeclareMathOperator{\Ad}{Ad}

\DeclareMathOperator{\dlim}{\varinjlim}
\DeclareMathOperator{\End}{End}

\DeclareMathOperator{\Ext}{Ext}

\DeclareMathOperator{\Gal}{Gal}
\DeclareMathOperator{\GL}{GL}

\DeclareMathOperator{\GU}{GU}

\DeclareMathOperator{\Hom}{Hom}
\DeclareMathOperator{\id}{id}

\DeclareMathOperator{\Ind}{Ind}

\DeclareMathOperator{\inn}{inn}
\DeclareMathOperator{\inv}{inv}

\DeclareMathOperator{\Ker}{Ker}

\DeclareMathOperator{\Mor}{Mor}

\DeclareMathOperator{\MT}{MT}

\DeclareMathOperator{\ob}{ob}
\DeclareMathOperator{\ord}{ord}

\DeclareMathOperator{\PGL}{PGL}
\DeclareMathOperator{\Pic}{\mathrm{Pic}}
\DeclareMathOperator{\plim}{\varprojlim}

\DeclareMathOperator{\pr}{pr}

\DeclareMathOperator{\Res}{Res}
\DeclareMathOperator{\Sh}{Sh}

\DeclareMathOperator{\SL}{SL}
\DeclareMathOperator{\SO}{SO}

\DeclareMathOperator{\Spec}{Spec}

\DeclareMathOperator{\SU}{SU}

\DeclareMathOperator{\SMT}{SMT}

\DeclareMathOperator{\Tr}{Tr}

\newcommand{\Aff}{\mathsf{Aff}}
\newcommand{\CM}{\mathsf{CM}}
\newcommand{\Hdg}{\mathsf{Hdg}}

\newcommand{\Isoc}{\mathsf{Isoc}}
\newcommand{\LCM}{\mathsf{LCM}}
\newcommand{\LMot}{\mathsf{LMot}}

\newcommand{\Mot}{\mathsf{Mot}}

\newcommand{\Rep}{\mathsf{Rep}}

\newcommand{\Vc}{\mathsf{Vec}}

\newcommand{\Fib}{\textsc{Fib}}

\begin{document}

\title{Abelian motives and Shimura varieties\\in nonzero characteristic}
\date{October 13, 2025, v2.0}
\author{J.S. Milne}
\maketitle

\hfill\begin{minipage}{4.0in}
\textit{For me, [the Hodge conjecture] is part of the story of motives,
and it is not crucial whether it is true or false. If it is true, that's very good, and it
solves a large part of the problem of constructing motives
in a reasonable way. If one can find another purely algebraic
notion of cycles for which the analogue of the Hodge conjecture
holds,
this will serve the same purpose,
and I would be as happy as if the Hodge conjecture were proved.
For me it is motives, not Hodge, that is crucial.}\\
\hspace*{\fill}Deligne, May 2013.\footnotemark
\end{minipage}\footnotetext{Interview on the award of the Abel prize,
Eur.\ Math.\ Soc.\ Newsl.\ No.\ 89 (2013), 15--23.}\bigskip

\babstract This article represents my attempt to construct a theory of Shimura
varieties as simple and elegant as that in Grothendieck's \textquotedblleft
paradis motiviques\textquotedblright\thinspace, but without assuming the
Hodge, Tate, or standard conjectures.

\bigskip Deligne's theorem (1982)\nocite{deligne1982}, that Hodge classes on
abelian varieties are absolutely Hodge, allows us to construct a category of
abelian motives $\Mot(k)$ over any field $k$ of characteristic zero. This is a
tannakian category over $\mathbb{Q}$ with most of the properties anticipated
for Grothendieck's category of abelian motives, and it equals Grothendieck's
category if the Hodge conjecture is true for abelian varieties. Deligne's
theorem makes it possible to realize Shimura varieties of abelian type with
rational weight as moduli varieties (\cite{milne1994b}), which greatly
simplifies the theory of Shimura varieties in characteristic zero.

The goal of this article is to extend the theory to characteristic $p$.

\bigskip We study elliptic modular curves by realizing them as moduli curves
for elliptic curves. This works, not only in characteristic zero, but also in
mixed characteristic and characteristic $p$. Some Shimura curves cannot be
realized as moduli curves, but a trick of Shimura allows us to deduce their
properties from those that can.

In this article, we suggest an approach that makes the theory of Shimura
varieties of abelian type as simple, at least conceptually, as that of Shimura curves.

\bigskip Much of the work on Shimura varieties over the last thirty years has
been devoted to constructing the theory that would follow from a good notion
of motives, one incorporating the Hodge, Tate, and standard conjectures. These
conjectures are believed to be beyond reach, and may not even be correct as
stated.\footnote{While we all hope that Grothendieck's motivic paradise
exists, it may be too optimistic to expect that it can be constructed using
actual algebraic classes.} I argue in this article that there exists a theory
of motives, accessible to proof, weaker than Grothendieck's, but with many of
the same consequences. \eabstract

\tableoc

\section*{Introduction}

\addcontentsline{toc}{section}{Introduction}

\hfill\begin{minipage}{4.0in}
\textit{To know a scheme is to know its functor of points,
and many properties of a scheme can be seen in its
functor of points.}\\
\hspace*{\fill}\cite{raynaud2014}, describing Grothendieck's view.\nocite{raynaud2014}
\end{minipage}\bigskip

In his Corvallis article (1979)\nocite{deligne1979}, Deligne introduced the
notion of a connected Shimura variety and its canonical model, and showed that
a Shimura variety has a canonical model if and only if its associated
connected Shimura variety has a canonical model. In particular, if two Shimura
varieties have isomorphic associated connected Shimura varieties, and one has
a canonical model, then both do. Starting from the Shimura varieties that are
moduli varieties for abelian varieties, he was able to prove in this way the
existence of canonical models for a large class of Shimura varieties, now said
to be of abelian type. His proof is a tour de force. It does not give a
description of the canonical model but only a characterization of it in terms
of reciprocity laws at the special points.

Later it was realized (\cite{milne1994b}) that the Shimura varieties of
abelian type with rational weight are exactly the moduli varieties of abelian
motives with additional structure. This allows us to prove the existence of
canonical models for these Shimura varieties by a simple descent argument and
it describes the canonical model as a moduli variety. The theory can be
extended to varieties with nonrational weight by applying a ``trick'' of Shimura.

The approach in the last paragraph applies only to Shimura varieties in
characteristic zero because the abelian motives are defined using Deligne's
theory of absolute Hodge classes, which works only in characteristic zero.

In this article, I outline a program to extend Deligne's theory of absolute
Hodge classes to characteristic $p$, thereby obtaining a good theory of
abelian motives in mixed characteristic. Once completed, this will make
possible similar simplifications in the theory of Shimura varieties in mixed characteristic.

\smallskip Throughout, $\mathbb{Q}^{\mathrm{al}}$ is an algebraic closure of
$\mathbb{Q}$ and $w$ a prime of $\mathbb{Q}^{\mathrm{al}}$ lying over $p$. The
residue field at $w$ is an algebraic closure $\mathbb{F}$ of $\mathbb{F}_{p}$.
For a variety $X$, we write $H_{\mathbb{A}{}}^{r}(X)$ for the restricted
product of the \'{e}tale cohomology groups $H_{\text{\'{e}t}}^{r}%
(X,\mathbb{Q}{}_{\ell})$ with the de Rham cohomology group (characteristic
$0$) or crystalline cohomology group (characteristic $p$).

\subsection{Statement of the conjectures}

In this subsection, we fix a collection $\mathscr{s}$ of abelian varieties
over $\mathbb{Q}^{\mathrm{al}}$ with good reduction at $w$, including the CM
abelian varieties. Every $A$ in $\mathcal{s}$ specializes to an abelian
variety $A_{0}$ over $\mathbb{F}{}$.

\begin{conjecturea}
\label{c2copy} Let $A\in\mathscr{s}$ and let $\gamma$ be an absolute Hodge
class on $A$. For all Lefschetz classes $\delta$ on $A_{0}$ of complementary
codimension, $\langle\gamma_{0}\cdot\delta\rangle\in\mathbb{Q}{}$.
\end{conjecturea}

In more detail, an absolute Hodge class on $A$ is an element $\gamma$ of
$H_{\mathbb{A}}^{2\ast}(A)(\ast)$, and its specialization $\gamma_{0}$ is an
element of $H_{\mathbb{A}{}}^{2\ast}(A_{0})(\ast)$. If $\delta_{1}%
,\ldots,\delta_{r}$, where $r=\dim(\gamma)$, are divisor classes on $A_{0}$,
then%
\[
\langle\gamma_{0}\cdot\delta_{1}\cdot\cdot\cdot\delta_{r}\rangle\in
H_{\mathbb{A}}^{2d}(A_{0})(d)\simeq\mathbb{A}{}_{f}^{p}\times\mathbb{Q}{}%
_{w}^{\mathrm{al}},\quad d=\dim A.
\]
The conjecture says that it lies in $\mathbb{Q}{}\subset\mathbb{A}{}_{f}%
^{p}\times\mathbb{Q}{}_{w}^{\mathrm{al}}$.

\begin{conjecturec}
There exists a unique family of a graded $\mathbb{Q}$-subalgebras
$\mathcal{R}^{\ast}(A)$ of $H_{\mathbb{A}}^{2\ast}(A)(\ast)$, indexed by the
abelian varieties over $\mathbb{F}$, whose elements we call rational Tate
classes, satisfying the following conditions:

\begin{description}
\item[(R1)] for any morphism $f$ of abelian varieties over $\mathbb{F}$,
$f^{\ast}$ and $f_{\ast}$ map rational Tate classes to rational Tate classes;

\item[(R2)] divisor classes are rational Tate classes;

\item[(R3)] absolute Hodge classes on abelian varieties in $\mathscr{s}$
specialize to rational Tate classes;

\item[(R4)] the inclusion $\mathcal{R}{}^{\ast}(A)\hookrightarrow
H_{\mathbb{A}}^{2\ast}(A)(\ast)$ induces an injection%
\[
\mathcal{R}{}^{\ast}(A)\otimes\mathbb{A}_{f}\rightarrow H_{\mathbb{A}{}%
}^{2\ast}(A)(\ast).
\]

\end{description}
\end{conjecturec}

\begin{conjectured}
There exists a $\mathbb{Q}$-linear tannakian category $\Mot(\mathbb{F}{})$ and
exact tenor functors $\LMot(\mathbb{F})\rightarrow\Mot(\mathbb{F}%
)\rightarrow\Mot(\mathbb{F};\mathbb{A}_{f})$ and $\Mot^{\mathscr{s}}%
(\mathbb{Q}^{\mathrm{al}})\rightarrow\Mot(\mathbb{F})$ such that
\[
\begin{tikzcd}
\LMot^{\mathscr{s}}(\mathbb{Q}^{\mathrm{al}})\arrow{d}{R}\arrow{r}
&\Mot^{\mathscr{s}}(\mathbb{Q}^{\mathrm{al}})\arrow{d}{R}\arrow{rd}
{\xi_f}\\
\LMot(\mathbb{F})\arrow{r}\arrow[bend right=20]{rr}[swap]{\xi_f}
&\Mot(\mathbb{F})\arrow{r}
&\Mot(\mathbb{F};\mathbb{A}_{f})\end{tikzcd}
\]
commutes and%
\[
\Mot(\mathbb{F})_{(\mathbb{A}_{f})}\longrightarrow\Mot(\mathbb{F};\mathbb{A}%
{}_{f})
\]
is faithful.
\end{conjectured}

Here $\LMot^{\mathscr{s}}(\mathbb{Q}^{\mathrm{al}})$ and $\LMot(\mathbb{F})$
are categories of abelian motives defined using Lefschetz classes,
$\Mot^{\mathscr{s}}(\mathbb{Q}^{\mathrm{al}})$ is the category of abelian
motives defined using absolute Hodge classes, and $\Mot(\mathbb{F}%
;\mathbb{A}_{f})$ is the category of abelian motives defined using Tate
classes. See later for precise definitions.

\subsection{Equivalence of the conjectures}

Conjecture C obviously implies Conjecture A because $\gamma_{0}$ and $\delta$
are both rational Tate clases, and in we prove a converse statement
(\ref{m5}): \bquote
Assume that Conjecture A holds for CM abelian varieties over $\mathbb{Q}%
{}^{\mathrm{al}}$; let $\mathcal{T}{}^{\ast}(A)$ denote the $\mathbb{A}{}_{f}%
$-algebra of Tate classes on an abelian variety $A$ over $\mathbb{F}{}$; then
there exists a unique family $\mathcal{R}{}^{\ast}(A)$ of $\mathbb{Q}{}%
$-structures on the $\mathcal{T}^{\ast}(A)$ satisfying (R1) and (R2) and such
that absolute Hodge classes on CM abelian varieties over $\mathbb{Q}%
^{\mathrm{al}}$ specialize to elements of $\mathcal{R}{}^{\ast}$.\equote

Given Conjecture C, we define $\Mot(\mathbb{F}{})$ to be the category of
motives based on the abelian varieties over $\mathbb{F}{}$ using the rational
Tate classes as correspondences. Conversely, we can recover the family
$(\mathcal{R}^{\ast}(A))$ from $\Mot(\mathbb{F})$ by setting
\[
\mathcal{R}^{r}(A)=\Hom_{\Mot(\mathbb{F})}(\1,h^{2r}(A)(r)).
\]

\subsection{Proof of the uniqueness of rational Tate classes}

For an abelian variety $A$ over $\mathbb{F}{}$ and a prime number $l$, we let
$\mathcal{T}_{l}^{\ast}(A)$ denote the $\mathbb{Q}{}_{l}$-algebra of Tate
classes on $A$,%
\[
\mathcal{T}_{\ell}^{r}(A)\overset{\df}{=}\bigcup_{A_{1}/k_{1}}H_{\text{\'{e}t}%
}^{2r}(A,\mathbb{Q}{}_{\ell}(r))^{\Gal(\mathbb{F}{}/k_{1})}\qquad
\mathcal{T}_{p}^{r}(A)\overset{\df}{=}\bigcup_{A_{1}/k_{1}}H_{p}^{2r}%
(A_{1})(r)^{F},
\]
where $A_{1}/k_{1}$ runs over the models of $A$ over finite subfields $k_{1}$
of $\mathbb{F}{}$. We let $\mathcal{T}{}^{\ast}(A)$ denote the restricted
product of the $\mathcal{T}{}_{l}^{\ast}(A)$ (an $\mathbb{A}{}_{f}$-algebra).

Tate conjectured that the algebraic cycles on $A$ modulo numerical equivalence
provide a $\mathbb{Q}{}$-structure on $\mathcal{T}^{\ast}(A)$. Here we prove
this for rational Tate classes (assuming they exist).

\begin{theorem}
\label{f1}Let $(\mathcal{R}^{\ast}(A))_{A\in\mathscr{s}{}}$ be a family of
rational Tate classes, as in Conjecture C. Then, for all $A\in\mathscr{s}$,
the map%
\[
\mathcal{R}{}^{\ast}(A)\otimes\mathbb{A}{}_{f}\rightarrow H_{\mathbb{A}%
}^{2\ast}(A)(\ast)
\]
has image $\mathcal{T}^{\ast}(A)$, i.e., $\mathcal{R}{}^{\ast}(A)$ is a
$\mathbb{Q}$-structure on $\mathcal{T}{}^{\ast}(A)$.
\end{theorem}

\begin{proof}
This can be proved by the argument in used in \cite{milne1999b} to show that
if the Hodge conjecture holds for CM abelian varieties then the Tate
conjecture holds for abelian varieties over $\mathbb{F}{}$. See
\cite{milne2009}, Theorem 2.2.
\end{proof}

\begin{theorem}
\label{f2}Let $(\mathcal{R}_{1}^{\ast}(A))_{A\in\mathscr{s}{}}$ and
$(\mathcal{\mathcal{R}{}}_{2}^{\ast}(A))_{A\in\mathscr{s}{}}$ be two families
as in the statement of Conjecture C. Then $\mathcal{R}_{1}^{\ast
}(A)=\mathcal{\mathcal{R}{}}_{2}^{\ast}(A)$ as $\mathbb{Q}{}$-subalgebras of
$H_{\mathbb{A}{}}^{2\ast}(A)(\ast)$.
\end{theorem}

\begin{proof}
If $\mathcal{R}_{1}^{\ast}(A)\subset\mathcal{R}_{2}^{\ast}(A)$, then they are
equal because they are both both $\mathbb{Q}{}$-structures $\mathcal{T}%
{}^{\ast}(A)$. The family $(\mathcal{R}{}_{1}^{\ast}(A)\cap
\mathcal{\mathcal{R}{}}_{2}^{\ast}(A))_{A\in\mathscr{s}{}}$ obviously
satisfies the conditions (R$_{1}$), (R$_{2}$), (R$_{3}$), and (R$_{4}$), and
so equals both $\mathcal{R}{}_{1}^{\ast}(A)$ and $\mathcal{\mathcal{R}{}}%
_{2}^{\ast}(A)$.
\end{proof}

\begin{aside}
For an abelian variety $A$ over $\mathbb{F}$,${}$ let $\mathcal{A}{}^{\ast
}(A)$ denote the $\mathbb{Q}{}$-subalgebra of $H_{\mathbb{A}{}}^{2\ast
}(A)(\ast)$ of algebraic classes. The family $(\mathcal{A}{}^{\ast}(A))$
satisfies (R1), (R2), and (R4). If it satisfies (R3), i.e., absolute Hodge
classes specialize to algebraic classes, then the Tate conjecture holds for
abelian varieties over $\mathbb{F}{}$.
\end{aside}

Let $\Mot(\mathbb{F}{};\mathbb{A}{}_{f})$ denote the category based on the
abelian varieties over $\mathbb{F}{}$, using the $\mathbb{A}{}_{f}$-algebras
of Tate classes as correspondences.

\begin{theorem}
\label{f3}Assume Conjecture D. The canonical functor
$\Mot(\mathbb{F)\rightarrow}\Mot(\mathbb{F}{};\mathbb{A}{}_{f})$ induces an
equivalence of categories%
\[
\Mot(\mathbb{F}{})_{(\mathbb{A}{}_{f})}\rightarrow\Mot\left(  \mathbb{F}%
{};\mathbb{A}{}_{f}\right)  .
\]
In other words, $\Mot(\mathbb{F)\rightarrow}\Mot(\mathbb{F}{};\mathbb{A}{}%
_{f})$ is a $\mathbb{Q}{}$-structure on the $\mathbb{A}{}_{f}$-linear
tannakian category $\Mot(\mathbb{F}{};\mathbb{A}{}_{f})$.
\end{theorem}

\begin{proof}
Restatement of Theorem \ref{f1}.
\end{proof}

\subsection{Consequences of the conjectures}

We list some consequences of the conjectures. For more details, see \S 6 and
later sections. We add question marks as a reminder that these statements
depend on conjectures.

\begin{plain}
[?]\label{f4a} Deligne (2006)\nocite{deligne2006} notes that the following
would be a \textquotedblleft particularly interesting corollary of the Hodge
conjecture\textquotedblright:

\bquote Let $A$ be an abelian variety over $\mathbb{F}$. Lift $A$ in two
different ways to characteristic $0$, to complex abelian varieties $A_{1}$ and
$A_{2}$ defined over $\mathbb{C}$. Pick Hodge classes $\gamma_{1}$ and
$\gamma_{2}$ on $A_{1}$ and $A_{2}$ of complementary dimension. Interpreting
$\gamma_{1}$ and $\gamma_{2}$ as $\ell$-adic cohomology classes, one can
define the intersection number $\kappa$ of the reductions of $\gamma_{1}$ and
$\gamma_{2}$ over $\mathbb{F}$. Is $\kappa$ a rational number independent of
$\ell$? \equote

\noindent The answer is yes, because the reductions of $\gamma_{1}$ and
$\gamma_{2}$ are both rational Tate classes on the abelian variety $A$.
\end{plain}

\begin{plain}
[?]\label{f4}Let $l\neq p$. For a certain countable subfield $Q$ of
$\mathbb{Q}_{l}$, Andr\'{e} defines a $Q$-linear tannakian category of motives
$\Mot(\mathbb{F};Q{})$ such that $\Mot(\mathbb{F}{};Q)_{(\mathbb{Q}{}_{l}%
)}\simeq\Mot(\mathbb{F}{};\mathbb{\mathbb{Q}{}}_{l})$. The category
$\Mot(\mathbb{F})$ is a $\mathbb{Q}$-structure on $\Mot(\mathbb{F};Q)$, that
is,
\[
\Mot(\mathbb{F})_{(Q)}\simeq\Mot(\mathbb{F}{};Q).
\]
Thus, rational Tate classes become motivated over the field $Q$.
\end{plain}

\begin{plain}
[?]\label{f5}The category $\Mot(\mathbb{F})$ and the functor $R\colon
\Mot^{\mathscr{s}}(\mathbb{Q}^{\mathrm{al}})\rightarrow\Mot\left(
\mathbb{F}{}\right)  $ have the properties predicted by the Hodge, Tate, and
standard conjectures. In particular, $\Mot(\mathbb{F})$ is a tannakian
category over $\mathbb{Q}{}$ with fundamental group the Weil-number protorus
$P$, and there is a unique polarization on $\Mot(\mathbb{F})$ for which $R$
becomes a functor of polarized tannakian categories.
\end{plain}

\begin{plain}
[?]\label{f6}Define $\Mot(\mathbb{F}_{q})$ to be the tannakian category over
$\mathbb{Q}{}$ whose objects are pairs $(M,\pi_{M})$ with $M$ and object of
$\Mot(\mathbb{F}{})$ and $\pi_{M}$ is a (Frobenius) endomorphism compatible
with the action of $P$ on $M$. Then $\Mot(\mathbb{F}_{q})$ is a tannakian
category over $\mathbb{Q}{}$ with the properties predicted by the Tate and
standard conjectures.
\end{plain}

\begin{plain}
[?]\label{f7}Let $M$ be an abelian motive over a number field $L\subset
\mathbb{Q}^{\mathrm{al}}$ (in the sense of absolute Hodge classes), and let
$G_{M}$ denote the Mumford--Tate group of the Hodge structure $\omega_{B}(M)$.
Let $\ell\neq p$ be a prime number. After possibly replacing $L$ with a finite
extension, we obtain a representation $\rho\colon\Gal(\mathbb{Q}%
{}^{\mathrm{al}}/L)\rightarrow G_{M}(\mathbb{Q}{}_{\ell})$. Assume that $M$
has good reduction at the primes of $L$ lying over $p$ (i.e., satisfies the
N\'eron condition). On applying $\rho$ to the Frobenius elements at such
primes, we get a conjugacy class in $G_{M}(\mathbb{Q}{}_{\ell})$, and hence an
element $\gamma_{\ell}\in(clG_{M})(\mathbb{Q}{}_{\ell})$. This element lies in
$(clG_{M})(\mathbb{Q})$ and is independent of $\ell$.
\end{plain}

\begin{plain}
\label{f8}We can extend the notion of an abelian motive to certain schemes,
for example:

\begin{itemize}
\item Let $S$ be a connected smooth scheme of finite type over a field $k$ of
characterstic zero, and let $\eta$ be the generic point of $S$. An abelian
motive over $S$ is an abelian motive over $\kappa(\eta)$ such that the action
of $\pi_{1}(\eta,\bar{\eta})$ on $\omega_{f}(M)$ factors through $\pi
_{1}(S,\bar{\eta}).$

\item Let $S$ be a local scheme with closed point $s$ such that $\kappa
(s)\subset\mathbb{F}{}$. An abelian motive over $S$ is a triple $(M,N,\varphi
)$, where $M$ is an abelian motive over $\kappa(s)$, $N$ is a $p$-shtuka (in
the sense of Scholze), and $\varphi$ is an isomorphism from the $p$-shtuka of
$M$ to $N_{s}$.
\end{itemize}
\end{plain}

\begin{plain}
\label{f9}As noted earlier, over $\mathbb{C}$ the Shimura varieties of abelian
type with rational weight are exactly the moduli schemes of abelian motives
with additional structure. This description extends to subfields of
$\mathbb{C}$, and Conjecture A allows us to extend it to mixed characteristic.
From a ``universal'' abelian motive over a Shimura variety, we can construct
the standard principal bundle, and hence a theory of automorphic vector
bundles. To understand Shimura varieties of abelian type without rational
weight, we apply Shimura's trick.
\end{plain}

\begin{plain}
\label{f11}We say that an abelian motive $M$ over $\mathbb{Q}^{\mathrm{al}}$
satisfies the N\'{e}ron condition if, for some model $M_{1}$ of $M$ over a
number field $K\subset\mathbb{Q}^{\mathrm{al}}$, the action of $\pi
_{1}(\Spec
K)$ on $\omega_{\ell}(M_{1})$ factors through $\pi_{1}(\Spec(\mathcal{O}%
_{K,w})$ (fundamental groups with respect to the base point $\Spec(\mathbb{Q}%
{}^{\mathrm{al}})$). The functor $R\colon\Mot^{\mathscr{s}}(\mathbb{Q}%
{}^{\mathrm{al}})\rightarrow\Mot(\mathbb{F}{})$ should extend to the category
$\Mot^{w}(\mathbb{Q}^{\mathrm{al}})$ of abelian motives over $\mathbb{Q}$
satisfying the N\'{e}ron condition.
\end{plain}

\begin{plain}
\label{f10} \textsc{Caveat.} The conjectures do not imply that algebraic
classes on abelian varieties over $\mathbb{F}$ are rational Tate classes.
Indeed, this would imply Grothendieck's standard conjecture of Hodge type for
abelian varieties. Nor do they imply the second of Deligne's \textquotedblleft
particularly interesting corollaries\textquotedblright\ (the intersection
number of the reduction of $\gamma_{1}$ with an algebraic cycle on $A$ is rational).
\end{plain}

\subsection{Positive results}

\begin{plain}
\label{f12}We have proved that the family $(\mathcal{R}{}^{\ast}(A))$ is
unique if it exists (see above).
\end{plain}

\begin{plain}
\label{f13}A commutative diagram of tannakian categories
\[
\begin{tikzcd}
\LCM(\mathbb{Q}^{\mathrm{al}})\arrow{d}{R}\arrow{r}
&\Mot^{\mathscr{s}}(\mathbb{Q}^{\mathrm{al}})\arrow{d}{R}\arrow{r}
&\Mot(\mathbb{Q}^{\mathrm{al}};\mathbb{Q}_{\ell})\arrow{d}{R}\\
\LMot(\mathbb{F})\arrow{r}
&\Mot(\mathbb{F})\arrow{r}
&\Mot(\mathbb{F};\mathbb{Q}_{\ell})
\end{tikzcd}
\]
would give rise to a commutative diagram of bands
\[
\begin{tikzcd}
T&G\arrow{l}\arrow{r}&G^\prime\\
L\arrow{u}&P\arrow[dashed]{u}[swap]{u}\arrow{l}\arrow{r}&P_{\mathbb{Q}_l}\arrow{u}
\end{tikzcd}
\]
We prove (unconditionally) that such a diagram exists, and that $u$ is
independent of $l$ (Theorem \ref{p1}). As a consequence, we find that an
abelian motive over a number field with good reduction (suitably defined)
defines a compatible system of $l$-adic representations.\footnote{The absence
of a proof of this last assertion has been cited by Deligne as one indication
that his theory of abelian motives is solely a characteristic zero theory.}
\end{plain}

\begin{plain}
\label{f14}A necessary condition that the morphism of bands $u\colon
P\rightarrow G$ arise from a morphism of tannakian categories $\Mot^{w}%
(\mathbb{Q}{}^{\mathrm{al}})\rightarrow\Mot(\mathbb{F}{})$ is that the
(conjectural) class of $\Mot(\mathbb{F}{})$ in $H^{2}(\mathbb{Q}{},P)$ map to
the trivial class in $H^{2}(\mathbb{Q}{},G)$. For this, see \S 10.
\end{plain}

\subsection{A variational approach to proving Conjecture A}

An obvious approach to proving Conjecture A is to mimic Deligne's proof (1982)
that all Hodge classes on abelian varieties are absolutely Hodge. Recall that
this has four main steps.

\begin{description}
\item[Step 0:] Behaviour of absolute Hodge classes in families:

\bquote Let $f\colon X\rightarrow S$ be a smooth proper map of complex
varieties with $S$ smooth and connected, and let $\gamma$ be a global section
of $R^{2n}f_{\ast}\mathbb{Q}(n)$. If $\gamma_{s}$ is a Hodge class
(resp.\ absolute Hodge class) for one $s\in S(\mathbb{C})$, then it is a Hodge
class (resp.\ absolute Hodge class) for every $s\in S(\mathbb{C})$.\equote

\item[Step 1:] Split Weil classes on abelian varieties are absolutely Hodge.
\bquote Let $(A_{1},\nu_{1},\lambda_{1})$ be a split Weil triple relative to a
CM-algebra $E$, and let $f\colon A\rightarrow S$ be the (universal) Weil
family containing $(A_{1},\nu_{1},\lambda_{1})$. Because the Weil family is
split, it contains the triple $(A_{2},\nu_{2},\lambda_{2})=(B\otimes
\mathcal{O}_{E},\ldots)$, where $B$ can be any abelian variety of dimension
$\dim(A_{1})/[E\colon\mathbb{Q}]$. The Weil classes on such a triple
$(A_{2},\nu_{2},\lambda_{2})$ are algebraic, in particular, absolutely Hodge.
Now apply Step 0.\equote

\item[Step 2:] Use Step 1 to show that Hodge classes on CM abelian varieties
are absolutely Hodge. \bquote Since pullbacks of absolute Hodge classes are
absolute Hodge, this follows from Theorem \ref{a6}: every Hodge class on a CM
abelian variety is a sum of pullbacks of split Weil classes on CM abelian
varieties. \equote

\item[Step 3:] Use Steps 0 and 2 to prove that all Hodge classes on abelian
varieties are absolutely Hodge. \bquote Let $A_{1}$ be an abelian variety over
$\mathbb{C}{}$ and $\gamma_{1}$ a Hodge class on $A_{1}$. Choose a
polarization $\lambda_{1}$ of $A_{1}$. There exists a family of polarized
abelian varieties $f\colon A\rightarrow S$ containing $(A_{1},\lambda_{1})$
and such that $\gamma_{1}$ extends to a global section of $R^{2n}f_{\ast
}\mathbb{A}{}(n)$ (\cite{deligne1982}, 6.1). Moreover, $S$ is a Shimura
variety, and so for some $s\in S(\mathbb{C})$ in the same connected component
as $(A_{1},\lambda_{1})$, $A_{s}$ is of CM-type. Now apply Steps 0 and
2.\equote

\end{description}

We try to mimic the above argument to prove Conjecture A

\begin{description}
\item[Step 0:] In \S \ref{CB}, we prove some variational theorems, for
example: \bquote Let $f\colon A\rightarrow S$ be an abelian scheme over a
connected normal scheme of finite type over an algebraically closed field $k$,
and let $\delta$ be a global section of $R^{2n}f_{\ast}\mathbb{A}(n)$ such
that $\delta$ is fixed by the Lefschetz group $L(A_{\bar{\eta}})$. If
$\delta_{s}$ is Lefschetz for one closed $s\in S$, then it is Lefschetz for
every closed $s\in S$ (Theorem \ref{c13}).\equote

\item[Step 1:] Conjecture A holds for split Weil classes. \bquote This is
unknown and is the crux.\equote\bquote Let $(A_{1},\nu_{1},\lambda_{1})$ be a
split Weil triple relative to a CM-algebra $E$, and let $f\colon A\rightarrow
S$ be the Weil family containing $(A_{1},\nu_{1},\lambda_{1})$ --- it is
defined over $\mathbb{Q}{}^{\mathrm{al}}$. If $(A_{1},\nu_{1},\lambda_{1}%
)_{0}$ lifts (up to isogeny) to a triple in the family whose Weil classes
satisfy Conjecture A, then the same is true for $(A_{1},\nu_{1},\lambda_{1}%
)$.\equote\bquote Let $\delta_{1}$ be a Lefschetz class on $(A_{1})_{0}$ of
dimension $d=\dim A_{1}/[E\colon\mathbb{Q}{}]$, and let $(A_{2},\nu
_{2},\lambda_{2})$ be a member of the family of the form $B\otimes
\mathcal{O}{}_{E}$. There exists a complete smooth curve $C$ in $S_{0}$ such
that $(A_{1},\nu_{1},\lambda_{1})_{0}$ and $(A_{2},\nu_{2},\lambda_{2})_{0}$
lie in the pullback $f_{C}\colon A_{C}\rightarrow C$ of $f_{0}\colon
A_{0}\rightarrow S_{0}$. If $C$ can be chosen so that $\delta_{1}$ extends to
a global section $\delta$ of $R^{2d^{\prime}}f_{C\ast}\mathbb{A}(d^{\prime})$,
$d+d^{\prime}=\dim A_{1}$, fixed by $L(A_{\bar{\eta}})$, then, for all Weil
classes $\gamma$ on $A/S$,%
\[
\langle(\gamma_{1})_{0}\cdot\delta_{1}\rangle=\langle(\gamma_{2})_{0}%
\cdot\delta_{2}\rangle
\]
lies in $\mathbb{Q}{}$ because $\gamma_{2}$ is algebraic and $\delta_{2}$ is
Lefschetz (Step 0). \equote

\item[Step 2:] Use Step 1 to show that Conjecture A holds for CM abelian
varieties. \bquote This again follows from Theorem \ref{a6} because pullbacks
of Hodge classes satisfying Conjecture A satisfy Conjecture A (this uses that
pushforwards of Lefschetz classes on abelian varieties are Lefschetz; see
\ref{x5}).\equote

\item[Step 3:] Use Step 2 to show that Conjecture A holds for all abelian
varieties over $\mathbb{Q}^{\mathrm{al}}$ with good reduction at $w$.

\bquote This is largely known. Let $A_{1}$ be an abelian variety over
$\mathbb{Q}^{\mathrm{al}}$ with good reduction at $w$, and let $\gamma_{1}$ be
a Hodge class on $A_{1}$. Choose a polarization $\lambda_{1}$ of $A_{1}$.
There exists a universal family of polarized abelian varieties $f\colon
A\rightarrow S$ over $\mathbb{Q}{}^{\mathrm{al}}$ containing $(A_{1}%
,\lambda_{1})$ in which $\gamma_{1}$ extends to a global section of
$R^{2n}f_{\ast}\mathbb{A}{}(n)$. Moreover, $S$ is a Shimura variety. Under
some hypotheses on $S$, the pair $(A_{1},\lambda_{1})_{0}$ lifts (up to
isogeny) to a CM pair $(A_{2},\lambda_{2})$ in the family in the family
(Kisin, Vasiu, \ldots). Now $\langle(\gamma_{1})_{0}\cdot\delta\rangle=$
$\langle(\gamma_{2})_{0}\cdot\delta\rangle$, which lies in $\mathbb{Q}$ by
Step 2.\footnote{Thus, to extend the reduction map from $\CM(\mathbb{Q}%
{}^{\mathrm{al}})$ to $\Mot^{w}(\mathbb{Q}{}^{\mathrm{al}})$, we need to know
that points on Shimura varieties over $\mathbb{F}{}$ lift to (up to isogeny)
to CM-points. There is a partial converse: given $R\colon\Mot^{w}(\mathbb{Q}%
{}^{\mathrm{al}})\rightarrow\Mot(\mathbb{F}{})$, it is possible to show that
every point in $\Sh(\mathbb{F}{})$ lifts to a CM point (at least when the
derived group is simply connected\ldots). See \cite{langlandsR1987}, 5.4;
\cite{milne1992}, 4.8.}\equote

\end{description}

For an inductive approach to proving Conjecture A, see Theorem \ref{x19}. For
an alternative variational approach, see \ref{c13}.

\subsection{Outline of the article}

After collecting various preliminaries in \S 1, we discuss in \S \S 2,3
possible proofs of Conjecture A, and a variant Conjecture B, for CM abelian
varieties. As outlined above, one approach is to mimic the argument in
\cite{deligne1982}.

In \S 4, we show that Conjecture A for CM abelian varieties implies Conjecture
C for CM abelian varieties. In particular, it implies that we have the notion
of a rational Tate class on an abelian variety over $\mathbb{F}$ and that
Hodge classes on CM abelian varieties specialize to rational Tate classes. In
\S 5 we investigate how to extend this last property to more general abelian
varieties. This comes down to a question of the existence of CM-lifts of
abelian varieties in families.

In \S 6, we describe the motivic paradise that results from Conjecture A. In
particular, we show how (given Conjecture A), it is possible to construct
tannakian categories of abelian motives over various schemes. This allows us
in the remaining sections to realize Shimura varieties of abelian type with
rational weight as moduli varieties, even in mixed characteristic (still in progress).

\clearpage

\subsection{Notation}

Throughout,
\[
l=2,3,\ldots,p,\ldots,\text{ or }\infty,
\]
and $\ell$ is a prime number $\neq p$,
\[
\ell\neq p,\infty.
\]
We sometimes let $\mathbb{Q}_{\infty}=\mathbb{R}$.

The symbol $\mathbb{Q}^{\mathrm{al}}$ denotes an algebraic closure of
$\mathbb{Q}$, $w$ denotes a prime of $\mathbb{Q}{}^{\mathrm{al}}$ lying over
$p$, and $\mathbb{F}{}$ is the residue field at $w$. We let $\iota$ or
$z\mapsto\bar{z}$ denote complex conjugation on $\mathbb{C}$ and its subfields.

Let $R$ be a $\mathbb{Q}{}$-algebra. By a $\mathbb{Q}{}$-structure on an
$R$-module $M$, we mean a $\mathbb{Q}{}$-subspace $V$ such that the map
$r\otimes v\mapsto rv\colon R\otimes_{\mathbb{Q}{}}V\rightarrow M$ is an isomorphism.

By a Hodge class on an abelian variety $A$ over a field of characteristic
zero, we mean an absolute Hodge class in the sense of \cite{deligne1982}. We
write $\mathcal{B}^{\ast}(A)$ for the $\mathbb{Q}$-algebra of Hodge classes on
$A$ and $\mathcal{D}^{\ast}(A)$ for the Lefschetz classes (elements of the
$\mathbb{Q}$-algebra generated by divisor classes).

We often regard abelian varieties as objects in the category with
homomorphisms $\Hom^{0}(A,B)$ (in which isogenies become isomorphisms).

For a complete smooth variety over a perfect field $k$ of characteristic $p$,
we let $H_{p}^{r}(X)$ denote the crystalline cohomology group $H^{r}(X/B)$,
where $B$ is the field of fractions of the ring of Witt vectors of $k$. It is
an $F$-isocrystal over $k$.

For a complete smooth variety $X$ over an algebraically closed field of
characteristic zero,%
\[
H_{\mathbb{A}}^{\ast}(X)\overset{\df}{=}\Big(\varprojlim_{m}H_{\text{\'{e}t}%
}^{\ast}(X,\mathbb{Z}{}/m\mathbb{Z}{})\otimes_{\mathbb{Z}{}}\mathbb{Q}%
{}\Big)\times H_{\mathrm{dR}}^{\ast}(X),
\]
and for a complete smooth variety $X_{0}$ over $\mathbb{F}{}$,%
\[
H_{\mathbb{A}{}}^{\ast}(X)\overset{\df}{=}\Big(\varprojlim_{p\nmid
m}H_{\text{\'{e}t}}^{\ast}(X,\mathbb{Z}{}/m\mathbb{Z}{})\otimes_{\mathbb{Z}{}%
}\mathbb{Q}{}\Big)\times H_{p}^{\ast}(X)\otimes_{B(\mathbb{F}{})}\mathbb{Q}%
{}_{w}^{\mathrm{al}},
\]
where $\mathbb{Q}{}_{w}^{\mathrm{al}}$ is the completion of $\mathbb{Q}%
{}^{\mathrm{al}}$ at $w$. If $X$ has good reduction at $w$ to $X_{0},$ then%
\begin{align*}
H_{\text{\'{e}t}}^{\ast}(X,\mathbb{Z}{}/m\mathbb{Z}{})  &  \simeq
H_{\text{\'{e}t}}^{\ast}(X_{0},\mathbb{Z}{}/m\mathbb{Z}{})\text{ for all
}m\text{ not divisible by }p\text{, }\\
H_{\mathrm{dR}}^{\ast}(X)\otimes_{\mathbb{Q}{}^{\mathrm{al}}}\mathbb{Q}{}%
_{w}^{\mathrm{al}}  &  \simeq H_{p}^{\ast}(X_{0})\otimes_{B(\mathbb{F}{}%
)}\mathbb{Q}{}_{w}^{\mathrm{al}},
\end{align*}
and so there is a canonical specialization map $H_{\mathbb{A}{}}%
^{i}(X)\rightarrow H_{\mathbb{A}}^{i}(X_{0})$.

For a connected normal scheme $S$, we let $\eta$ denote the generic point of
$S$ and $\bar{\eta}$ a geometric generic point (if $\eta=\Spec(K)$, then
$\bar{\eta}=\Spec(K^{\mathrm{sep}})$).

The quotient of an affine group scheme $G$ by its action on itself by inner
automorphisms is denoted by $clG$.

Commutative diagrams of categories and functors are required to commute ``on
the nose'' --- the maps of objects and arrows defined by two paths are
required to coincide.

\medskip$X=Y$ means that $X$ equals $Y$;

$X\simeq Y$ means that $X$ is isomorphic to $Y$ with a specific isomorphism
(often only implicitly described);

$X\approx Y$ means that $X$ is isomorphic to $Y$.

\clearpage

\section{Preliminaries}

\label{P}

We collect various preliminaries, partly to fix notation.

\subsection{Tannakian categories}

\begin{plain}
\label{t1}We assume that the reader is familiar with tannakian categories. Let
$\mathsf{T}$ be a tannakian category over a field $k$. Recall that the
fundamental group of $\mathsf{T}{}$ is the group scheme $\pi(\mathsf{T}{})$ in
$\mathsf{T}$ (better, $\Ind\mathsf{T}{}$) such that $\omega(\pi(\mathsf{T}%
))=\mathcal{A}ut^{\otimes}(\omega)$ for every fibre functor $\omega$. If
$\pi(\mathsf{T})$ is commutative, then it is an affine group scheme over $k$
in the usual sense.
\end{plain}

We briefly review the theory of quotients of tannakian categories
(\cite{milne2007d}).

\begin{plain}
\label{t2} Let $k$ be a field. An exact tensor functor $q\colon\mathsf{T}%
{}\rightarrow\mathsf{Q}$ of tannakian categories over $k$ is said to be a
quotient functor if every object of $\mathsf{Q}{}$ is a subquotient of an
object in the image of $q$. Then the full subcategory $\mathsf{T}^{q}$ of
$\mathsf{T}{}$ consisting of the objects that become trivial\footnote{That is,
isomorphic to a direct sum of copies of $\mathbb{1}$} in $\mathsf{Q}{}$ is a
tannakian subcategory of $\mathsf{T}{}$, and $X\rightsquigarrow\Hom(\1,qX)$ is
a $k$-valued fibre functor $\omega^{q}$ on $\mathsf{T}{}^{q}$. In particular
$\mathsf{T}{}^{q}$ is neutral. For $X,Y$ in $\mathsf{T}{}$,%
\begin{equation}
\Hom(qX,qY)\simeq\omega^{q}(\mathcal{H}{}om(X,Y)^{H})\text{,} \label{ep4}%
\end{equation}
where $H$ is the subgroup of $\pi(\mathsf{T})$ such that $\mathsf{T}%
^{q}=\mathsf{T}^{H}$.

Conversely, every $k$-valued fibre functor $\omega_{0}$ on a tannakian
subcategory $\mathsf{S}$ of $\mathsf{T}$ arises (as above) from a well-defined
quotient functor $\mathsf{T}\to\mathsf{T}{}/\omega_{0}$. For example, when
$\mathsf{T}{}$ is semisimple, we can take $\mathsf{T}{}/\omega_{0}$ to be the
pseudo-abelian hull of the category with one object $qX$ for each object $X$
of $\mathsf{T}{}$ and whose morphisms are given by (\ref{ep4}).
\end{plain}

\begin{plain}
\label{t3}Let $q\colon\mathsf{T}\rightarrow\mathsf{Q}$ be a quotient functor.
Fix a unit object $\1$ in $\mathsf{T}{}$, and let $\omega^{q}$ denote the
fibre functor $\Hom(\1,q(-))$ on $\mathsf{T}{}^{q}$. A fibre functor $\omega$
on $\mathsf{Q}{}$ defines a fibre functor $\omega\circ q$ on $\mathsf{T}{}$,
and the unique isomorphism of fibre functors $\Hom(\1,-)\rightarrow
\omega|\mathsf{Q}{}^{\pi(\mathsf{Q}{})}$ determines an isomorphism $\omega
^{q}\rightarrow(\omega\circ q)|\mathsf{T}{}^{q}$. Conversely, a pair
consisting of a fibre functor $\omega^{\prime}$ on $\mathsf{T}{}$ and an
isomorphism $\omega^{q}\rightarrow\omega^{\prime}|\mathsf{T}{}^{q}$ arises
from a unique fibre functor on $\mathsf{Q}$.
\end{plain}

\subsection{CM abelian varieties}

\begin{plain}
\label{t0} Let $A$ be an abelian variety over an algebraically closed field
$k$. The reduced degree\footnote{The reduced degree of a simple $Q$-algebra
$R$ with centre $C$ is $[C\colon Q][R\colon C]^{1/2}$.} of the $\mathbb{Q}%
$-algebra $\End^{0}(A)$ is $\leq2\dim A$; when equality holds the abelian
variety is said to be CM. An isotypic abelian variety is CM if and only if
$\End^{0}(A)$ contains a field of degree $2\dim A$ over $\mathbb{Q}{}$, and an
arbitrary abelian variety is CM if and only if each isotypic isogeny factor of
it is. Equivalent conditions:

\begin{enumerate}
\item \ the $\mathbb{Q}$-algebra $\End^{0}(A)$ contains an \'{e}tale
subalgebra of degree $2\dim A$ over $\mathbb{Q}{}$;

\item \ for a Weil cohomology $X\rightsquigarrow H^{\ast}(X)$ with coefficient
field $Q$, the centralizer of $\End^{0}(A)$ in $\End_{Q}(H^{1}(A))$ is
commutative (in which case it equals $C(A)\otimes_{\mathbb{Q}{}}Q$, where
$C(A)$ is the centre of $\End^{0}(A)$);

\item \ (characteristic zero) the Mumford-Tate group of $A$ is commutative
(hence a torus);

\item \ (characteristic $p\neq0$) $A$ is isogenous to an abelian variety
defined over $\mathbb{F}$ (theorems of Tate and Grothendieck).
\end{enumerate}
\end{plain}

\subsection{Abelian motives}

\subsubsection{Definition}

\label{AM}

To define a category of motives, we need a collection of smooth projective
varieties over a field $k$ and a theory of correspondences, i.e., a graded
$Q$-algebra $\mathcal{C}{}^{\ast}(X)$ for each variety $X$ in the collection
together with, for each morphism $\phi\colon X\rightarrow Y$, pull-back and
push-forward maps\footnote{In the formulas, we are assuming, for simplicity,
that the varieties have pure dimension.}%
\begin{align*}
\phi^{\ast}\colon\mathcal{C}{}^{\ast}(Y)  &  \rightarrow\mathcal{C}{}^{\ast
}(X)\\
\phi_{\ast}\colon\mathcal{C}{}^{\ast}(X)  &  \rightarrow\mathcal{C}^{\ast+\dim
Y-\dim X}(Y)
\end{align*}
satisfying certain obvious axioms. Here $Q$ is usually a field. The
corresponding category of motives $M(k)$ is defined as follows: an object is a
triple $(X,e,m)$, where $X$ is a variety in the collection, $e$ is an
idempotent in the $Q$-algebra $\mathcal{C}{}^{\dim X}(X\times X)$, and
$m\in\mathbb{Z}{}$; morphisms are defined by%
\[
\Hom((X,e,m),(Y,f,n))=f\cdot\mathcal{C}^{n-m+\dim X}(X\times Y)\cdot e\text{.}%
\]
The category $M(k)$ has a tensor product structure%
\[
(X,e,m)\otimes(Y,f,n)=(X\times Y,e\otimes f,m+n)
\]
with unit object%
\[
\1\overset{\df}{=}(\Spec k,\id,0).
\]
In all the cases in this article, the K\"unneth components of the diagonal lie
in $\mathcal{C}^{\ast}$, and we use them to modify the commutativity
constraint. In good cases, $M(k)$ is a tannakian category over the $Q$.

Let $k$ be an algebraically closed field of characteristic zero. We define the
category $\Mot(k)$ of \emph{abelian motives} over $k$ to be the category of
motives based on the abelian varieties over $k$ using the (absolute) Hodge
classes as correspondences.

\subsubsection{CM motives}

Let $k$ be an algebraically closed field of characteristic zero${}$. We define
the category $\CM(k)$ of \emph{CM motives} over $k$ to be the category of
motives based on the CM abelian varieties over $k$ using the (absolute) Hodge
classes as correspondences.

\begin{plain}
\label{b1}Let $K$ be a CM subfield of $\mathbb{\mathbb{Q}{}}^{\mathrm{al}}{}$,
finite and Galois over $\mathbb{Q}{}$. The Serre torus $S^{K}$ is the quotient
of $(\mathbb{G}_{m})_{K/\mathbb{Q}{}}$ such that
\[
X^{\ast}(S^{K})=\{f\colon\Gal(K/\mathbb{Q})\rightarrow\mathbb{Z}{}\mid
f(\sigma)+f(\iota\circ\sigma)=\text{constant}\}.
\]
The constant value ~$f(\sigma)+f(\iota\circ\sigma)$ is called the weight of
$f$. For $K^{\prime}\supset K$, there is a norm map $S^{K^{\prime}}\rightarrow
S^{K}$, and the Serre protorus is%
\[
S\overset{\df}{=}\varprojlim S^{K}.
\]

\end{plain}

\begin{plain}
\label{b1a} The category of CM motives over $k$ is a tannakian category over
$\mathbb{Q}{}$. The choice of an embedding $k\hookrightarrow\mathbb{C}{}$
defines a (Betti) fibre functor $\omega_{B}$ on $\CM(k)$, and%
\[
S\simeq\mathcal{A}{}ut^{\otimes}(\omega_{B}).
\]
In particular, we see that the functor (extension of the base field),%
\[
\CM(k)\rightarrow\CM(\mathbb{C}{})
\]
is an equivalence of tensor categories.
\end{plain}

\subsubsection{Abelian motives over the complex numbers}

Let $\omega_{B}$ denote the Betti fibre functor on $\Mot(\mathbb{C})$, and let
$G=\mathcal{A}ut^{\otimes}(\omega_{B})$. The functor
\[
\left(  \omega_{B}\right)  _{\left(  \mathbb{R}{}\right)  }{}\colon
\Mot(\mathbb{C})_{(\mathbb{R})}\rightarrow\Vc_{\mathbb{R}}%
\]
factors canonically through the category $\Hdg_{\mathbb{R}{}}$ of polarizable
real Hodge structures, and so defines a homomorphism $h\colon\mathbb{S}%
{}\rightarrow G_{\mathbb{R}{}}$, where $\mathbb{S}{}\overset{\df}{=}%
(\mathbb{G}_{m})_{\mathbb{C}{}/\mathbb{R}{}}$ is the Deligne torus. We wish to
describe the pair $(G,h)$. Let%
\[
\bar{h}\colon\mathbb{S}{}\rightarrow G_{\mathbb{R}}^{\mathrm{ad}}%
\]
denote the composite of $h$ with the quotient map $G_{\mathbb{R}{}}\rightarrow
G_{\mathbb{R}}^{\mathrm{ad}}$.

Consider the following conditions on a homomorphism $h\colon\mathbb{S}%
{}\rightarrow H$ of connected algebraic groups over $\mathbb{R}{}$:

\begin{description}
\item[\textmd{SV1:}] the Hodge structure on the Lie algebra of $H$ defined by
$\Ad\circ h$ is of type
\[
\{(1,-1),(0,0),(-1,1)\};
\]

\item[\textmd{SV2:}] $\inn(h(i))$ is a Cartan involution of $H^{\text{ad}}$.
\end{description}

\begin{theorem}
\label{a1}Let $(G,h)$ be the pair attached to $(\Mot(\mathbb{C}),\omega_{B})$
as above.

\begin{enumerate}
\item The quotient of $G$ by its derived group is the Serre protorus $(S,h)$.

\item Let $H$ be a semisimple algebraic group over $\mathbb{Q}{}$ and $\bar
{h}\colon\mathbb{S}{}/\mathbb{G}_{m}\rightarrow H_{\mathbb{R}{}}^{\mathrm{ad}%
}$ a homomorphism generating $H^{\mathrm{ad}}$ and satisfying the conditions
SV1,2. The pair $(H,\bar{h})$ is a quotient of $(G^{\mathrm{der}},\bar{h})$ if
and only if there exists an isogeny $H^{\prime}\rightarrow H$ with $H^{\prime
}$ a product of almost-simple groups $H_{i}^{\prime}$ over $\mathbb{Q}$ such
that either

\begin{enumerate}
\item $H_{i}^{\prime}$ is simply connected of type $A$, $B$, $C$, or
$D^{\mathbb{R}{}}$, or

\item $H_{i}^{\prime}$ is of type $D_{n}^{\mathbb{H}{}}$ ($n\geq5$) and equals
$\Res_{F/\mathbb{Q}{}}H_{0}$ for $H_{0}$ the double covering of an adjoint
group that is a form of $\SO(2n)$.
\end{enumerate}
\end{enumerate}
\end{theorem}

\noindent The homomorphism $G\rightarrow S$ is defined by the exact tensor
functor $\CM(\mathbb{C}{})\hookrightarrow\Mot(\mathbb{C}{})$.

\begin{proof}
See \cite{milne1994b}, 1.27.
\end{proof}

\subsubsection{Abelian motives over $\mathbb{Q}^{\mathrm{al}}$}

Take $\mathbb{Q}^{\mathrm{al}}$ to be the algebraic closure of $\mathbb{Q}$ in
$\mathbb{C}$, and let $\omega_{B}$ be the Betti fibre functor defined by the
inclusion $\mathbb{Q}^{\mathrm{al}}\hookrightarrow\mathbb{C}$.

\begin{theorem}
\label{a2}Let $(G,h)$ be the pair attached to $(\Mot(\mathbb{Q}^{\mathrm{al}%
}),\omega_{B})$. The quotient of $G$ by its derived group is $S$, and the
algebraic quotients of $(G^{\mathrm{der}},\bar{h})$ have the same description
as in Theorem \ref{a1}.
\end{theorem}

\begin{proof}
(a) Almost by definition, the motivic Galois group of $\CM(\mathbb{C}{})$ is
the Serre group $S$. The functor $A\rightsquigarrow A_{\mathbb{C}{}}$ defines
an equivalence of categories $\CM(\mathbb{Q}^{\mathrm{al}})\rightarrow
\CM(\mathbb{C}{})$, and so the motivic Galois group of $\CM(\mathbb{Q}%
^{\mathrm{al}})$ is also $S$.

(b) Let $H$ be a semisimple algebraic group over $\mathbb{Q}{}$ and $\bar{h}$
a homomorphism $\mathbb{S}{}/\mathbb{G}{}_{m}\rightarrow H_{\mathbb{R}{}%
}^{\mathrm{ad}}$ generating $H^{\mathrm{ad}}$ and satisfying the conditions
(SV1,2). Let $H\rightarrow\GL_{V}$ be a symplectic representation of
$(H,\bar{h})$. Recall (\cite{milne2013b}, 10.9), that this means that there
exists a commutative diagram%
\[
\begin{tikzcd}
H\arrow{d}\arrow{rd}\\
(H^{\mathrm{ad}},\bar{h})&(G,h)\arrow{l}\arrow{r}{\rho}&(G(\psi),D(\psi))
\end{tikzcd}
\]
in which $\psi$ is a nondegenerate alternating form on $V$, $G$ is a reductive
group (over $\mathbb{Q}$), and $h$ is a homomorphism $\mathbb{S}{}\rightarrow
G_{\mathbb{R}{}}$; the homomorphism $H\rightarrow G$ is required to have image
$G^{\mathrm{der}}$. In particular, $G^{\mathrm{ad}}\simeq H^{\mathrm{ad}}$.

Consider the map of Shimura varieties $\Sh(G,X)\rightarrow\Sh(H^{\mathrm{ad}%
},X^{\mathrm{ad}})$ corresponding to the map $(G,h)\rightarrow(H^{\mathrm{ad}%
},\bar{h})$. This is defined over $\mathbb{Q}{}^{\mathrm{al}}$, and induces a
finite surjective map
\[
\Sh_{K}(G,X)^{\circ}\rightarrow\Sh_{K^{\mathrm{ad}}}(H^{\mathrm{ad}%
},X^{\mathrm{ad}})^{\circ}.
\]
Since $H^{\mathrm{ad}}$ is a quotient of $G$, there is an abelian motive (+
structure) corresponding to a point of the second connected Shimura variety,
which lifts to the first Shimura variety.
\end{proof}

\begin{caution}
\label{a3}The theorem characterizes the algebraic quotients of
$G^{\mathrm{der}}$, where $G$ is the motivic Galois group of $\Mot(\mathbb{Q}%
{}^{\mathrm{al}})$. The algebraic quotients of $H^{\mathrm{der}}$, where $H$
is the motivic Galois group of $\Mot(\mathbb{C})$ have exactly the same
description, but the two groups are not isomorphic. For example,
$H^{\mathrm{der}}$ has an uncountably product of copies of $\SL_{2}$ as a
quotient, but $G^{\mathrm{der}}$ has only a countable product as a quotient.
\end{caution}

\subsubsection{Abelian motives with good reduction.}

We say that an abelian motive over $\mathbb{Q}^{\mathrm{al}}$ has \emph{good
reduction} at $w$ if some model of it over a subfield of $\mathbb{Q}%
{}^{\mathrm{al}}$ satisfies the N\'{e}ron condition. We let $\Mot^{w}%
(\mathbb{Q}{}^{\mathrm{al}})$ denote the full subcategory of $\Mot(\mathbb{Q}%
^{\mathrm{al}})$ whose objects have good reduction at $w$. Note that
$\CM(\mathbb{Q}{}^{\mathrm{al}})\subset\Mot^{w}(\mathbb{Q}{}^{\mathrm{al}})$.

\begin{theorem}
\label{b14}Let $(G,h)$ be the pair attached to $(\Mot^{w}(\mathbb{Q}%
^{\mathrm{al}}),\omega_{B})$. The quotient of $G$ by its derived group is $S$,
and the algebraic quotients of $(G^{\mathrm{der}},\bar{h})$ have the same
description as in Theorem \ref{a1}.
\end{theorem}

\begin{proof}
Omitted for the moment.
\end{proof}

In particular, $G^{\mathrm{der}}$ is not simply connected unless we exclude
abelian varieties of type $D^{\mathbb{H}{}}$ in the definition of
$\Mot^{w}(\mathbb{Q}^{\mathrm{al}})$.

\begin{aside}
\label{b14a} It will be important to extend everything in this article to
abelian varieties over $\mathbb{Q}^{\mathrm{al}}$ with (stable) bad reduction
at $w$.
\end{aside}

\subsection{The Weil-number protorus}

\begin{plain}
\label{b2a}Let $q\in p^{\mathbb{N}{}}$. An algebraic number is a Weil
$q$-number of weight $m$ if

\begin{enumerate}
\item for every homomorphism $\rho\colon\mathbb{Q}{}[\pi]\hookrightarrow
\mathbb{C}{}$, $\rho(\pi)\cdot\overline{\rho(\pi)}=q^{m},$ and

\item for some $n>0,$ $q^{n}\pi$ is an algebraic integer.
\end{enumerate}

\noindent The set of Weil $q$-numbers in $\mathbb{Q}{}^{\mathrm{al}}$ is
denoted by $W(q)$. It is a subgroup of the multiplicative group of
$\mathbb{Q}{}^{\mathrm{al}}$, stable under the action of $\Gamma
\overset{\df}{=}\Gal(\mathbb{Q}{}^{\mathrm{al}}/\mathbb{Q}{})$. We let $P(q)$
denote the affine group scheme of multiplicative type with%
\[
X^{\ast}(P(q))=W(q).
\]
There is a universal element $\pi_{\text{univ}}\in P(q)(\mathbb{Q})$ such that
$\chi(\pi_{\text{univ}})=\pi$ if $\chi$ is the character of $P(q)$
corresponding to $\pi$ --- it is the element corresponding to the inclusion
map under%
\[
P(q)(\mathbb{Q}{})\simeq\Hom(W(q),\mathbb{Q}^{\mathrm{al}\times}
)^{\Gal\mathbb{Q}^{\mathrm{al}}/\mathbb{Q}{})}.
\]

\end{plain}

\begin{plain}
\label{b2}If $n|n^{\prime}$, then $\pi\mapsto\pi^{n^{\prime}/n}$ is a
homomorphism $W(p^{n})\rightarrow W(p^{n^{\prime}})$. Define%
\[
W(p^{\infty})=\varinjlim W(p^{n}).
\]
There is an action of $\Gal(\mathbb{Q}{}^{\mathrm{al}}/\mathbb{Q}{})$ on
$W(p^{\infty})$, and the protorus $P$ over $\mathbb{Q}$ such that%
\[
X^{\ast}(P)=W(p^{\infty})
\]
is called the Weil-number protorus.
\end{plain}

\subsection{The Shimura--Taniyama homomorphism}

\begin{plain}
\label{b3}Again let $K$ be a CM subfield of $\mathbb{Q}{}^{\mathrm{al}}$,
finite and Galois over $\mathbb{Q}{}$. After possibly enlarging $K$, we may
suppose that $\iota$ is not in the decomposition group at $w_{K}$. Let
$\mathfrak{p}{}$ be the prime ideal of $\mathcal{O}{}_{K}$ corresponding to
$w_{K}$. For some $h$, $\mathfrak{p}{}^{h}$ will be principal, say,
$\mathfrak{p}{}^{h}=(a)$. Let $\alpha=a^{2n}$, where $n=(U(K)\colon U(K_{+}%
))$. Then, for $f\in X^{\ast}(S^{K})$, $f(\alpha)$ is independent of the
choice of $a$, and it is a Weil $p^{2nf(\mathfrak{p}{}/p)}$-number of weight
equal to the weight of $f$. The map $f\mapsto f(\alpha)\colon X^{\ast}%
(S^{K})\rightarrow W^{K}(p^{\infty})$ is a surjective homomorphism
(\cite{milne2001}, A.8), and so corresponds to an injective homomorphism
$\rho^{K}\colon P^{K}\rightarrow S^{K}$. On passing to the direct limit over
the $K\subset\mathbb{Q}{}^{\mathrm{al}}$, we obtain an injective homomorphism%
\[
\rho\colon P\rightarrow S,
\]
called the \emph{Shimura--Taniyama homomorphism}.
\end{plain}

\begin{plain}
\label{b12}The category $\CM(\mathbb{Q}^{\mathrm{al}})$ of CM motives over
$\mathbb{Q}{}^{\mathrm{al}}$ is the subcategory $\Mot(\mathbb{Q}^{\mathrm{al}%
})$ generated by the abelian varieties of CM-type. For any embedding
$\mathbb{Q}{}^{\mathrm{al}}\hookrightarrow\mathbb{C}{}$, the functor
$\CM(\mathbb{Q}{}^{\mathrm{al}})\rightarrow\CM(\mathbb{C}{})$ is an
equivalence of categories, and so the fundamental group of $\CM(\mathbb{Q}%
{}^{\mathrm{al}})$ is the Serre group. All CM abelian varieties over
$\mathbb{Q}{}^{\mathrm{al}}$ have good reduction at $w$ (because they satisfy
the N\'{e}ron condition).
\end{plain}

\begin{plain}
\label{b13}The Shimura-Taniyama homomorphism $P\rightarrow S$ (see \ref{b3})
has a geometric description. Let $A$ be a CM abelian variety over a subfield
$k$ of $\mathbb{\mathbb{Q}}^{\mathrm{al}}$. After possibly enlarging $k$, we
may suppose that $A$ has good reduction at $w$ and that some set of generators
for the Hodge classes on $A_{\mathbb{Q}{}^{\mathrm{al}}}$ and its powers is
defined over $k$. Under these assumptions, there is an endomorphism $F$ of $A$
reducing mod $p$ to the Frobenius endomorphism of $A_{0}$. Moreover, $F$ lies
in the Mumford--Tate group of $A$, and so defines a homomorphism
$P\rightarrow\MT(A)$. For varying $A$, these homomorphism define a
homomorphism $P\rightarrow\plim\MT(A)=S$. The theory of complex multiplication
(Shimura, Taniyama, Weil) shows that this agrees with the Shimura--Taniyama
homomorphism $\rho$ defined earlier.
\end{plain}

\begin{plain}
\label{b13a}The Galois groupoid attached to the category $\CM(\mathbb{Q}{})$
and its Betti fibre functor is called the Taniyama group. There is an explicit
description of it, due to Deligne and Langlands. See \cite{milne1990}, \S 6.
\end{plain}

\subsection{The Galois representations defined by abelian motives}

\begin{quotation}
A morphism $\Mot^{\mathscr{s}}(\mathbb{Q}{}^{\mathrm{al}})\rightarrow
\Mot(\mathbb{F}{})$ as in Conjecture D gives rise to a morphism $P\rightarrow
G$ of affine bands. Using a recent result of Kisin and Zhou (2025a) we show
(unconditionally) that such a morphism exists.
\end{quotation}

In this subsection, $\mathbb{Q}{}^{\mathrm{al}}$ is the algebraic closure of
$\mathbb{Q}{}$ in $\mathbb{C}{}$. We define the Mumford--Tate group of an
abelian motive $M$ (or abelian variety $A)$ over $\mathbb{Q}{}^{\mathrm{al}}$
to be the Mumford--Tate group of the rational Hodge structure $\omega_{B}(M)$
(or $\omega_{B}(A)$).

An abelian motive over $\mathbb{Q}{}^{\mathrm{al}}$ has \emph{visibly good
reduction} at $w$ if it is of the form $(A,e,m)$ with $A$ an abelian variety
with good reduction at $w$. Let $\Mot^{vis}(\mathbb{Q}{}^{\mathrm{al}})$
denote the category of abelian motives over $\mathbb{Q}{}^{\mathrm{al}}$ with
visibly good reduction at $w$, and let
\[
G=\mathcal{A}{}ut^{\otimes}(\omega_{B}\mid\Mot^{vis}(\mathbb{Q}{}%
^{\mathrm{al}})).
\]

Let $\ell\neq p$. Let $\Mot^{vis}(\mathbb{Q}{}^{\mathrm{al}};\mathbb{Q}%
{}_{\ell})$ and $\Mot(\mathbb{F}{};\mathbb{Q}{}_{\ell})$ denote the categories
of abelian motives defined using the $\ell$-adic Tate classes as
correspondences. They are $\mathbb{Q}{}_{\ell}$-linear tannakian categories
with canonical $\mathbb{Q}{}_{\ell}$-valued fibre functors $\omega_{\ell}$,
and we let%
\begin{align*}
G_{\ell}  &  =\mathcal{\mathcal{A}{}}ut^{\otimes}(\omega_{\ell}\mid
\Mot^{vis}(\mathbb{Q}^{\mathrm{al}};\mathbb{Q}_{\ell}))\\
P_{\ell}  &  =\mathcal{\mathcal{A}}ut^{\otimes}(\omega_{\ell}\mid
\Mot(\mathbb{F};\mathbb{Q}{}_{\ell})).
\end{align*}
From the reduction functor
\[
\Mot^{vis}(\mathbb{Q}^{\mathrm{al}};\mathbb{Q}{}_{\ell})\rightarrow
\Mot(\mathbb{F}{};\mathbb{Q}_{\ell}),
\]
we get a morphism $u_{\ell}\colon P_{\ell}\rightarrow G_{\ell}$ of affine
bands over $\mathbb{Q}_{\ell}$.

Suppose that there exists a commutative diagram of tannakian categories%
\[
\begin{tikzcd}
\Mot^{vis}(\mathbb{Q}^{\mathrm{al}})\arrow{r}\arrow{d}
&\Mot^{vis}(\mathbb{Q}^{\mathrm{al}};\mathbb{Q}_{\ell})\arrow{d}\\
\Mot(\mathbb{F})\arrow{r}&\Mot(\mathbb{F}{};\mathbb{Q}{}_{\ell}),
\end{tikzcd}
\]
such that%
\[
\Mot(\mathbb{F}{})_{(\mathbb{Q}{}_{\ell})}\rightarrow\Mot(\mathbb{F}%
;\mathbb{Q}_{\ell})
\]
is a tensor equivalence. From such a diagram, we get a commutative diagram of
bands%
\begin{equation}
\begin{tikzcd} G\arrow{r}&G_\ell\\ P\arrow{r}\arrow{u}[swap]{u} &P_\ell\arrow{u}[swap]{u_\ell} \end{tikzcd} \tag{*}%
\end{equation}
such that%
\[
P_{\mathbb{Q}_{\ell}}\rightarrow P_{\ell}%
\]
is an isomorphism. Here $P$ is the Weil-number protorus.

Using the main result of \cite{kisinZ2021}, we show that the diagram (*)
exists unconditionally.

\begin{theorem}
\label{p1}There exists a unique morphism of bands%
\[
u\colon P\rightarrow G
\]
over $\mathbb{Q}$ such that (*) commutes. It is independent of $\ell\neq p.$
\end{theorem}

\begin{proof}
To give a morphism $P\rightarrow G$ of bands amounts to giving a conjugacy
class of homomorphisms $P_{\mathbb{Q}{}^{\mathrm{al}}}\rightarrow
G_{\mathbb{Q}{}^{\mathrm{al}}}$ stable under the action of $\Gal(\mathbb{Q}%
^{\mathrm{al}}/\mathbb{Q})$.

Let $A$ be an abelian variety over $\mathbb{Q}{}^{\mathrm{al}}$ with good
reduction at $w$, and let $G_{A}$ be its Mumford--Tate group. From a model of
$A$ over a number field $L\subset\mathbb{Q}{}^{\mathrm{al}}$, we get a
Frobenius element $\gamma_{\ell}$ in $(clG_{A})(\mathbb{Q}{}_{\ell})$
(possibly after extending $L$). According to the main theorem of
\cite{kisinZ2021}, this arises from an element $\gamma\in(clG_{A}%
)(\mathbb{Q}{})$ that is independent of $\ell$. Let $\mathbb{F}_{q}%
\subset\mathbb{F}$ be the residue field at $w|L$. For each lift $\tilde
{\gamma}$ of $\gamma$ to $G_{A}(\mathbb{Q}{}^{\mathrm{al}})$, there is a
unique homomorphism $P(q)_{\mathbb{Q}{}^{\mathrm{al}}}\rightarrow
(G_{A})_{\mathbb{Q}{}^{\mathrm{al}}}$ of affine group schemes sending the
universal element $\pi_{\mathrm{univ}}\in P(q)(\mathbb{Q})$ to $\gamma$ (apply
Lemma \ref{p2}, if this is not obvious). As we vary the lift $\tilde{\gamma}$
of $\gamma$, we get a conjugacy class of homomorphisms $P(q)_{\mathbb{Q}%
^{\mathrm{al}}}\rightarrow(G_{A})_{\mathbb{Q}{}^{\mathrm{al}}}$ stable under
$\Gal(\mathbb{Q}{}^{\mathrm{al}}/\mathbb{Q}{})$, which we regard as a morphism
of affine bands $P(q)\rightarrow G_{A}$. Now pass to the inverse limit over
larger and larger abelian varieties,
\[
A_{1}\subset A_{2}\subset\cdots,
\]
and use that
\[
P=\varprojlim_{q} P(q)\quad\text{and}\quad G=\varprojlim_{A} G_{A},
\]
to obtain the required morphism of bands $P\rightarrow G$.
\end{proof}

\begin{lemma}
\label{p2}Let $k$ be a field and $R$ a $k$-algebra. Let $X$ and $Y$ be
algebraic $k$-schemes with $X$ reduced and $Y$ separated, and let
$\Sigma\subset X(k)$ be Zariski dense in $|X|$. A morphism of $R$-schemes
$\phi\colon X_{R}\rightarrow Y_{R}$ arises from a morphism of $k$-schemes if
and only if $\phi(\Sigma)\subset Y(k)$.
\end{lemma}

\begin{proof}
The necessity is obvious. Let $S=\Spec(k)$ and $T=\Spec(R)$. For the
sufficiency, we have to show that $\pr_{1}^{\ast}(\phi)=\pr_{2}^{\ast}(\phi)$,
where $\pr_{1}$ and $\pr_{2}$ are the projections $T\times_{S}T\rightarrow T$.
Because $X$ is reduced, $\Sigma$ is schematically dense in $X$, and so its
inverse image $\Sigma^{\prime}$ in $X\times_{S}(T\times_{S}T)$ is
schematically dense. As $Y$ is separated and $\pr_{1}^{\ast}(\phi)$ and
$\pr_{2}^{\ast}(\phi)$ agree on $\Sigma^{\prime}$, they must be equal.
\end{proof}

\begin{application}
\label{p7}Let $M$ be an abelian motive over a number field $L\subset
\mathbb{Q}^{\mathrm{al}}$ with good reduction at $w|L$, and let $G_{M}%
\overset{\df}{=}\mathcal{A}ut^{\otimes}(\omega_{B}\mid\langle M_{\mathbb{Q}%
^{\text{al}}}\rangle^{\otimes})$ be its Mumford--Tate group. Let $\ell\neq p$
be a prime number. After possibly replacing $L$ with a finite extension, we
obtain a representation $\rho\colon\Gal(\mathbb{Q}{}^{\mathrm{al}%
}/L)\rightarrow G_{M}(\mathbb{Q}{}_{\ell})$, and hence a Frobenius element
$\gamma_{\ell}\in(clG_{M})(\mathbb{Q}{}_{\ell})$. Assume that $M$ has visibly
good reduction at $w|L$, i.e., that $M=(A,e,m)$, where $A$ is an abelian
variety over $L$ with good reduction at $w|L$. In the proof of Theorem
\ref{p1}, we constructed a morphism of bands $P(q)\rightarrow G_{A}$. As
$G_{M}$ is a quotient of $G_{A}$, this gives us a morphism of bands
$P(q)\rightarrow G_{M}$, and hence a morphism of schemes $P(q)\rightarrow
clG_{M}$. The image of $\pi_{\text{univ}}\in P(q)(\mathbb{Q}{})$ in $\left(
clG_{M}\right)  (\mathbb{Q}{})$ maps to $\gamma_{\ell}\in(clG_{M}%
)(\mathbb{Q}{}_{\ell})$ and is independent of $\ell$. Thus, we obtain a
version of the main theorem of \cite{kisinZ2021} for abelian motives. (See
also \ref{p7c} and Theorem \ref{d9c} below.)
\end{application}

\subsubsection{Notes}

\begin{plain}
\label{p7a} Of course, it would be better to \textit{deduce} Theorem \ref{p1}
from the existence of a diagram of tannakian categories.
\end{plain}

\begin{plain}
\label{p7b} Note that we proved Theorem \ref{p1} by deducing a statement about
a \textit{cofinal} collection of algebraic quotients of $G$ from the theorem
of Kisin and Zhou, and then we proved \ref{p7} by applying Theorem \ref{p1} to
\textit{all} algebraic quotients of $G$.
\end{plain}

\begin{plain}
\label{p7c} By applying the results of \cite{kisinZ2025}, it is possible to
extend some of the above results to all abelian motives with good reduction at
$w$, and even include $p$. Specifically, from an abelian motive $M$ over a
number field $L\subset\mathbb{Q}{}^{\mathrm{al}}$, we get a Weil--Deligne
representation with values in $G_{M}(\mathbb{Q}{}_{l})$, which becomes a
Galois representation when $M$ satisfies the N\'{e}ron condition. By applying
the main theorem of ibid.\ we get a morphism of affine bands $P(q)\rightarrow
G_{M}$ with the correct local properties. Now consider the image of
$\pi_{\text{univ}}$ in $\left(  clG_{M}\right)  (\mathbb{Q}{})$.
\end{plain}

\begin{plain}
\label{p7d} For CM abelian varieties, this is all much easier, because the
Frobenius endomorphism on $A_{0}$ lifts to $A$. See \ref{b13}.
\end{plain}

\begin{plain}
\label{p7e} From the morphism of affine bands $P\rightarrow G$ we get an
action of $P$ on the objects of $\Mot^{w}(\mathbb{Q}^{\mathrm{al}})$, and
hence we can define $\Mot^{w}(\mathbb{Q}{}^{\mathrm{al}})^{P}$. Constructing a
commutative diagram%
\[
\begin{tikzcd}
\Mot^{w}(\mathbb{Q}^{\mathrm{al}})\arrow{r}\arrow{d}{R}
&\Mot^{w}(\mathbb{Q}^{\mathrm{al}};\mathbb{Q}_{\ell})\arrow{d}{R}\\
\Mot(\mathbb{F})\arrow{r}&\Mot(\mathbb{F}{};\mathbb{Q}{}_{\ell}),
\end{tikzcd}
\]
amounts to showing that the $\mathbb{Q}{}_{\ell}$-valued fibre functor on
$\Mot^{w}(\mathbb{Q}^{\mathrm{al}};\mathbb{Q}{}_{l})^{P}$ defined by the
right-hand arrow becomes $\mathbb{Q}{}$-valued when restricted to
$\Mot^{w}(\mathbb{Q}{}^{\mathrm{al}})^{P}$ (see \ref{t2}). In other words, for
each object $M$ of $\Mot^{w}(\mathbb{Q}^{\mathrm{al}})^{P}$ we need a
canonical $\mathbb{Q}$-structure
\[
\text{\textquotedblleft}\Hom_{\Mot(\mathbb{F})}%
(\1,R(M))\text{\textquotedblright}
\]
on $\Hom_{\Mot(\mathbb{F};\mathbb{Q}_{\ell})}(\1,R(M))$.
\end{plain}

\begin{plain}
\label{p7f}We have a well-defined morphism of affine bands $u\colon
bP\rightarrow bG$. Does this arise from a morphism of tannakian categories
$\Mot^{w}(\mathbb{Q}^{\mathrm{al}})\rightarrow\Mot(\mathbb{F}{})$ making the
diagram in (*) commute up to isomorphism? This is question in cohomology. For
what can be said about it, see the discussion preceding Theorem \ref{s0} below.

Similarly, one can ask whether there exists a $\mathbb{Q}{}$-valued functor
$\omega_{0}$ on $\Mot^{w}(\mathbb{Q}{}^{\mathrm{al}})^{P}$ such that
$\omega_{0}\otimes\mathbb{Q}_{l}\approx\omega_{l}$ for all $l$? Again, this is
a question in cohomology. For what can be said about it, see Theorem \ref{s14} below.
\end{plain}

\begin{plain}
\label{p7g}For related results, see
\cite{noot2009,noot2013,laskar2014,commelin2019,kisinZ2021,kisinZ2025}
\end{plain}

\subsection{Lefschetz motives}

\begin{plain}
\label{l1}For an adequate equivalence relation $\sim$ and a smooth projective
variety $X$ over a field $k$, we let $\mathcal{D}_{\sim}^{\ast}(X)$ denote the
$\mathbb{Q}{}$-subalgebra of $\mathrm{\CH}_{\sim}^{\ast}(X)$ generated by the
divisor classes. The elements of $\mathcal{D}_{\sim}^{\ast}(X)$ are called
\emph{Lefschetz classes} on $X$ (for the relation $\sim$).
\end{plain}

\begin{plain}
\label{l2}Let $k$ be an algebraically closed field. For abelian varieties, the
equivalence relation on Lefschetz classes defined by any Weil cohomology
theory coincides with that defined by numerical equivalence (\cite{milne1999a}%
, 5.3). In particular, it is independent of the Weil cohomology theory, and so
we set%
\[
\mathcal{D}{}^{\ast}(A)=\mathcal{D}{}_{\hom}^{\ast}(A)=\mathcal{D}%
{}_{\mathrm{num}}^{\ast}(A)
\]
when $A$ is an abelian variety. If $f\colon A\rightarrow B$ is a morphism of
abelian varieties (as varieties) then $f^{\ast}$ and $f_{\ast}$ preserve
Lefschetz classes (ibid. 5.5). This allows us to define $\LMot(k)$ to be the
category based on the abelian varieties over $k$ using the Lefschetz classes
as correspondences. It is a semisimple tannakian category over $\mathbb{Q}$
through which the Weil cohomologies factor.
\end{plain}

\begin{plain}
\label{l3}Let $H$ be a Weil cohomology theory with coefficient field $Q$, and
let $\omega_{H}$ be the fibre functor on $\LMot(k)$ it defines. For $A$ an
abelian variety, $\langle hA\rangle^{\otimes}$ denotes the tannakian
subcategory of $\LMot(k)$ generated by $hA$. Define the \emph{Lefschetz group}
of $A$ by%
\[
L(A)=\mathcal{A}{}ut^{\otimes}(\omega_{H}|\langle hA\rangle^{\otimes}).
\]
It is an algebraic group over $Q$.
\end{plain}

\begin{plain}
\label{l4}Let $C(A)$ be the centralizer of $\End^{0}(A)$ in $\End(H^{1}(A)$.
The Rosati involution $^{\dagger}$ of any ample divisor on $A$ preserves
$C(A)$, and its action on $C(A)$ is independent of the ample divisor. Now
$L(A)$ is the algebraic group over $Q$ such that%
\[
L(A)(Q)=\{\gamma\in C(A)\mid\gamma^{\dagger}\gamma\in Q^{\times}\}.
\]
For a CM abelian variety $A$, $C(A)=C_{0}(A)\otimes_{\mathbb{Q}{}}Q$, where
$C_{0}(A)$ is the centre of $\End^{0}(A)$. In this case, we let $L(A)$ denote
the algebraic group over $\mathbb{Q}{}$ such that%
\[
L(A)(\mathbb{Q}{})=\{\gamma\in C_{0}(A)\mid\gamma^{\dagger}\gamma\in
\mathbb{Q}{}^{\times}\}.
\]
Recall (\ref{t0}) that all abelian varieties $A$ over $\mathbb{F}$ are CM, and
that $C_{0}(A)=\mathbb{Q}\{\pi_{A}\}$.
\end{plain}

\begin{plain}
\label{l5}The inclusion $\mathcal{D}{}^{\ast}(A)\rightarrow H^{2\ast}%
(A)(\ast)$ induces an isomorphism%
\[
\mathcal{D}{}^{\ast}(A)\otimes_{\mathbb{Q}{}}Q\rightarrow H^{2\ast}%
(A)(\ast)^{L(A)},
\]
i.e., $\mathcal{D}{}^{\ast}(A)$ is a $\mathbb{Q}{}$-structure on $H^{2\ast
}(A)(\ast)^{L(A)}$ (ibid.\ 4.5, 5.3). An element of $H^{2\ast}(A)(\ast)$ is
said to be \emph{Lefschetz} if it is in the image of $\mathcal{D}{}^{\ast
}(A)\rightarrow H^{2\ast}(A)(\ast)$ and \emph{weakly Lefschetz} if it is in
the image of $\mathcal{D}{}^{\ast}(A)\otimes_{\mathbb{Q}{}}Q\rightarrow
H^{2\ast}(A)(\ast)$. Thus, an element of $H^{2\ast}(A)(\ast)$ is Lefschetz if
it is in the $\mathbb{Q}{}$-algebra generated by the divisor classes and
weakly Lefschetz if is fixed by $L(A)$.

Similarly, an element of $H_{\mathbb{A}{}}^{2\ast}(A)(\ast)$ is
\emph{Lefschetz} (resp.~\emph{weakly Lefschetz) }if it is in the image of
$\mathcal{D}{}^{\ast}(A)\rightarrow H_{\mathbb{A}{}}^{2\ast}(A)(\ast)$
(resp.~$\mathcal{D}{}^{\ast}(A)\otimes_{\mathbb{Q}{}}\mathbb{A}{}\rightarrow
H_{\mathbb{A}{}}^{2\ast}(A)(\ast)$).
\end{plain}

\begin{question}
\label{l6}Let $E$ be a CM field and $\mathbb{F}{}$ an algebraic closure of
$\mathbb{F}_{p}$. Does there exist a simple abelian variety $A$ over
$\mathbb{F}$ such that $\End^{0}(A)$ has centre $E$?
\end{question}

A positive answer would allow us to describe the fundamental group of
$\LMot(\mathbb{F}{})$.\footnote{In arXiv:2505.09589, the following weaker
statement is proved (Theorem 1.9): Let $E$ be a CM field. There exists a prime
number p and a simple abelian variety $A$ over $\mathbb{F}_{p}^{\mathrm{al}}$
such that $\End^{0}(A)$ has centre $E$.}

\begin{aside}
\label{l6a}Let $X$ be a smooth projective variety of dimension $d$ over an
algebraically closed field. A hyperplane section of $X\subset\mathbb{P}{}^{n}$
defines an isomorphism%
\[
L^{d-2r}\colon H^{2r}(X,\mathbb{Q}{}_{\ell}(r))\rightarrow H^{2d-2r}%
(X,\mathbb{Q}{}_{\ell}(d-r))
\]
for $r\leq d/2$ (strong Lefschetz theorem). In analogy with the standard
conjecture of Lefschetz type, one can ask whether%
\[
L^{d-2r}\colon\mathcal{D}{}^{r}(X)\rightarrow\mathcal{D}{}^{d-r}(X)
\]
is an isomorphism for $r<d/2$. Apart from abelian varieties, for which it is
proved in \cite{milne1999a}, this is known for only a few special varieties,
for example, toric varieties, full flag varieties, and products of such
varieties, and it fails already for blow-ups of such varieties and for partial
flag varieties.
\end{aside}

\begin{note}
The original source of the above theory is \cite{milne1999a,milne1999b}. For a
more recent exposition, see \cite{kahn2024}.
\end{note}

\subsection{Weil classes}

In this subsection, $H^{r}(A)=H^{r}(A,\mathbb{C})\simeq H^{r}(A,\mathbb{Q}%
{})\otimes\mathbb{C}{}$.

\begin{plain}
\label{a4}(\cite{deligne1982}, \S 5.) Let $A$ be a complex abelian variety and
$\nu$ a homomorphism from a CM-algebra $E$ into $\End^{0}(A)$. If $H^{1,0}(A)$
is a free $E\otimes_{\mathbb{Q}}\mathbb{C}$-module, then $\dim_{E}%
H^{1}(A,\mathbb{Q})$ is even, equal to $2d$ say, and the subspace
$W_{E}(A)\overset{\df}{=}\bigwedge\nolimits_{E}^{2d}H^{1}(A,\mathbb{Q}{})$ of
$H^{2d}(A,\mathbb{Q}{})$ consists of Hodge classes (\cite{deligne1982}, 4.4).
We say that $(A,\nu)$ is of \emph{Weil type. }When $[E\colon\mathbb{Q}]=2$,
these classes were studied by Weil (1977),\nocite{weil1977} and for this
reason are called \emph{Weil classes}. A polarization of $(A,\nu)$ is a
polarization $\lambda$ of $A$ whose Rosati involution stabilizes $\nu(E)$ and
acts on it as complex conjugation. We then call $(A,\nu,\lambda)$ a \emph{Weil
triple}. The Riemann form of such a polarization can be written%
\[
(x,y)\mapsto\Tr_{E/\mathbb{Q}{}}(f\phi(x,y))
\]
for some totally imaginary element $f$ of $E$ and $E$-hermitian form $\phi$ on
$H_{1}(A,\mathbb{Q}{})$. If $\lambda$ can be chosen so that $\phi$ is split
(i.e., admits a totally isotropic subspace of dimension $d$), then
$(A,\nu,\lambda)$ is said to be of \emph{split Weil type}.
\end{plain}

\begin{example}
\label{a4a}Let $E$ be a CM-algebra over $\mathbb{Q}{}$, and let $A=B\otimes
\mathcal{O}{}_{E}$, where $B$ is an abelian variety of dimension $d$. With the
obvious action $\nu$ of $E$, this is of Weil type. The Weil classes on it are
algebraic (\cite{deligne1982}, 4.5).
\end{example}

\begin{plain}
\label{a5}(\cite{deligne1982}, \S 5.) Let $E$ be a CM-field, let $\phi
_{1},\ldots,\phi_{2p}$ be CM-types on $E$, and let $A_{i}$ be a complex
abelian variety of CM-type $(E,\phi_{i})$. If $\sum_{i}\phi_{i}(s)=p$ for all
$s\in T\overset{\df}{=}\Hom(E,\mathbb{\mathbb{Q}}^{\mathrm{al}}{})$, then
$A\overset{\df}{=}\prod\nolimits_{i}A_{i}$, equipped with the diagonal action
of $E$, is of split Weil type. Let $I=\{1,\ldots,2p\}$. Then $A$ has complex
multiplication by the CM-algebra $E^{I}$, and $\Hom(E^{I},\mathbb{C})=I\times
T$. When tensored with $\mathbb{\mathbb{C}{}}$, the inclusion $W_{E}%
(A)\hookrightarrow H^{2p}(A,\mathbb{Q}{})$ becomes,%
\[
\begin{tikzcd}
W_{E}(A)\otimes\mathbb{C}\arrow[hook]{r}\arrow[equals]{d}
&H^{2p}(A)\arrow[equals]{d}\\
\dstyle\bigoplus_{t\in T}H^{2p}(A)_{I\times \{t\}}\arrow[hook]{r}
&\dstyle\bigoplus_{\substack{J\subset I\times T\\|J|=2p\\}}%
H^{2p}(A)_{J}.
\end{tikzcd}
\]

\end{plain}

\begin{example}
\label{a5a} Let $(A,\nu,\lambda)$ be a Weil triple relative to the CM field
$Q$, and let
\begin{align*}
\SU(\phi)  &  =\{a\in\SL_{Q}(V(A))\mid\phi(ax,ay)=\phi(x,y)\}\\
U(\phi)  &  =\{a\in\GL_{Q}(V(A))\mid\phi(ax,ay)=\phi(x,y)\}\\
\GU(\phi)  &  =\{a\in\SL_{Q}(V(A))\mid\phi(ax,ay)=\mu(a)\phi(x,y),~\mu(a)\in
Q^{\times}\}
\end{align*}
(unitary similitudes). When $(A,\nu,\lambda)$ is general, there is an exact
commutative diagram%
\[
\begin{tikzcd}
&1\arrow{d}&1\arrow{d}&1\arrow{d}\\
1\arrow{r}&\SU(\phi)\arrow{r}\arrow{d}&\GU(\phi)\arrow{r}\arrow{d}
&Q^{\times}\arrow{d}\arrow{r}&1\\
1\arrow{r}&\MT(A)\arrow{r}\arrow{d}&L(A)\arrow{r}\arrow{d}
&Q^{\times}\arrow{r}&1\\
&\mathbb{G}_{m}\arrow{r}\arrow{d}&\mathbb{G}_{m}\arrow{d}\\
&1&1%
\end{tikzcd}
\]
It follows that for a general $A$,

\begin{enumerate}
\item the Weil classes are Hodge classes but not Lefschetz classes;

\item if the Weil classes are algebraic, then the Hodge conjecture holds for
$A$ and its powers.
\end{enumerate}
\end{example}

\begin{theorem}
[\cite{andre1992}]\label{a6} Let $A$ be a complex abelian variety of CM-type.
There exist CM abelian varieties $A_{\Delta}$ of split Weil type and
homomorphisms $f_{\Delta}\colon A\rightarrow A_{\Delta}$ such that every Hodge
class $t$ on $A$ can be written as a sum $t=\sum f_{\Delta}^{\ast}(t_{\Delta
})$ with $t_{\Delta}$ a split Weil class on $A_{\Delta}$.
\end{theorem}

\begin{proof}
Let $p\in\mathbb{N}{}$. After replacing $A$ with an isogenous variety, we may
suppose that it is a product of simple abelian varieties $A_{i}$ (not
necessarily distinct). Let $E=\prod_{i}\End^{0}(A_{i})$. Then $E$ is a
CM-algebra, and $A$ is of CM-type $(E,\phi)$ for some CM-type $\phi$ on $E$.
Let $K$ be a CM subfield of $\mathbb{\mathbb{Q}{}}^{\mathrm{al}}{}$, finite
and Galois over $\mathbb{Q}$, splitting the centre of $\End^{0}(A)$, and let
$S=\Hom(E,K)$. Then%
\[
e\otimes c\leftrightarrow(se\cdot c)_{s}\colon E\otimes K\simeq\prod_{s\in
S}K_{s},
\]
where $K_{s}$ denotes the $E$-algebra $(K,s)$. Let $T=\Gal(K/\mathbb{Q})$.

Let $H^{i}(A)=H^{i}(A_{\mathbb{C}{}},\mathbb{Q}{})\otimes_{\mathbb{Q}{}}K$.
Then%
\[
H^{2p}(A)=\bigoplus\nolimits_{\Delta}H^{2p}(A)_{\Delta},
\]
where $\Delta$ runs over the subsets $\Delta$ of $S$ of order $2p$, and
\[
\mathcal{B}^{p}\otimes_{\mathbb{Q}{}}K=\bigoplus\nolimits_{\Delta}%
H^{2p}(A)_{\Delta},
\]
where $\Delta$ runs over the subsets satisfying
\begin{equation}
\left\vert t\Delta\cap\Phi\right\vert =p\quad\text{all }t\in\Gal(K/\mathbb{Q}%
{})\text{.} \label{eq2}%
\end{equation}
Let $K_{\Delta}=\prod\nolimits_{s\in\Delta}K_{s}$, and let $A_{\Delta
}=A\otimes_{\mathbb{Q}{}}K_{\Delta}$. Then $A_{\Delta}$ is of split Weil type
relative to the diagonal action of $K$. Let $f_{\Delta}\colon A\rightarrow
A_{\Delta}$ be the homomorphism such that
\[
H_{1}(f_{\Delta})\colon H_{1}(A,\mathbb{Q}{})\rightarrow H_{1}(A_{\Delta
},\mathbb{Q}{})\simeq H_{1}(A,\mathbb{Q})\otimes_{E}K_{\Delta}%
\]
is the obvious inclusion. There is a diagram%
\[
\begin{tikzcd}
W_{K}(A_{\Delta})\otimes\mathbb{Q}^{\mathrm{al}}\arrow[hook]{r}\arrow[equals]{d}
&H^{2p}(A_\Delta)\arrow[equals]{d}\\
\dstyle\bigoplus_{t\in T}H^{2p}(A_\Delta)_{\Delta\times \{t\}}\arrow[hook]{r}
&\dstyle\bigoplus_{\substack{J\subset \Delta\times T\\|J|=2p\\}}%
H^{2p}(A_\Delta)_{J},
\end{tikzcd}
\]
where $H^{2p}(A_{\Delta})_{J}$ is the $1$-dimensional subspace of
$H^{2p}(A_{\Delta})$ on which $a\in K_{\Delta}$ acts as $\prod\nolimits_{j\in
J}j(a)$. Note that $a\in E$ acts on $H^{2p}(A_{\Delta})_{\Delta\times\{t\}}$
as multiplication by $\prod\nolimits_{s\in\Delta}(t\circ s)(a)$, and so
$f_{\Delta}^{\ast}\colon H^{2p}(A_{\Delta})\rightarrow H^{2p}(A)$ maps
$H^{2p}(A_{\Delta})_{\Delta\times\{t\}}$ onto $H^{2p}(A)_{t\circ\Delta}$.
Therefore,%
\[
H^{2p}(A)_{\Delta}\subset f_{\Delta}^{\ast}(W_{K}(A_{\Delta}))\otimes
_{\mathbb{Q}{}}K\subset\mathcal{B}{}^{p}(A)\otimes_{\mathbb{Q}{}}K.
\]
As the subspaces $H^{2p}(A)_{\Delta}$ for $\Delta$ satisfying (\ref{eq2}) span
$\mathcal{B}{}^{p}(A)\otimes\mathbb{C}$, this shows that%
\[
\sum_{\Delta\text{ satisfies (\ref{eq2}) }}f_{\Delta}^{\ast}(W_{K}(A_{\Delta
}))\otimes_{\mathbb{Q}{}}K=\mathcal{B}{}^{p}(A)\otimes_{\mathbb{Q}{}%
}K\mathbb{{}}\text{.}%
\]
As $f_{\Delta}^{\ast}(W_{K}(A_{\Delta}))$ and $\mathcal{B}{}^{p}(A)$ are both
$\mathbb{Q}$-subspaces of $H^{2p}(A,\mathbb{Q}{})$, it follows (from
\ref{a6a}) that%
\[
\sum_{\Delta\text{ satisfies (\ref{eq2}) }}f_{\Delta}^{\ast}(W_{K}(A_{\Delta
}))=\mathcal{B}{}^{p}(A)\text{.}%
\]
See \cite{milne2020a} for more details.
\end{proof}

\begin{lemma}
\label{a6a}Let $W$ and $W^{\prime}$ be subspaces of a $k$-vector space $V$,
and let $R$ be a $k$-algebra. If $W\otimes_{k}R\subset W^{\prime}\otimes_{k}R$
(inside $V\otimes R$), then $W\subset W^{\prime}.$
\end{lemma}

\begin{proof}
Because $k$ is a field, the lattice of subspaces of $V$ is preserved when we
tensor with $R$.
\end{proof}

\subsection{Weil families}

\begin{plain}
\label{w1}Let $(A_{1},\nu_{1},\lambda_{1})$ be a split Weil triple relative to
a CM field $E$, and let $f\colon A\rightarrow S$ (a polarized abelian scheme
over $S$ with an action of $E)$ be the universal Weil family containing
$(A_{1},\nu_{1},\lambda_{1})$ (\cite{deligne1982}, 4.8). Here $S$ is a smooth
variety over $\mathbb{C}{}$ and the fibres of $f$ are split Weil triples
(including $(A_{1},\nu_{1},\lambda_{1})$). There is a local subsystem
$W_{E}(A/S)$ of $R^{2d}f_{\ast}\mathbb{Q}{}(d)$ such that $W_{E}%
(A/S)_{s}=W_{E}(A_{s})$ for all $s\in S(\mathbb{C}{})$. Here $d=\dim
(A)/[E\colon\mathbb{Q}{}]$.
\end{plain}

\begin{plain}
\label{w1a}Let $B$ be an abelian variety over $\mathbb{C}{}$ of dimension $d$.
Then $B\otimes\mathcal{O}{}_{E}$ with its action of $E$ and a suitable
polarization is a Weil triple $(A_{2},\nu_{2},\lambda_{2})$. Because
$(A_{1},\nu_{1},\lambda_{1})$ is split, $(A_{2},\nu_{2},\lambda_{2})$ lies in
the universal family containing $(A_{1},\nu_{1},\lambda_{1})$.
\end{plain}

\begin{plain}
\label{w2}The variety $S$ has a unique model over $\mathbb{Q}{}^{\mathrm{al}}$
with the property that every CM-point $s\in S(\mathbb{C})$ lies in
$S(\mathbb{Q}{}^{\mathrm{al}})$. This follows from the general theory of
Shimura varieties; or from the general theory of locally symmetric varieties
(Faltings; \cite{peters2017}); or (best) from descent theory
(\cite{milne1999c}, 2.3) using that $S$ is a moduli variety over $\mathbb{C}%
{}$ and that the moduli problem is defined over $\mathbb{Q}{}^{\mathrm{al}}$.
The morphism $f$ is also defined over $\mathbb{Q}{}^{\mathrm{al}}$, and we
will now simply write $f\colon A\rightarrow S$ for the family over
$\mathbb{Q}{}^{\mathrm{al}}$. There is a $\mathbb{Q}{}$-local subsystem
$W_{E}(A/S)$ of $R^{2d}f_{\ast}\mathbb{\mathbb{A}{}}(d)$ such that
$W_{E}(A/S)_{s}=W_{E}(A_{s})$ for all $s\in S(\mathbb{Q}^{\mathrm{al}})$.
\end{plain}

\begin{plain}
\label{w3}Now assume that $(A_{1},\nu_{1},\lambda_{1})$ and $B$ are defined
over $\mathbb{Q}{}^{\mathrm{al}}$. Let $s_{1}$ and $s_{2}$ be the points of
$S(\mathbb{Q}{}^{\mathrm{al}})$ corresponding to $(A_{1},\nu_{1},\lambda_{1})$
and $(A_{2},\nu_{2},\lambda_{2})$. We assume that $A_{1}$ and $B$ both have
good reduction at $w$. We also assume that $s_{1}$ and $s_{2}$ lie in the same
connected component of $S$, and we replace $S$ with that connected component.
\end{plain}

\begin{plain}
\label{w4}For technical reasons, we now assume that $E$ is contains an
imaginary quadratic field in which the prime $p$ splits.
\end{plain}

\begin{plain}
\label{w5}The family $f\colon A\rightarrow S$ (without the action of $E$)
defines a morphism from $S$ into a moduli variety $M$ over $\mathbb{Q}%
{}^{\mathrm{al}}$ for polarized abelian varieties with certain level
structures. Let $\mathcal{M}{}$ be the corresponding moduli scheme over
$\mathcal{O}{}_{w}$ and $\mathcal{M}{}^{\ast}$ its minimal compactification
(\cite{chaiF1990}). Let $\mathcal{S}{}^{\ast}$ be the closure of $S$ in
$\mathcal{M}{}^{\ast}$. The complement of $\mathcal{S}_{\mathbb{F}{}}^{\ast
}\cap\mathcal{M}{}_{\mathbb{F}{}}$ in $\mathcal{S}{}_{\mathbb{F}{}}^{\ast}$
has codimension at least two (\cite{andre2006a}, 2.4.2).
\end{plain}

\begin{plain}
\label{w6}Recall that $s_{1}$ and $s_{2}$ are points in $S(\mathbb{Q}%
{}^{\mathrm{al}})$ such that $A_{s_{1}}=A_{1}$ and $A_{s_{2}}=B\otimes
\mathcal{O}{}_{E}$. As $A_{1}$ and $B$ have good reduction, the points extend
to points $\mathscr{s}_{1}$ and $\mathscr{s}_{2}$ of $\mathcal{S}{}^{\ast}%
\cap\mathcal{\mathcal{M}}$. Let $\mathcal{\bar{S}}$ denote the blow-up of
$\mathcal{S}{}^{\ast}$ centred at the closed subscheme defined by the image of
$\mathscr{s}_{1}$ and $\mathscr{s}_{2}$, and let $\mathcal{S}{}$ be the open
subscheme obtained by removing the strict transform of the boundary
$\mathcal{S}{}^{\ast}\smallsetminus(\mathcal{S}{}^{\ast}\cap\mathcal{M}{})$.
It follows from \ref{w5} that $\mathcal{S}_{\mathbb{F}{}}$ is connected, and
that any sufficiently general linear section of relative dimension $\dim(S)-1$
in a projective embedding $\mathcal{\bar{S}}\hookrightarrow\mathbb{P}%
{}_{\mathcal{O}{}_{w}}^{N}{}$ is a projective flat $\mathcal{O}{}_{w}$-curve
$\mathcal{C}{}$ contained in $\mathcal{S}{}$ with smooth geometrically
connected generic fibre (Andr\'{e} 2.5.1). Consider $\left(  \mathcal{A}%
|\mathcal{C}{}\right)  _{\mathbb{F}{}}\rightarrow\mathcal{C}_{\mathbb{F}{}}$.
\end{plain}

\begin{plain}
\label{w7}After replacing $\mathcal{C}{}_{\mathbb{F}{}}$ by its normalization
$C$ and pulling back $\left(  \mathcal{A}|\mathcal{C}{}\right)  _{\mathbb{F}%
{}}$, we are in the following situation:

\begin{itemize}
\item $C$ is a complete smooth curve over $\mathbb{F}{};$

\item $\bar{f}\colon\bar{A}\rightarrow C$ is a polarized abelian scheme over
$C$ with a compatible action of $E$;

\item there is a $\mathbb{Q}{}$-local system $W_{E}(\bar{A}/C)$ over $C$ that
is a $\mathbb{Q}{}$-structure on $\bigwedge_{E\otimes\mathbb{A}{}}^{2d}%
R^{1}\bar{f}_{\ast}(\mathbb{A}{})(d);$

\item there are points $1,2\in C(\mathbb{F}{})$ such that $(\bar{A}%
,\bar{\lambda},\bar{\nu})_{1}\approx(A_{1},\lambda_{1},\nu_{1})_{0}$ and
$(\bar{A},\bar{\lambda},\bar{\nu})_{2}\simeq(A_{2},\lambda_{2},\nu_{2})_{0}$.
\end{itemize}
\end{plain}

\subsection{The realization categories}

We construct categories $\mathsf{R}_{l}(\mathbb{F}{})$ for $l=2,3,5,\ldots
,p,\ldots,\infty$. Each is a tannkian category over $\mathbb{Q}_{l}$ with
commutative fundamental group $P_{l}$. There is a canonical homomorphism
$P_{l}\rightarrow P_{\mathbb{Q}_{l}}$ which we can use to modify the category
so that its fundamental group is $P_{\mathbb{Q}{}_{l}}$ --- we denote the new
category by $\mathsf{V}_{l}(\mathbb{F})$.\footnote{An object $\mathsf{V}%
_{l}(\mathbb{F})$ is object of $\mathsf{R}_{l}(\mathbb{F})$ together with an
action of $P_{\mathbb{Q}_{l}}$ compatible with the action of $P_{l}$.} Under
the Tate and standard conjectures, there are (realization) functors $\eta
_{l}\colon\Mot_{\mathrm{num}}(\mathbb{F}{})\rightarrow\mathsf{R}%
_{l}(\mathbb{F}{})$ on Grothendieck's category of numerical motives inducing
equivalences of $\mathbb{Q}{}_{l}$-linear tensor categories%
\[
\Mot_{\mathrm{num}}(\mathbb{F}{})_{(\mathbb{Q}{}_{l})}%
\overset{\lsim}{\longrightarrow}\mathsf{V}_{l}(\mathbb{F}{})\text{.}%
\]

\subsubsection{The realization category at $\ell\neq p,\infty.$}

\begin{plain}
\label{b5}Let $\mathsf{R}_{\ell}(\mathbb{F}_{p^{m}})$ denote the category of
finite-dimensional $\mathbb{Q}{}_{\ell}$-vector spaces equipped with a
continuous semisimple action of $\Gamma_{m}\overset{\df}{=}\Gal(\mathbb{F}%
/\mathbb{F}_{p^{m}})$. It is a tannakian category over $\mathbb{Q}{}_{\ell}$
with the forgetful functor $\omega$ as fibre functor. The affine group scheme
$T_{m}\overset{\df}{=}\mathcal{A}{}ut^{\otimes}(\omega)$ is the algebraic hull
of $\Gamma_{m}$ over $\mathbb{Q}{}_{\ell}$, and $\mathsf{R}_{\ell}%
(\mathbb{F}{}_{p^{m}})\simeq\Rep_{\mathbb{Q}{}_{\ell}}(T_{m})$. In particular,
$T_{m}$ is commutative, and it is of multiplicative type because
$\mathsf{R}_{\ell}(\mathbb{F}{}_{p^{m}})$ is semisimple. On extending scalars
to $\mathbb{Q}{}_{\ell}^{\mathrm{al}}$, we see that $\Rep_{\mathbb{Q}{}_{\ell
}}(T_{m})_{(\mathbb{Q}_{\ell}^{\mathrm{al}})}=\Rep_{\mathbb{Q}_{\ell
}^{\mathrm{al}}}(T_{m})$ is the category of continuous semisimple
representations of $\Gal(\mathbb{F}/\mathbb{F}{}_{p^{m}})$ on
finite-dimensional $\mathbb{Q}{}_{\ell}^{\mathrm{al}}$-vector spaces. One
shows easily that this is the category of diagonalizable representations of
$\Gal(\mathbb{F}/\mathbb{F}{}_{p^{m}})$ on finite-dimensional $\mathbb{Q}%
_{\ell}^{\mathrm{al}}$-vector spaces such that the eigenvalues of the
Frobenius element in $\Gal(\mathbb{F}{}/\mathbb{F}{}_{p^{m}})$ are $\ell$-adic
units. The simple representations are one-dimensional, parametrized by the
units in $\mathcal{O}{}_{\mathbb{Q}{}_{\ell}^{\mathrm{al}}}$. Therefore%
\[
X^{\ast}(T_{m})\simeq\mathcal{O}_{\mathbb{Q}_{\ell}^{\mathrm{al}}}^{\times}.
\]
The map on characters corresponding to $\mathsf{R}{}_{\ell}(\mathbb{F}%
{}_{p^{m}})\rightarrow\mathsf{R}{}_{\ell}(\mathbb{F}{}_{p^{m^{\prime}}})$ is
$a\mapsto a^{m^{\prime}/m}$. Let $T=T_{1}$. There is an exact sequence%
\[
1\rightarrow T^{\circ}\rightarrow T\rightarrow\Gamma_{\mathbb{Q}{}_{\ell}%
}\rightarrow1,
\]
where $\Gamma_{\mathbb{Q}{}_{\ell}}$ is the profinite $\mathbb{Q}{}_{\ell}%
$-group defined by $\Gal(\mathbb{F}{}/\mathbb{F}{}_{p}).$
\end{plain}

\begin{plain}
\label{b6}Let $\mathsf{R}_{\ell}(\mathbb{F}{})=\varinjlim R_{\ell}%
(\mathbb{F}{}_{p^{m}})$. This is a tannakian category over $\mathbb{Q}{}%
_{\ell}$ with a canonical $\mathbb{Q}_{\ell}$-valued fibre functor $\omega$,
namely, the forgetful functor.
\end{plain}

\subsubsection{The crystalline realization category.}

\begin{plain}
\label{b6b}When $(M,F)$ is an isocrystal over $\mathbb{F}{}_{q}$, $q=p^{n}$,
we let $\pi_{M}=F^{n}$. It is an endomorphism of $M$ regarded as a vector
space over the field of fractions of $W(\mathbb{F}{}_{q})$.
\end{plain}

\begin{plain}
\label{b6a}The following conditions on an isocrystal $(M,F)$ over
$\mathbb{F}{}_{q}$ are equivalent:

\begin{enumerate}
\item $(M,F)$ is semisimple, i.e., it is a direct sum of simple isocrystals
over $\mathbb{F}{}_{q}$;

\item $\End(M,F)$ is semisimple;

\item $\pi_{M}$ is a semisimple endomorphism of $M$.
\end{enumerate}

\noindent When these conditions hold, the centre of $\End(M,F)$ is
$\mathbb{Q}{}_{p}[\pi_{M}]$. See, for example, \cite{milne1994a}, 2.10.
\end{plain}

\begin{plain}
\label{b7}Let $\mathsf{R}{}_{p}(\mathbb{F}{}_{q})$ be the category of
semisimple $F$-isocrystals over $\mathbb{F}{}_{q}$. When $V$ is an object of
weight $0$,
\[
V^{F}\overset{\df}{=}\{v\in V\mid Fv=v\}
\]
is a $\mathbb{Q}{}_{p}$-structure on $V$ . The functor $V\rightsquigarrow
V^{F}$ is a $\mathbb{Q}{}_{p}$-valued fibre functor on the tannakian
subcategory of isocrystals of weight $0$
\end{plain}

\begin{plain}
\label{b8}Let $\mathsf{R}{}_{p}(\mathbb{F}{})=\varinjlim\mathsf{R}{}%
_{p}(\mathbb{F}_{q})$. Then $\mathsf{R}{}_{p}(\mathbb{F}{})$ is a semisimple
tannakian category over $\mathbb{Q}_{p}$.
\end{plain}

\begin{caution}
\label{b9}The canonical functor
\[
\mathrm{\varinjlim}\Isoc(\mathbb{F}_{q})\rightarrow\Isoc(\mathbb{F})
\]
is faithful and essentially surjective, but not full. For example, if $A_{1}$
and $A_{2}$ are ordinary elliptic curves over $\mathbb{F}{}_{q}$ with
different $j$-invariants, and $\Lambda_{1}$ and $\Lambda_{2}$ are their
Dieudonn\'{e} isocrystals, then%
\[
\left\{
\begin{array}
[c]{l}%
\Hom_{\varinjlim\Isoc(\mathbb{F}{}_{q})}(\Lambda_{1},\Lambda_{2}%
)=\Hom(A_{1\mathbb{F}{}},A_{2\mathbb{F}{}})=0\text{, but}\\
\Hom_{\Isoc(\mathbb{\mathbb{F}{})}}(\Lambda_{1},\Lambda_{2})\approx
\mathbb{Q}{}_{p}\oplus\mathbb{Q}{}_{p}.
\end{array}
\right.
\]

\end{caution}

\subsubsection{The realization category at infinity.}

\begin{plain}
\label{b10}Let $\mathsf{R}_{\infty}$ be the category of pairs $(V,F)$
consisting of a $\mathbb{Z}$-graded finite-dimensional complex vector space
$V=\bigoplus\nolimits_{m\in\mathbb{Z}}V^{m}$ and an $\iota$-semilinear
endomorphism $F$ such that $F^{2}=(-1)^{m}$ on $V^{m}$. With the obvious
tensor structure, $\mathsf{R}_{\infty}$ becomes a tannakian category over
$\mathbb{R}$ with fundamental group $\mathbb{G}_{m}$. The objects fixed by
$\mathbb{G}_{m}$ are those of weight zero. If $(V,F)$ is of weight zero, then%
\[
V^{F}\overset{\textup{{\tiny def}}}{=}\{v\in V\mid Fv=v\}
\]
is an $\mathbb{R}$-structure on $V$. The functor $V\rightsquigarrow V^{F}$ is
an $\mathbb{R}$-valued fibre functor on $\mathsf{R}_{\infty}^{\mathbb{G}_{m}}$.
\end{plain}

\subsection{The local realizations}

In this subsection, we construct, for each prime $l$ (including $p$ and
$\infty)$, an exact tensor functor $\xi_{l}$ from $\Mot^{w}(\mathbb{Q}%
{}^{\mathrm{al}})$ to the $\mathbb{Q}_{l}$-linear tannakian category
$\mathsf{R}_{l}({\mathbb{F}})$,%
\[
\xi_{l}\colon\Mot^{w}(\mathbb{Q}{}^{\mathrm{al}})\rightarrow\mathsf{R}{}%
_{l}(\mathbb{F}{}).
\]
The main goal of this article is to construct a universal factorization of
these functors through a $\mathbb{Q}$-linear tannakian category,%
\[
\begin{tikzcd}
\Mot^{w}(\mathbb{Q}^{\mathrm{al}})\arrow{r}[swap]{R}
\arrow[rr,bend left=15,"\xi_l"]
&\Mot(\mathbb{F})\arrow{r}[swap]{\eta_{l}}
&\mathsf{R}_{l}(\mathbb{F).}%
\end{tikzcd}
\]
Each of the functors $\xi_{l}$ defines a fibre functor\footnote{See \ref{p1}
for the action of $P$ on $\Mot^{w}(\mathbb{Q}^{\mathrm{al}})$.}
\[
\omega_{l}\colon\Mot^{w}(\mathbb{Q}^{\mathrm{al}})^{P}\rightarrow
\Vc(\mathbb{Q}_{l}),
\]
and, to achieve our goal, we need to define a fibre functor
\[
\omega_{0}\colon\Mot^{w}(\mathbb{Q}^{\mathrm{al}})^{P}\rightarrow
\Vc(\mathbb{Q})
\]
that is a $\mathbb{Q}$-structure on the restricted product of the $\omega_{l}$
(see \S 5).

\subsubsection{The local realization at $\ell$.}

\begin{plain}
\label{b16}For each $\ell\neq p,\infty$, we let $\xi_{\ell}$ and $\omega
_{\ell}$ denote the functors on $\Mot^{w}(\mathbb{Q}^{\mathrm{al}})$ defined
by $\ell$-adic \'{e}tale cohomology.
\end{plain}

\subsubsection{The local realization at $p$.}

\begin{plain}
\label{b17} The map%
\[
(A,e,m)\mapsto e\cdot H_{\mathrm{crys}}^{\ast}(A_{0})(m)
\]
extends to an exact tensor functor
\[
\xi_{p}\colon\Mot^{w}(\mathbb{Q}^{\mathrm{al}})\rightarrow\mathsf{R}%
_{p}(\mathbb{F})\text{.}%
\]
Let $x_{p}$ denote the homomorphism $\mathbb{G}\rightarrow G_{\mathbb{Q}_{p}}$
defined by $\xi_{p}$. We obtain a $\mathbb{Q}_{p}$-valued fibre functor
$\omega_{p}$ on $\Mot^{w}(\mathbb{Q}^{\mathrm{al}})^{\mathbb{G}}$ as follows:%
\[
\begin{tikzcd}[column sep=large]
\Mot^w(\mathbb{Q}^{\mathrm{al}})^{\mathbb{G}}
\arrow[r,"\xi_{p}"']
\arrow[rr,bend left=15,"\omega_p"]
&\mathsf{R}_{p}^{\mathbb{G}}
\arrow[r,"V\rightsquigarrow V^{F}"']
&\Vc(\mathbb{Q}_p).
\end{tikzcd}
\]

\end{plain}

\subsubsection{The local realization at $\infty$.}

\begin{plain}
\label{b15}Let $(V,h)$ be a real Hodge structure, and let $C$ act on $V$ as
$h(i)$. Then the square of the operator $v\mapsto C\bar{v}$ acts as $(-1)^{m}$
on $V^{m}$. Therefore, $\mathbb{C}\otimes_{\mathbb{R}}V$ endowed with its
weight gradation and this operator is an object of $\mathsf{R}_{\infty}.$ We
let
\[
\xi_{\infty}\colon\Mot^{w}(\mathbb{Q}^{\mathrm{al}})\rightarrow\mathsf{R}%
_{\infty},\quad X\rightsquigarrow(\omega_{B}(X)_{\mathbb{R}{}},C),
\]
denote the functor sending $X$ to the object of $\mathsf{R}_{\infty}$ defined
by the real Hodge structure $\omega_{B}(X)_{\mathbb{R}}$. Then $\xi_{\infty}$
is an exact tensor functor, and the cocharacter $x_{\infty}\colon
\mathbb{G}_{m}\rightarrow G_{\mathbb{R}}$ it defines is equal to
$w_{\mathrm{\mathbb{R}}}$. We obtain an $\mathbb{R}$-valued fibre functor
$\omega_{\infty}$ on $\Mot^{w}(\mathbb{Q}^{\mathrm{al}})^{\mathbb{G}_{m}}$ as
follows:%
\[
\begin{tikzcd}[column sep=large]
\Mot^w(\mathbb{Q}^{\mathrm{al}})^{\mathbb{G}_m}
\arrow[r,"\xi_{\infty}"']
\arrow[rr,bend left=15,"\omega_\infty"]
&\mathsf{R}_{\infty}^{\mathbb{G}_m}
\arrow[r,"V\rightsquigarrow V^{F}"']
&\Vc(\mathbb{R}).
\end{tikzcd}
\]

\end{plain}

\subsection{The proper-smooth base change theorem}

Let $S$ be a connected normal scheme and $f\colon X\rightarrow S$ a smooth
proper morphism. Then $R^{r}f_{\ast}\mathbb{Q}{}_{\ell}$ is a locally constant
sheaf. Let%
\[
M=\left(  R^{r}f_{\ast}\mathbb{Q}{}_{\ell}\right)  _{\bar{\eta}}\simeq
H^{r}(X_{\bar{\eta}},\mathbb{Q}{}_{\ell}).
\]
Then%
\[
H^{0}(S,R^{r}f_{\ast}\mathbb{Q}{}_{\ell})=M^{\pi_{1}(S)},
\]
and, for any closed point $s$ of $S$,%
\[
\left(  R^{r}f_{\ast}\mathbb{Q}{}_{\ell}\right)  _{s}=H^{r}(X_{s},\mathbb{Q}%
{}_{\ell})=M^{\pi_{1}(s)}.
\]

\subsection{The Leray spectral sequence}

\begin{theorem}
[Blanchard, Deligne]\label{l7}If $f\colon X\rightarrow S$ is smooth projective
morphism of smooth varieties over $\mathbb{C}{}$, then the Leray spectral
sequence,
\[
H^{p}(S,R^{q}f_{\ast}\mathbb{Q}{})\implies H^{p+q}(X,\mathbb{Q}{}),
\]
degenerates at $E_{2}$.
\end{theorem}

\begin{proof}
The relative Lefschetz operator $L=c_{1}(\mathcal{L}{})\cup\cdot$ acts on the
whole spectral sequence, and induces a Lefschetz decomposition%
\[
R^{q}f_{\ast}\mathbb{Q}{}=\bigoplus\nolimits_{r}L^{r}(R^{q-2r}f_{\ast
}\mathbb{Q}{})_{\mathrm{prim}}.
\]
It suffices to prove that $d_{2}\alpha=0$ for $\alpha\in H^{p}(S,(R^{q}%
f_{\ast}\mathbb{Q})_{\mathrm{prim}})$. In the diagram,%
\[
\begin{tikzcd}
H^{p}(S,(R^{q}f_{\ast}\mathbb{Q})_{\mathrm{prim}})\arrow{r}{d_2}
\arrow{d}{L^{n-q+1}}[swap]{0}
&H^{p+2}(S,R^{q-1}f_{\ast}\mathbb{Q})\arrow{d}{L^{n-q+1}}[swap]{\simeq}\\
H^{p}(S,R^{2n-q+2}\mathbb{Q})\arrow{r}{d_2}
&H^{p+2}(S,R^{2n-q+1}\mathbb{Q}),
\end{tikzcd}
\]
the map at left is zero because $L^{n-q+1}$ is zero on $(R^{q}f_{\ast
}\mathbb{Q}{})_{\mathrm{prim}}$ and the map at right is an isomorphism because
$L^{n-q+1}\colon R^{q-1}f_{\ast}\mathbb{Q}{}\rightarrow R^{2n-q+1}f_{\ast
}\mathbb{Q}{}$ is an isomorphism. Hence $d_{2}\alpha=0.$
\end{proof}

Grothendieck conjectured the degeneration of the Leray spectral sequence by
consideration of weights. \cite{blanchard1956} proved the result when the base
is simply connected, and \cite{deligne1968} proved it in general. See also
\cite{GriffithsH1978}, p.~466.

Now consider an abelian scheme $f\colon A\rightarrow S$. For $n\in\mathbb{N}%
{}$, let $\theta_{n}$ denote the endomorphism of $A/S$ acting as
multiplication by $n$ on the fibres. By a standard argument
(\cite{kleiman1968}, p.~374), $\theta_{n}^{\ast}$ acts as $n^{j}$ on
$R^{j}f_{\ast}\mathbb{Q}{}$. As $\theta_{n}^{\ast}$ commutes with the
differentials $d_{2}$ of the Leray spectral sequence $H^{i}(S,R^{j}f_{\ast
}\mathbb{Q}{})\implies H^{i+j}(A,\mathbb{Q}{})$,
\[
\begin{tikzcd}
H^{p}(S,R^{q}f_{\ast}\mathbb{Q})\arrow{r}{d_2}
\arrow{d}{\theta_n}[swap]{n^q}
&H^{p+2}(S,R^{q-1}f_{\ast}\mathbb{Q})
\arrow{d}{\theta_n}[swap]{n^{q-1}}\\
H^{p}(S,R^{q}f_{\ast}\mathbb{Q})\arrow{r}{d_2}
&H^{p+2}(S,R^{q-1}f_{\ast}\mathbb{Q}),
\end{tikzcd}
\]
we see that the spectral sequence degenerates at the $E_{2}$-term and
\[
H^{2r}(A,\mathbb{Q}{})\simeq\bigoplus_{i+j=2r}H^{i}(S,R^{j}f_{\ast}%
\mathbb{Q})
\]
with $H^{i}(S,R^{j}f_{\ast}\mathbb{Q}{})$ the subspace of $H^{2r}%
(A,\mathbb{Q})$ on which $\theta_{n}$ acts as $n^{j}$. As $\theta_{n}^{\ast}$
preserves algebraic classes, this induces a decomposition%
\[
a\!H^{2r}(A,\mathbb{Q}{})\simeq\bigoplus_{i+j=2r}a\!H^{i}(S,R^{j}f_{\ast
}\mathbb{Q})
\]
of the subspaces of algebraic classes.

\begin{theorem}
[\cite{deligne1971}, 4.1.1]\label{l8}Let $f\colon X\rightarrow S$ be a smooth
proper morphism of smooth varieties over $\mathbb{C}$.

\begin{enumerate}
\item The Leray spectral sequence
\[
H^{r}(S,R^{s}f_{\ast}\mathbb{Q})\Rightarrow H^{r+s}(X,\mathbb{Q})
\]
degenerates at $E_{2}$; in particular, the edge morphism
\[
H^{n}(X,\mathbb{Q})\rightarrow\Gamma(S,R^{n}f_{\ast}\mathbb{Q})
\]
is surjective.

\item If $\bar{X}$ is a smooth compactification of $X$ with $\bar
{X}\smallsetminus X$ a union of smooth divisors with normal crossings, then
the canonical morphism
\[
H^{n}(\bar{X},\mathbb{Q})\rightarrow H^{0}(S,R^{n}f_{\ast}\mathbb{Q})
\]
is surjective.

\item Let $(R^{n}f_{\ast}\mathbb{Q})^{0}$ be the largest constant local
subsystem of $R^{n}\pi_{\ast}\mathbb{Q}$ (so $(R^{n}f_{\ast}\mathbb{Q}%
)_{s}^{0}=\Gamma(S,R^{n}f_{\ast}\mathbb{Q})$ for all $s\in S(\mathbb{C}{})$).
For each $s\in S$, $(R^{n}f_{\ast}\mathbb{Q})_{s}^{0}$ is a Hodge substructure
of $(R^{n}f_{\ast}\mathbb{Q})_{s}=H^{n}(X_{s},\mathbb{Q})$, and the induced
Hodge structure on $\Gamma(S,R^{n}f_{\ast}\mathbb{Q})$ is independent of $s$.
\end{enumerate}

\noindent In particular, the map
\[
H^{n}(\bar{X},\mathbb{Q})\rightarrow H^{n}(X_{s},\mathbb{Q})
\]
has image $(R^{n}f_{\ast}\mathbb{Q})_{s}^{0}$, and its kernel is independent
of $s$.
\end{theorem}

Part (b) follows from (a) and the theory of weights. There is an $\ell$-adic
variant of Theorem \ref{l8}.

\begin{theorem}
\label{l9}Let $S$ be a smooth connected scheme over an algebraically closed
field $k$, let $f\colon X\rightarrow S$ be a smooth projective morphism, and
let $\bar{X}$ be a smooth projective compactification of $X$. For all $n$, the
canonical map%
\[
H^{n}(\bar{X},\mathbb{Q}{}_{\ell})\rightarrow H^{0}(S,R^{n}f_{\ast}%
\mathbb{Q}{}_{\ell})
\]
is surjective.
\end{theorem}

\begin{proof}
When $k$ has characteristic zero, this follows from the case $k=\mathbb{C}{}$.
When $k=\mathbb{F}{}$, the same argument as in the case $k=\mathbb{C}{}$
applies when one takes weights in the sense of \cite{deligne1980}. Otherwise,
in characteristic $p$, it can be proved by a specialization argument (see
\cite{andre2006a}, 1.1.1).
\end{proof}

\begin{aside}
\label{l10}Let $f\colon X\rightarrow S$ be a smooth projective morphism of
smooth algebraic varieties over $\mathbb{C}{}$. Then the Leray spectral
sequence degenerates at $E_{2}$, so%
\[
H^{r}(X,\mathbb{Q}{})\approx\bigoplus_{i}H^{i}(S,R^{r-i}f_{\ast}\mathbb{Q}%
{}).
\]
Moreover, $H^{r}(X,\mathbb{Q}{})$ is equipped with a mixed Hodge structure.
Each summand $H^{i}(S,R^{r-i}f_{\ast}\mathbb{Q}{})$ is equipped with a pure
Hodge structure if $S$ is complete, but not in general otherwise.
\end{aside}

\section{Conjecture A}

\label{CA}

In this section, we state the rationality conjecture (Conjecture A), and we
suggest possible proofs of it.

Throughout, $\mathbb{Q}^{\mathrm{al}}$ is the algebraic closure of
$\mathbb{Q}{}$ in $\mathbb{C}{}$, $w$ is a prime of $\mathbb{Q}{}%
^{\mathrm{al}}$ lying over $p$, and $\mathbb{F}{}$ is the residue field at $p$.

\subsection{Statement of the conjecture}

\begin{definition}
\label{x1}Let $X$ be a smooth projective variety over $\mathbb{Q}%
{}^{\mathrm{al}}$ with good reduction at $w$ to a variety $X_{0}$ over
$\mathbb{F}{}$. An absolute Hodge class $\gamma$ on $X$ is $w$\emph{-rational}
if $\langle\gamma_{0}\cdot\delta\rangle\in\mathbb{Q}{}$ for all Lefschetz
classes $\delta$ on $X_{0}$ of complementary dimension.
\end{definition}

Note that algebraic classes are $w$-rational.

\begin{conjecturea}
\label{c2} Let $A$ be an abelian variety over $\mathbb{Q}^{\mathrm{al}}$ with
good reduction at $w$. All (absolute) Hodge classes on $A$ are $w$%
-rational.\footnote{The conjecture should also be stated, \textit{mutatis
mutandis}, for abelian varieties over $\mathbb{Q}^{\mathrm{al}}$ with bad
(semistable) reduction.}
\end{conjecturea}

In more detail, a Hodge class on $A$ is an element $\gamma$ of $H_{\mathbb{A}%
{}}^{2\ast}(A)(\ast)$, and its specialization $\gamma_{0}$ is an element of
$H_{\mathbb{A}{}}^{2\ast}(A_{0})(\ast)$. If $\delta_{1},\ldots,\delta_{r}$,
where $r=\dim(\gamma)$, are divisor classes on $A_{0}$, then%
\[
\langle\gamma_{0}\cdot\delta_{1}\cdots\delta_{r}\rangle\in H_{\mathbb{A}}%
^{2d}(A_{0})(d)\simeq\mathbb{A}{}_{f}^{p}\times\mathbb{Q}{}_{w}^{\mathrm{al}%
},\quad d=\dim A.
\]
The conjecture says that it lies in $\mathbb{Q}{}\subset\mathbb{A}{}_{f}%
^{p}\times\mathbb{Q}{}_{w}^{\mathrm{al}}$.

\begin{plain}
\label{c5} If $\gamma$ is algebraic, then $\gamma_{0}$ is algebraic, and so
Conjecture A holds for $\gamma$. Thus, Conjecture A holds for $A$ if the Hodge
conjecture holds for $A$, for example, if $A$ has no exotic Hodge classes.
\end{plain}

\begin{proposition}
\label{c4a}If $\End(A)\simeq\End(A_{0}),$ then Conjecture $A$ holds for $A$
and its powers.
\end{proposition}

\begin{proof}
The hypothesis implies that $L(A)\simeq L(A_{0})$. If follows that, for all
$n$, the specialization map $\mathcal{D}^{\ast}(A^{n})\rightarrow
\mathcal{D}^{\ast}(A_{0}^{n})$ becomes an isomorphism when tensored with
$\mathbb{Q}_{\ell}$ (any $\ell\neq p$),%
\[
\begin{tikzcd}
\mathcal{D}^{\ast}(A^n)\otimes_{\mathbb{Q}}\mathbb{Q}_{\ell}\arrow{r}{\simeq}\arrow{d}
&H^{2\ast}(A^n,\mathbb{Q}_{\ell}(\ast))^{L(A)}\arrow{d}{\simeq}\\
\mathcal{D}^{\ast}(A^n_0)\otimes_{\mathbb{Q}{}}\mathbb{Q}{}_{\ell}\arrow{r}{\simeq}
&H^{2\ast}(A^n_0,\mathbb{Q}_{\ell}(\ast))^{L(A_0)}.
\end{tikzcd}
\]
Therefore, it is an isomorphism, i.e., every Lefschetz class $\delta$ on
$A_{0}^{n}$ lifts uniquely to a Lefschetz class $\delta^{\prime}$ on $A^{n}$.
If $\gamma$ is a Hodge class on $A^{n}$ and $\delta$ is a Lefschetz class on
$A_{0}$ of complementary dimension, then%
\[
\langle\gamma_{0}\cup\delta\rangle=\langle\gamma\cup\delta^{\prime}\rangle
\in\mathbb{Q}{}.
\]

\end{proof}

The hypothesis in \ref{c4a} holds if $A$ is CM and $A_{0}$ is a product of
simple ordinary abelian varieties, no two of which are isogenous.

\begin{proposition}
\label{x6}Let $X$ and $Y$ be smooth projective varieties over $\mathbb{Q}%
{}^{\mathrm{al}}$ with good reduction at $w$, and let $f\colon X\rightarrow Y$
be a morphism. If $\gamma$ is $w$-rational on $X$, then $f_{\ast}\gamma$ is
$w$-rational on $B$.
\end{proposition}

\begin{proof}
Let $\delta$ be a Lefschetz class on $Y$ of complementary dimension to
$f_{\ast}\gamma$. Then%
\[
\langle\left(  f_{\ast}\gamma\right)  _{0}\cdot\delta\rangle=\langle f_{0\ast
}\gamma_{0}\cdot\delta\rangle=\langle f_{0\ast}(\gamma_{0}\cdot f_{0}^{\ast
}\delta)\rangle=\langle\gamma_{0}\cdot f_{0}^{\ast}\delta\rangle\in
\mathbb{Q}{}%
\]
because $f_{0}^{\ast}\delta$ is Lefschetz.
\end{proof}

\begin{proposition}
\label{x5}Let $A$ and $B$ be abelian varieties over $\mathbb{Q}{}%
^{\mathrm{al}}$ with good reduction at $w$, and let $f\colon A\rightarrow B$ a
morphism. If $\gamma$ is $w$-rational on $B$, then $f^{\ast}\gamma$ is
$w$-rational on $A$.
\end{proposition}

\begin{proof}
Let $\delta$ be a Lefschetz class on $A_{0}$ of complementary dimension to
$f^{\ast}\gamma$. Then%
\[
\langle(f^{\ast}\gamma)_{0}\cdot\delta\rangle=\langle f_{0}^{\ast}\gamma
_{0}\cdot\delta\rangle=\langle f_{0\ast}(f_{0}^{\ast}\gamma_{0}\cdot
\delta)\rangle=\langle\gamma_{0}\cdot f_{0\ast}\delta\rangle\in\mathbb{Q}{}%
\]
because $f_{0\ast}\delta$ is Lefschetz (\cite{milne1999b}, 5.5).
\end{proof}

Let $X$ be a smooth projective variety of dimension $n$ and $L$ a Lefschetz
operator. The following is one form of the Lefschetz standard conjecture:
\[
A(X,L): \text{The map $L^{n-2r}\colon A^{r}(X)\rightarrow A^{n-r}(X)$ is an
isomorphism for all } r\leq n/2.
\]
Equivalently, a cohomology class $\gamma$ lies in $A^{r}(X)$ if $L^{n-2r}%
\gamma$ lies in $A^{n-r}(X)$.

\begin{proposition}
\label{x11}Assume that $X$ has good reduction at $w$. If $A(X_{0},L_{0})$
holds for Lefschetz classes on $X_{0}$, then $A(X,L)$ holds for $w$-rational classes.
\end{proposition}

\begin{proof}
To see this, let $\gamma$ be a Hodge class on $X$ of codimension $r$ such that
$L^{n-2r}\gamma$ is $w$-rational, and let $\delta$ be a Lefschetz class on
$X_{0}$ of codimension $n-r$. Then $\delta=L_{0}^{n-2r}\delta^{\prime}$ for
some Lefschetz class $\delta^{\prime}$ of codimension $r$ on $X_{0}$, and%
\[
\langle\gamma_{0}\cdot\delta\rangle=\langle\gamma\cdot L_{0}^{n-2r}%
\delta^{\prime}\rangle=\langle(L^{n-2r}\gamma)_{0}\cdot\delta^{\prime}%
\rangle\in\mathbb{Q}\text{.}%
\]
In particular, we see that $A(X,L)$ holds for $w$-rational classes when $X$ is
an abelian variety. Therefore, $\ast$ preserves $w$-Lefschetz classes on an
abelian variety, the Hodge standard conjecture holds, and conjecture $D(X)$ holds.
\end{proof}

\begin{corollary}
\label{x11a}Let $A$ be an abelian variety over $\mathbb{Q}{}^{\mathrm{al}}$
with good reduction at $w$. The Lefschetz standard conjecture holds for
$w$-rational classes on $A$.
\end{corollary}

\begin{proof}
The Lefschetz standard conjecture holds for Lefschetz classes on abelian varieties.
\end{proof}

\subsection{Nifty abelian varieties}

We let $\MT(A)$ denote the Mumford--Tate group of an abelian variety, and
$\SMT(A)$ its special Mumford--Tate group.

\begin{definition}
\label{n1}Let $A$ be an abelian variety over $\mathbb{Q}{}^{\mathrm{al}}$ with
good reduction at $w$. We say that $A$ is \emph{nifty} if $\MT(A)\cdot
L(A_{0})=L(A)$; equivalently, $\SMT(A)\cdot S(A_{0})=S(A)$.
\end{definition}

As $\SMT(A)=\SMT(A^{r})$, $S(A_{0})=S(A_{0}^{r})$, and $S(A)=S(A^{r})$ for all
$r\geq1$, $A^{r}$ is nifty if $A$ is.

\begin{example}
\label{n2}An abelian variety $A$ is nifty if $\MT(A)=L(A)$, i.e., if all Hodge
classes on $A$ and its powers are Lefschetz. There is a large literature
listing abelian varieties satisfying this condition.
\end{example}

\begin{example}
\label{n3}If $\End^{0}(A)=\End^{0}(A_{0})$, then $A$ is nifty. This only
happens when $A$ is CM.
\end{example}

\begin{proposition}
\label{n4}Nifty abelian varieties satisfy Conjecture A.
\end{proposition}

\begin{proof}
Let $A$ be a nifty abelian variety over $\mathbb{Q}^{\mathrm{al}}$, and let
$\langle A\rangle_{H}$ and $\langle A\rangle_{L}$ denote the tannakian
subcategories of $\Mot(\mathbb{Q}^{\mathrm{al}})$ and $\LMot(\mathbb{Q}
{}^{\mathrm{al}})$ generated by $A$. To simplify, we assume that $A$ is CM (so
that the groups involved are commutative). Let $P$ denote the kernel of the
homomorphism $\MT(A)\rightarrow L(A)/L(A_{0})$, and consider the diagrams
\[
\begin{tikzcd}
\langle A\rangle_{H}^{P}\arrow{d}
&\langle A\rangle_{L}^{L(A_{0})}\arrow{l}\arrow{d}
&\MT(A)/P\arrow{r}&L(A)/L(A_{0})\\
\langle A\rangle_{H}\arrow[dashed]{d} & \langle A\rangle_{L}\arrow{l}\arrow{d} &
\MT(A)\arrow[hook]{r}\arrow{u} & L(A)\arrow{u}\\
\langle A_{0}\rangle & \langle A_{0}\rangle_{L}\arrow[dashed]{l}
& P\arrow[hook]{r}\arrow[hook]{u} & L(A_{0})\arrow[hook]{u}
\end{tikzcd}
\]
of tannakian categories and their fundamental groups. The category $\langle
A_{0}\rangle_{L}$ is a quotient of $\langle A\rangle_{L}$, and we let $\omega$
denote the fibre functor on $\langle A\rangle_{L}^{L(A_{0})}$ corresponding to
it. Because the homomorphism $\MT(A)/P\rightarrow L(A)/L(A_{0})$ is an
isomorphism, the functor $\langle A\rangle_{L}^{L(A_{0})}\rightarrow\langle
A\rangle_{H}^{P}$ is an equivalence of tensor categories, and so $\omega$
defines a fibre functor on $\langle A\rangle_{H}^{P}$, which we again denote
by $\omega$. Define $\langle A_{0}\rangle$ to be the quotient of $\langle
A\rangle_{H}/\omega$ of $\langle A\rangle_{H}$.

Let $\gamma$ be a Hodge class on $A$ and $\delta$ a Lefschetz class on $A_{0}$
of complementary dimension. Then it is obvious from the diagram that
$(\gamma_{\mathbb{A}})_{0}\cdot\delta_{\mathbb{A}{}}\in\mathbb{Q}$ because the
intersection takes place inside the $\mathbb{Q}$-algebra%
\[
\Hom(\1,hA_{0}),
\]
where $hA_{0}$ is the object of $\langle A_{0}\rangle\overset{\df}{=}\langle
A\rangle_{H}/\omega$ defined by $A_{0}$.
\end{proof}

\subsection{Weil families}

\begin{definition}
\label{l11}Let $(A,\nu,\lambda)$ be a Weil triple over $\mathbb{Q}%
{}^{\mathrm{al}}$.

\begin{enumerate}
\item We say $(A,\lambda,\nu)$ has \emph{good reduction} at $w$ if $A$ has
good reduction at $w$ (then $\lambda$ specializes to a polarization on
$\lambda_{0})$.

\item We say that $(A,\lambda,\nu)$ is CM if $A$ is CM.
\end{enumerate}
\end{definition}

\begin{definition}
\label{t9}Let $(A_{1},\lambda_{1},\nu_{1})$ and $(A_{2},\lambda_{2},\nu_{2})$
be two Weil triples over $k\subset\mathbb{C}$ relative to the CM-algebra $E$.
We say that $(A_{1},\lambda_{1},\nu_{1})$ and $(A_{2},\lambda_{2},\nu_{2})$
\emph{lie in the same Weil family} if there exists an $E$-linear isomorphism
\[
H_{1}(A_{1},\mathbb{Q}{})\rightarrow H_{1}(A_{2},\mathbb{Q}{})
\]
under which the Riemann forms of $\lambda_{1}$ and $\lambda_{2}$ correspond up
to an element of $\mathbb{Q}^{\times}$.
\end{definition}

In fact, they do then lie in the same Weil family (see \cite{deligne1982},
proof of Theorem 4.8).

\begin{lemma}
\label{t11}Let $(A_{1},\lambda_{1},\nu_{1})$ and $(A_{2},\lambda_{2},\nu_{2})$
be Weil triples over $\mathbb{Q}{}^{\mathrm{al}}$ having good reduction to
isogenous triples over $\mathbb{F}$. If $(A_{1},\lambda_{1},\nu_{1})$ and
$(A_{2},\lambda_{2},\nu_{2})$ lie in the same Weil family, then the $E$-vector
spaces of Weil classes on $A$ and $A^{\prime}$ specialize to the same
$E$-subspace of
\[
H_{\mathbb{A}}^{2d}((A_{1})_{0})(d)\simeq H_{\mathbb{A}}^{2d}((A_{2}%
)_{0})(d).
\]

\end{lemma}

\begin{proof}
Let $f\colon A\rightarrow S$ be the universal family over $\mathbb{C}{}$
containing $(A_{1},\lambda_{1},\nu_{1})$ and $(A_{2},\lambda_{2},\nu_{2})$
(see \cite{deligne1982}, 4.8). Then $f$ has a unique model over $\mathbb{Q}%
{}^{\mathrm{al}}$, also denoted $f\colon A\rightarrow S$, such that the
special points of $S(\mathbb{C}{})$ lie in $S(\mathbb{Q}{}^{\mathrm{al}})$.
The Weil classes form a local $\mathbb{Q}{}$-subsystem $W_{E}(A/S)$ of
$R^{2d}f_{\ast}\mathbb{A}{}(d)$, where $d=\dim A/[E\colon\mathbb{Q}{}]$. By
assumption, the points in $S(\mathbb{Q}{}^{\mathrm{al}})$ corresponding to
$(A_{1},\lambda_{1},\nu_{1})$ and $(A_{2},\lambda_{2},\nu_{2})$ specialize to
the same point in $S(\mathbb{F})$ (or a translate), from which the statement follows.
\end{proof}

\subsection{An inductive approach to proving Conjecture A}

We suggest an approach to proving Conjecture A for CM abelian varieties by
induction on the codimension of $\gamma$.

\begin{proposition}
\label{x17}Let $A$ be an abelian variety over $\mathbb{Q}{}^{\mathrm{al}}$
with good reduction at $w$, and let $r\in\mathbb{N}$. If all Hodge classes of
codimension $r$ on $A$ are $w$-rational, then the same is true of the Hodge
classes of dimension $r$ on $A$.
\end{proposition}

\begin{proof}
Let $g=\dim A$. Let $\gamma$ be a Hodge class of dimension $r$ on $A$, and let
$\delta$ be a Lefschetz class on $A_{0}$ of codimension $r$. Let $\lambda$ be
a polarization of $A$, and let $\xi$ be the corresponding ample divisor.

Suppose first that $2r\leq g$. In this case, there is an isomorphism%
\[
\xi^{g-2r}\colon H_{\mathbb{A}{}}^{2r}(A)(r)\rightarrow H_{\mathbb{A}{}%
}^{2g-2r}(A)(g-r).
\]
As the Lefschetz standard conjecture holds for Hodge classes, we can write
$\gamma=\xi^{g-2r}\cdot\gamma^{\prime}$, where $\gamma^{\prime}$ is a Hodge
class of codimension $r$ on $A$. Now
\[
\langle\gamma_{0}\cdot\delta\rangle=\langle(\xi^{g-2r}\cdot\gamma^{\prime
})_{0}\cdot\delta\rangle=\langle\gamma_{0}^{\prime}\cdot(\xi_{0}^{g-2r}%
\cdot\delta)\rangle,
\]
which lies in $\mathbb{Q}{}$ by the hypothesis.

When $2r>g$, we can replace the Lefschetz operator $L\colon x\mapsto\xi\cdot
x$ in the argument with its quasi-inverse $\Lambda$ (\cite{kleiman1994},
\S 4). As $\Lambda$ is a Lefschetz class (\cite{milne1999a}, 5.9), the same
argument applies.
\end{proof}

\begin{definition}
\label{x18}Let $(A,\nu,\lambda)$ be a CM Weil triple over $\mathbb{Q}%
{}^{\mathrm{al}}$. We say that a divisor $d$ on $A_{0}$ is \emph{liftable} if
there exists a CM Weil triple $(A_{1},\nu_{1},\lambda_{1})$ in the same Weil
family as $(A,\nu,\lambda)$ and a divisor $d_{1}$ on $A_{1}$ such that
$(A_{1},\nu_{1},\lambda_{1},d_{1})_{0}$ is isogenous to $(A_{0},\nu
_{0},\lambda_{0},d)$. We say that a Lefschetz class $\delta$ of codimension
$r$ on $A_{0}$ is \emph{weakly liftable} if it is in the $\mathbb{Q}$ span of
the classes $d_{1}\cdots d_{r}$ with at least one of the $d_{i}$ liftable.
\end{definition}

\begin{question}
\label{Q2}Let $(A,\nu,\lambda)$ be a CM Weil triple over $\mathbb{Q}%
{}^{\mathrm{al}}$. Is every Lefschetz class of codimension $\dim
(A)/[E\colon\mathbb{Q}{}]$ on $A_{0}$ weakly liftable?
\end{question}

\begin{theorem}
\label{x19}If Question \ref{Q2} has a positive answer, then Conjecture A is
true for all CM abelian varieties.
\end{theorem}

\begin{proof}
We have to prove that every Hodge class $\gamma$ on a CM abelian variety $A$
is $w$-rational. Hodge classes of codimension $\leq1$ are $w$-rational because
they are algebraic. We prove the theorem by induction on the codimension of
$\gamma$. Let codim$(\gamma)=r>1$, and assume that every Hodge class of codim
$<r$ on a CM abelian variety is $w$-rational.

After Theorem \ref{a6} and \ref{x5}, we may suppose that $\gamma$ is a Weil
class on a CM Weil triple $(A,\nu,\lambda)$, still of codimension $r$, so
$r=\dim A/[E\colon\mathbb{Q}]$.

Let $g=\dim A$ and let $\delta$ be a Lefschetz class of codimension $g-r$ on
$A_{0}$. We have to show that $\langle\gamma_{0}\cdot\delta\rangle
\in\mathbb{Q}{}$.

Let $\xi$ be the divisor class on $A$ attached to $\lambda$. The isomorphism
(strong Lefschetz)%
\[
\xi_{0}^{g-2r}\colon H_{\mathbb{A}{}}^{2r}(A_{0})(r)\rightarrow H_{\mathbb{A}%
{}}^{2g-2r}(A_{0})(g-r)
\]
induces an isomorphism%
\[
\mathcal{D}^{r}(A_{0})\rightarrow\mathcal{D}{}^{g-r}(A_{0})
\]
on Lefschetz classes (\cite{milne1999a}, 5.9). Therefore,%
\[
\delta=\xi_{0}^{g-2r}\cdot\delta^{\prime}%
\]
with $\delta^{\prime}$ a Lefschetz class of codimension $r$ on $A_{0}$.

Since we are assuming that Question \ref{Q2} has a positive answer, we may
suppose that $\delta^{\prime}=d\cdot\delta^{\prime\prime}$ with $d$ a liftable
divisor on $A_{0}$. This means that there exists a CM Weil triple $(A_{1}%
,\nu_{1},\lambda_{1})$ in the same Weil family as $(A,\nu,\lambda)$ and a
divisor class $d_{1}$ on $A_{1}$ such that $(A_{0},\nu_{0},\lambda_{0},d)$ is
isogenous to $(A_{1},\nu_{1},\lambda_{1},d_{1})_{0}$. According to \ref{t11},
there exists a Weil class $\gamma_{1}$ on $A_{1}$ such that $\gamma
_{0}=(\gamma_{1})_{0}$. Now%
\[
\langle\gamma_{0}\cdot\delta\rangle=\langle\gamma_{0}\cdot\xi_{0}^{g-2r}%
\cdot\delta^{\prime}\rangle=\langle\gamma_{0}\cdot\xi_{0}^{g-2r}\cdot
d\cdot\delta^{\prime\prime}\rangle=\langle(\gamma_{1}\cdot\xi_{1}^{g-2r}\cdot
d_{1})_{0}\cdot\delta^{\prime\prime}\rangle.
\]
This lies in $\mathbb{Q}{}$ because $\gamma_{1}\cdot\xi_{1}^{g-2r}\cdot d_{1}$
is a Hodge class of codimension%
\[
r+g-2r+1=g-(r-1)
\]
on $A$, hence of dimension $r-1$, and so we can apply Proposition \ref{x17}
and the induction hypothesis.
\end{proof}

\begin{plain}
\label{x22}If Conjecture A holds for CM abelian varieties over $\mathbb{Q}%
{}^{\mathrm{al}}$, then Conjecture C holds for the same class. Once the
arguments in \S 5 have been completed, this will imply that Conjecture C holds
for all abelian varieties over $\mathbb{C}$ with good reduction at $w$.

Thus, an affirmative answer to Question \ref{Q2} would allow us to extend
Deligne's theory of absolute Hodge classes on abelian varieties to
characteristic $p$. As Tate once wrote in a similar
context,\footnote{\cite{tate1965}, p.~107. Tate's question was answered
(negatively) by Mumford.} \textquotedblleft we have a completely down-to-earth
question which could be explained to a bright freshman and which should be
settled one way of the other.\textquotedblright
\end{plain}

\subsubsection{A variant inductive argument}

\begin{plain}
\label{x30}Of course, there are variations of the above argument, for example,
by dropping CM in the definition of liftable.
\end{plain}

\subsubsection{Notes}

\begin{plain}
\label{x20}Let $(A,\lambda)$ be a polarized abelian variety over a field $k.$
Then%
\[
\mathrm{NS}^{0}(A)\simeq\{\alpha\in\End^{0}(A)\mid\alpha^{\dagger}=\alpha\}
\]
(\cite{mumford1970}, p.~208). Thus, Questions \ref{Q2} can be restated in
terms of symmetric endomorphisms, or even in terms of the subalgebras they
generate. Not all subfields of endomorphism algebras of abelian varieties can
be lifted to characteristic zero. For example, a subfield $E$ of $\End^{0}(A)$
such that $[E:\mathbb{Q}]=2\dim A$ must be CM in characteristic zero, but need
not be so in characteristic $p$. In particular, a real quadratic subfield of
the endomorphism algebra of an elliptic curve does not lift to characteristic
zero. However, with some obvious restrictions, every endomorphism lifts (up to
isogeny). See, for example, \cite{zink1983}, 2.7.
\end{plain}

\begin{plain}
\label{x23}It may be possible to use the natural Jordan algebra structure on
$\mathrm{NS}^{0}(A)$ (\cite{mumford1970}, p.~208).
\end{plain}

\begin{plain}
\label{x31}Let $A$ be an abelian variety over $\mathbb{F}$. If $A_{1}$ is a
model of $A$ over a finite subfield $k$ of $\mathbb{F}$ such that
$\End_{k}^{0}(A_{1})=\End_{\mathbb{F}{}}^{0}(A)$, then we let $\mathbb{Q}%
\{\pi\}$ denote the $\mathbb{Q}$-subalgebra of $\End_{\mathbb{F}}^{0}(A)$
generated by the Frobenius endomorphism of $A_{1}$ (relative to $k$). It is
independent of the choice of the model.
\end{plain}

\begin{theorem}
\label{x24}Let $A_{0}$ be a simple abelian variety over $\mathbb{F}{}$, and
let $L$ be a CM subfield of $\End^{0}(A)$ such that

\begin{enumerate}
\item $L$ contains $\mathbb{Q}{}\{\pi\}$,

\item $L$ splits $\End^{0}(A)$, and

\item $[L\colon\mathbb{Q}{}]=2\dim A$.
\end{enumerate}

\noindent Then, up to isogeny, $A_{0}$ lifts to an abelian variety $A$ in
characteristic zero such that $L\subset\End^{0}(A).$
\end{theorem}

\begin{proof}
See \cite{tate1968}, Thm 2.
\end{proof}

\begin{theorem}
[?]\label{x25}Let $(A,\nu,\lambda)$ be a Weil triple over $\mathbb{Q}%
{}^{\mathrm{al}}$ with respect to $E$. Assume that the degree of $\lambda$ is
prime to $p$ and that $p$ is unramified in $E$. Let $R\subset\End_{E}^{0}(A)$
be a product of CM fields respecting the polarization and of degree $2\dim A$
over $\mathbb{Q}{}$. There exists a Weil triple $(A^{\prime},\nu^{\prime
},\lambda^{\prime})$ over $\mathbb{Q}{}^{\mathrm{al}}$ in the same family as
$(A,\nu,\lambda)$ equipped with an action of $R$ such that $(A^{\prime}%
,\nu^{\prime},\lambda^{\prime},R)_{0}$ is isogenous to $(A,\nu,\lambda,R)_{0}$.
\end{theorem}

\begin{proof}
Compare \cite{zink1983}, especially Theorem 2.7.
\end{proof}

I expect that it will be possible to use Theorem \ref{x19} to prove Conjecture
A for abelian varieties with CM by a field $E$ unramified over $p$. The
general case is less certain.

\subsection{A variational approach to proving Conjecture A for Weil classes}

As explained in the introduction, the key to proving the conjectures is to
prove that split Weil classes are $w$-rational. Not having a proof, here we
simply make some remarks.

Let $(A_{1},\nu_{1},\lambda_{1})$ be a split Weil triple over $\mathbb{Q}%
^{\mathrm{al}}$ (relative to a CM algebra $E$) with good reduction at $w.$

\begin{plain}
\label{w10}If $(A_{1},\nu_{1},\lambda_{1})$ is of the form $(B\otimes
\mathcal{O}{}_{E},\ldots)$, then the Weil classes on $A_{1}$ are $w$-rational
because they are algebraic (\cite{deligne1982}, 4.5).
\end{plain}

\begin{plain}
\label{w12}Consider the Weil family $f\colon A\rightarrow S$ over
$\mathbb{Q}{}^{\mathrm{al}}$ containing $(A_{1},\nu_{1},\lambda_{1})$ (see
\ref{w2}). If $(A_{1},\nu_{1},\lambda_{1})_{0}$ lifts (up to isogeny) to a
triple in the family whose Weil classes are $w$-rational, then the Weil
classes on $A_{1}$ are $w$-rational (obviously).
\end{plain}

\begin{plain}
\label{w13}Let $\delta_{1}$ be a Lefschetz class on $(A_{1})_{0}$. If $B$ and
the family $\bar{f}\colon\bar{S}\rightarrow C$ in \ref{w7} can be chosen so
that $\delta_{1}$ extends to a section $\delta$ of $R^{2d^{\prime}}\bar
{f}_{\ast}\mathbb{A}{}(d^{\prime})$ satisfying the hypothesis of \ref{c13} (or
\ref{g13c}), then, for all global sections $\gamma$ of $W_{E}(A/S),$%
\[
\langle(\gamma_{1})_{0}\cdot\delta_{1}\rangle=\langle(\gamma_{2})_{0}%
\cdot\delta_{2}\rangle,
\]
which lies in $\mathbb{Q}{}$ because $\gamma_{2}$ is algebraic (\ref{w10}) and
$\delta_{2}$ is Lefschetz (\ref{c13}).
\end{plain}

\section{Conjecture B}

\label{CB}

\begin{quotation}
Every CM abelian variety $A$ over $\mathbb{Q}^{\mathrm{al}}$ has good
reduction at $w$ to an abelian variety $A_{0}$ over $\mathbb{F}$. The Hodge
classes on $A$ define a $\mathbb{Q}$-structure on the part of $H_{\mathbb{A}%
}^{2\ast}(A)(\ast)$ fixed by the Mumford--Tate group of $A$, and the Lefschetz
classes on $A_{0}$ define a $\mathbb{Q}$-structure on the part of
$H_{\mathbb{A}}^{2\ast}(A_{0})(\ast)$ fixed by the Lefschetz group of $A_{0}$.
The weak rationality conjecture says that the two structures are compatible.
\end{quotation}

In this section, we state the weak rationality conjecture (Conjecture B),
which is a variant of Conjecture A, and we suggest an variational proof of it
for CM abelian varieties.

\subsection{Statement of the conjecture}

Let $A$ be an abelian variety over $\mathbb{Q}^{\mathrm{al}}$ with good
reduction at $w$. Let $\gamma$ be a Hodge class in $H_{\mathbb{A}{}}^{2\ast
}(A)(\ast)$ and $\gamma_{0}$ its image in $H_{\mathbb{A}{}}^{2\ast}%
(A_{0})(\ast)$.

\begin{definition}
\label{c1} We say that $\gamma$ is $w$\emph{-Lefschetz} if $\gamma_{0}$ is
Lefschetz and \emph{weakly }$w$\emph{-Lefschetz} if $\gamma_{0}$ is weakly Lefschetz.
\end{definition}

Thus $\gamma$ is $w$-Lefschetz (resp.~weakly $w$-Lefschetz) if $\gamma_{0}$ is
in the $\mathbb{Q}{}$-algebra (resp.~$\mathbb{A}{}$-algebra) generated by the
divisor classes.

\begin{conjectureb}
\label{c3} Let $A$ be an abelian variety over $\mathbb{Q}{}^{\mathrm{al}}$
with good reduction at $w$, and let $\gamma$ be a Hodge class on $A$. If
$\gamma$ is weakly $w$-Lefschetz, then it is $w$-Lefschetz.
\end{conjectureb}

\begin{proposition}
\label{c4} Let $A$ be an abelian variety over $\mathbb{Q}{}^{\mathrm{al}}$
with good reduction at $w$, and let $\gamma$ be a Hodge class on $A$. If
$\gamma$ is $w$-rational (e.g., algebraic) and weakly $w$-Lefschetz, then it
is $w$-Lefschetz.
\end{proposition}

\begin{proof}
Let $\gamma$ be a $w$-rational Hodge class of codimension $r$ on $A$. Choose a
$\mathbb{Q}{}$-basis $e_{1},\ldots,e_{t}$ for the space of Lefschetz classes
of codimension $r$ on $A_{0}$, and let $f_{1},\ldots,f_{t}$ be the dual basis
for the space of Lefschetz classes of complementary dimension (here we use
\ref{l2}). If $\gamma$ is weakly $w$-Lefschetz, then $\gamma_{0}=\sum
c_{i}e_{i}$ with $c_{i}\in\mathbb{A}{}$. Now%
\[
\langle\gamma_{0}\cup f_{j}\rangle=c_{j},
\]
and $c_{j}$ lies in $\mathbb{Q}$ because $\gamma$ is $w$-rational.
\end{proof}

\begin{corollary}
\label{c4b}If Conjecture A holds for $A$, then so does Conjecture B.
\end{corollary}

\begin{proof}
Conjecture A says that all the Hodge classes on $A$ are $w$-rational
\end{proof}

If Conjecture B holds for all CM abelian varieties over $\mathbb{Q}%
{}^{\mathrm{al}}$, then so does Conjecture A (see \S 4).

\subsection{Homomorphisms in families}

In this subsection, we study how homomorphisms of abelian varieties behave in families.

We shall need one trivial lemma and two theorems.

\begin{lemma}
\label{c6}Let $Q$ be a field and $R$ a $Q$-algebra. Consider a commutative
diagram of linear maps%
\[
\renewcommand{\arraystretch}{1.3}\begin{tikzcd}
W\arrow{r}\arrow{d}{a}&W^{\prime}\arrow{d}{b}\\
V\arrow{r}&V^{\prime}%
\end{tikzcd}\quad%
\begin{array}
[c]{l}%
W,V\,\,\mathbb{Q}\text{-vector spaces}\\
W^{\prime},V^{\prime}\,\,R\text{-modules},
\end{array}
\]
If either $a$ or $b$ is injective and the horizontal arrows are such that%
\begin{equation}
W\otimes_{\mathcal{Q}}R\overset{\lsimeq}{\longrightarrow}W^{\prime},\quad
V\otimes_{\mathcal{Q}}R\hookrightarrow V^{\prime}, \tag{**}%
\end{equation}
then both $a$ and $b$ are injective, and
\[
W=V\cap W^{\prime}\qquad\text{(intersection in }V^{\prime}).
\]

\end{lemma}

\begin{proof}
If $b$ is injective, then $a$ is injective because $W\rightarrow W^{\prime
}\xrightarrow{b}V^{\prime}$ is injective. Consider the diagram%
\[
\begin{tikzcd}
W\arrow{r}\arrow{d}{a}
&W\otimes_Q R\arrow{r}{\simeq}\arrow{d}{a\otimes 1}
&W^{\prime}\arrow{d}{b}\\
V\arrow{r}&V\otimes_Q R\arrow[hook]{r}&V^{\prime}%
\end{tikzcd}
\]
If $a$ is injective, then $b$ is injective because $W\otimes
R\xrightarrow{a\otimes1}V\otimes R\rightarrow V^{\prime}$ is injective and
$W\otimes R\rightarrow W^{\prime}$ is surjective. For the remaining statement,
we may replace $W^{\prime}$ and $V^{\prime}$ with $W\otimes R$ and $V\otimes
R$. Let $V=W\oplus U$, and let $v=w+u\in V$. Then
\[
v\otimes1=w\otimes1+u\otimes1\in(W\otimes R)\oplus(U\otimes R).
\]
If $v\otimes1\in W\otimes R$, then $u=0$ and so $v=w\in W$.
\end{proof}

\begin{theorem}
\label{c11}Let $A$ and $B$ be abelian schemes over a connected noetherian
normal scheme $S$. Every homomorphism $A_{\eta}\rightarrow B_{\eta}$ of the
generic fibres extends uniquely to a homomorphism $A\rightarrow B$ over $S$.
\end{theorem}

\begin{proof}
When $\dim(S)=1$, $B$ is the N\'{e}ron model of $B_{\eta}$, so this follows
from the universal property of such models. For a proof that the general case
follows from this case, see \cite{chaiF1990}, I, Proposition 2.7.
\end{proof}

\begin{theorem}
[Tate, de Jong]\label{c7} Let $G$ and $H$ be $p$-divisible groups over a
connected noetherian normal scheme $S$. Every homomorphism $G_{\eta
}\rightarrow H_{\eta}$ of the generic fibres extends uniquely to a
homomorphism $G\rightarrow H$ over $S$.
\end{theorem}

\begin{proof}
Let $\eta=\Spec K$. When $K$ has characteristic zero, this is Theorem 4 of
\cite{tate1967}, and when $K$ has characteristic $p\neq0$, it is Theorem 2 of
\cite{dejong1998}.
\end{proof}

For an abelian scheme $A$ over a scheme $S$ and integer $n>0$, we let
\[
A_{n}=\Ker(n\colon A\rightarrow A).
\]
This is a finite flat group scheme over $S$, and we let $TA$ denote the
projective system $(A_{n})_{n}$. Then $A\rightsquigarrow TA$ is a faithful
functor, compatible with base change.

\begin{theorem}
\label{c8}Let $A$ and $B$ be abelian schemes over a connected normal scheme
$S$ of finite type over a field $k$, and let $u\colon TA\rightarrow TB$ be a
homomorphism. If there exists a closed point $s\in S$ such that $u_{s}\colon
TA_{s}\rightarrow TB_{s}$ equals $Tw$ for some $w\colon A_{s}\rightarrow
B_{s}$, then there exists an integer $n>0$ and a homomorphism $v\colon
A\rightarrow B$ such that $Tv=n\,u$.
\end{theorem}

\begin{proof}
Grothendieck (1966, p.~60)\nocite{grothendieck1966} states this as a
conjecture, but remarks that it is a consequence of the Tate conjecture. We
explain how. In proving the theorem, we may suppose that the field $k$ is a
finitely generated field (ibid.\ 2.2).\footnote{Alternatively, replace
$\Hom(TA,TB)$ etc.\ with $\dlim\Hom(TA^{\prime},TB^{\prime}))$, where the
limit runs over the models of $A\to S$, $B\to S$ over finitely generated
subfields of k).} Consider the diagram%
\[
\begin{tikzcd}
\Hom(A_{\eta},B_{\eta})\arrow{r}&\Hom(TA_{\eta},TB_{\eta})\\
\Hom(A,B)\arrow{r}\arrow{u}[swap]{\simeq}\arrow{d}{a}
&\Hom(TA,TB)\arrow{u}[swap]{\simeq}\arrow{d}{b}\\
\Hom(A_{s},B_{s})\arrow{r}&\Hom(TA_{s},TB_{s}).
\end{tikzcd}
\]
The restriction maps%
\begin{align*}
\Hom(A,B)  &  \rightarrow\Hom(A_{\eta},B_{\eta})\\
\Hom(TA,TB)  &  \rightarrow\Hom(TA_{\eta},TB_{\eta})
\end{align*}
are bijective by Theorems \ref{c11} and \ref{c7}. The top and bottom
horizontal maps induce isomorphisms%
\begin{align}
\Hom^{0}(A_{\eta},B_{\eta})\otimes_{\mathbb{Q}}\mathbb{A}  &  \rightarrow
\Hom(TA_{\eta},TB_{\eta})_{\mathbb{Q}{}}\\
\Hom^{0}(A_{s},B_{s})\otimes_{\mathbb{Q}{}}\mathbb{A}{}  &  \rightarrow
\Hom(TA_{s},TB_{s})_{\mathbb{Q}{}}.
\end{align}
by the Tate conjecture (proved in this case by Tate, Zarhin, and Faltings).
The map $a$ is injective because $\Hom(T_{l}A,T_{l}B)\rightarrow
\Hom(T_{l}A_{s},T_{l}B_{s})$ is obviously injective. On applying Lemma
\ref{c6} to the bottom square, we find that%
\[
\Hom^{0}(A,B)=\Hom^{0}(A_{s},B_{s})\cap\Hom(TA,TB)_{\mathbb{Q}}%
\]
(intersection inside $\Hom(TA_{s},TB_{s})_{\mathbb{Q}}$) as required.
\end{proof}

\begin{corollary}
\label{c18}With the notation of the theorem,%
\[
\Hom^{0}(A,B)=\Hom^{0}(A_{s},B_{s})\cap\Hom(TA_{\eta},TB_{\eta})\otimes
\mathbb{Q}{}%
\]
(intersection inside $\Hom(TA_{s},TB_{s})\otimes\mathbb{Q}{}$).
\end{corollary}

\begin{remark}
\label{c12}A d\'{e}vissage (\cite{grothendieck1966}, 1.2) shows that the
Theorem \ref{c8} is true over any reduced connected scheme $S$, locally of
finite type over $\Spec(\mathbb{Z})$ or a field.
\end{remark}

\subsection{Divisor classes in families}

Let $A$ be an abelian variety. We say that an element of $H_{\mathbb{A}{}%
}^{2n}(A)(n)$ is \emph{algebraic} if it is in the $\mathbb{Q}$-span of the
classes of algebraic cycles, i.e., if it is in the image of the cycle class
map $\mathrm{CH}^{n}(A)\otimes\mathbb{Q}\rightarrow H_{\mathbb{A}}^{2n}(A)(n)$.

\begin{theorem}
\label{c9}Let $f\colon A\rightarrow S$ be an abelian scheme over a connected
normal scheme $S$ of finite type over a field $k$, and let $\gamma$ be a
global section of $R^{2}f_{\ast}\mathbb{A}(1)$. If $\gamma_{s}\in H^{2}%
(A_{s},\mathbb{A}(1))$ is algebraic for one closed $s\in S$, then it is
algebraic for all closed $s\in S$.
\end{theorem}

As Grothendieck (1966, p.~66)\nocite{grothendieck1966} notes, because of the
correspondence between endomorphisms of abelian varieties and divisor classes,
this is essentially equivalent to Theorem \ref{c8}.

We explain how to prove Theorem \ref{c9}. Recall that $M\overset{\otimes
}{\longrightarrow}N$ means that $M$ is a $\mathbb{Q}$-structure on $N$, i.e.,
that $M\otimes_{\mathbb{Q}}\mathbb{A}\simeq N$.

For an abelian scheme $A$ over $S$, we let%
\[
\Pic(A/S)\overset{\df}{=}\Pic_{A/S}(S)=\frac{\Pic(A)}{\Pic(S)}.
\]
Recall that, for abelian schemes $A$ and $B$ over a scheme $S,$%
\begin{align*}
\mathrm{DC}_{S}(A,B)  &  \overset{\df}{=}\frac{\Pic(A\times_{S}B/S)}%
{\pr_{1}^{\ast}\Pic(A/S)+\pr_{2}^{\ast}\Pic(B/S)}\\
\mathrm{DC}_{S}(A,B)  &  \simeq\Hom_{S}(A,B^{\vee}),
\end{align*}
and that the map $\mu^{\ast}-\pr_{1}^{\ast}-\pr_{2}^{\ast}\colon
\Pic(A/S)\rightarrow\Pic(A\times_{S}A/S)$ factors through an injection%
\[
\mathrm{NS}(A/S)\hookrightarrow\Pic(A\times_{S}A/S).
\]
Consider the diagram%
\[
\begin{tikzcd}[column sep=huge]
\mathrm{NS}(A_{\eta})
\arrow[shift left=0.6ex]{r}{\mu^\ast-\pr_1^\ast-\pr_2^\ast}
&\mathrm{DC}(A_{\eta},A_{\eta})\arrow[shift left=0.6ex]{l}{\Delta^\ast}
\simeq\Hom(A_{\eta},A_{\eta}^{\vee})\\
\mathrm{NS}(A/S)\arrow{d}\arrow{u}\arrow[shift left=0.6ex]{r}{\mu^\ast-\pr_1^\ast-\pr_2^\ast}
&{\mathrm{DC}_{S}(A,A)}\arrow[shift left=0.6ex]{l}{\Delta^\ast}
\simeq\Hom_{S}(A,A^{\vee})\arrow[hook,shift left=6ex]{d}
\arrow[shift right=6ex]{u}[swap]{\simeq}\\
\mathrm{NS}(A_{s})\arrow[shift left=0.6ex]{r}{\mu^\ast-\pr_1^\ast-\pr_2^\ast}
&\mathrm{DC}(A_{s},A_{s})\arrow[shift left=0.6ex]{l}{\Delta^\ast}
\simeq\Hom(A_{s},A_{s}^{\vee}).
\end{tikzcd}
\]
The composite of each map $\mu^{\ast}-\pr_{1}^{\ast}-\pr_{2}^{\ast}$ with
$\Delta^{\ast}$ is multiplication by $2$. Therefore, after tensoring with
$\mathbb{Q}$, we get the left hand side of the following diagram,
\[
\begin{tikzcd}
\mathrm{NS}^0(A_{\eta})\arrow{r}{\otimes}
&H^{2}_{\mathbb{A}}(A_{\bar{\eta}})(1)^{\pi_{1}(\eta)}\\
\mathrm{NS}^0(A/S)\arrow{r}{d}\arrow{u}{\simeq}\arrow[hook]{d}
&H^{0}(S,R^{2}f_{\ast}\mathbb{A}(1))\arrow{u}{e}\arrow[hook]{d}\arrow{r}
{\simeq}&H^{2}_{\mathbb{A}}(A_{\bar{\eta}})(1)^{\pi_{1}(S)}\\
\mathrm{NS}^0(A_{s})\arrow{r}{\otimes} &H^{2}_{\mathbb{A}}(A_{\bar{s}})(1)^{\pi_{1}(s)}.
\end{tikzcd}
\]
As in the previous case, in proving the theorem, we may suppose that the field
$k$ is finitely generated, which allows us to apply the Tate conjecture (known
in this case) to the top and bottom rows of the diagram. From the diagram, we
see that $d$ is injective. The map $e$ is injective, and it follows from the
diagram that it is an isomorphism. Hence the map $d$ induces an isomorphism%
\[
\mathrm{NS}^{0}(A/S)\otimes\mathbb{A}\longrightarrow H^{0}(S,R^{2}f_{\ast
}\mathbb{A}{}(1)).
\]
On applying Lemma \ref{c6} to the bottom square, we obtain the theorem.

\begin{corollary}
\label{c17}With the notation of the theorem, for any closed point $s$ of $S$,%
\[
\mathrm{NS^{0}}(A/S)=\mathrm{NS^{0}}(A_{s})\cap H^{0}(S,R^{2}f_{\ast
}\mathbb{A}{}(1))
\]
(intersection inside $H_{\mathbb{A}}^{2}(A_{s})(1))$.
\end{corollary}

\begin{remark}
\label{d13}When $f\colon A\rightarrow S$ is an abelian scheme over a connected
normal scheme $S$, we define $\mathcal{D}^{1}(A/S)$ to be the image of
$\mathrm{NS^{0}}(A/S)$ in $H^{0}(S,R^{2}f_{\ast}\mathbb{A}{}(1))$. In the
above proof, when $S$ is of finite type over a field, we obtained a diagram%
\[
\begin{tikzcd}
\mathcal{D}^{1}(A/S)\arrow{r}{\otimes}\arrow[hook]{d}
&H^{0}(S,R^{2}f_{\ast}\mathbb{A}{}(1))\arrow[hook]{d}\\
\mathcal{D}^{1}(A_{s})\arrow{r}{\otimes}
&H_{\mathbb{A}}^{2}(A_{s}(1))^{\pi_{1}(s)}.
\end{tikzcd}
\]

\end{remark}

\begin{aside}
\label{c10}See also Conjecture 1.4 of Grothendieck 1966 and Theorem 0.2 (=
Theorem 1.4) of Morrow 2019.
\end{aside}

\subsection{Algebraic classes in families}

For an abelian variety $A$, we let $\mathcal{A}^{\ast}(A)$ denote the
$\mathbb{Q}{}$-algebra of algebraic classes in $H_{\mathbb{A}{}}^{2\ast
}(A)(n)$.

Let $f\colon A\rightarrow S$ be an abelian scheme over a connected normal
scheme $S$ over a finite field $k$. Let $s$ be a closed point of $S$ and
$\eta$ the generic point. Consider the diagram%
\[
\begin{tikzcd}
\mathcal{A}{}^{n}(A_{\eta})\arrow{r}{\otimes}
&H_{\mathbb{A}}^{2n}(A_{\bar{\eta}})(n)^{\pi_{1}(\eta)}\\
\mathcal{A}^{n}(A/S)\arrow{r}{d}\arrow{u}{c}\arrow{d}
&H^{0}(S,R^{2n}f_{\ast}\mathbb{A}(n))\arrow{r}{\simeq}\arrow{u}{e}\arrow[hook]{d}
&H_{\mathbb{A}}^{2n}(A_{\bar{\eta}})(n)^{\pi_{1}(S)}\\
\mathcal{A}{}^{n}(A_{s})\arrow{r}{\otimes}&H_{\mathbb{A}}^{2n}(A_{s})^{\pi_{1}(s)}%
\end{tikzcd}
\]
The maps $d$ and $e$ are injective, so $c$ is injective.

\begin{theorem}
\label{d14}Assume that $\mathcal{A}{}^{n}(A/S)\rightarrow\mathcal{A}{}%
^{n}(A_{\eta})$ is surjective and that the Tate conjectures holds for
algebraic cycles of codimension $n$ on $A_{\eta}$ and $A_{s}$.Then%
\[
\mathcal{A}{}^{n}(A/S)=\mathcal{A}{}^{n}(A_{s})\cap H^{0}(S,R^{2n}f_{\ast
}\mathbb{A}{}(n))
\]
(intersection inside $H_{\mathbb{A}{}}^{2n}(A_{s})^{\pi_{1}(s)}$).
\end{theorem}

\begin{proof}
Under the assumptions, $c$ and $e$ are isomorphisms, so $d$ becomes an
isomorphism when $\mathcal{A}{}^{n}(A/S)$ is tensored with $\mathbb{A}{}$. Now
apply Lemma \ref{c6} to the lower square.
\end{proof}

\begin{corollary}
\label{d15}With the assumptions of the theorem, if $\gamma\in H^{0}%
(S,R^{2n}f_{\ast}\mathbb{A}{}(n))$ is algebraic for one closed $s$, then it is
algebraic for all closed $s.$
\end{corollary}

\begin{proof}
If $\gamma$ is algebraic for one $s$, then the theorem shows that it lies in
$\mathcal{A}{}^{n}(A/S)$, and hence its image in $H_{\mathbb{A}{}}^{2n}%
(A_{s})$ lies in $\mathcal{A}{}^{n}(A_{s})$ for all $s$.
\end{proof}

\subsubsection{Notes}

\begin{plain}
\label{d17}This section is only of heuristic significance. We certainly do not
want to assume the Tate conjecture.
\end{plain}

\subsection{Weakly Lefschetz classes in families}

\begin{plain}
\label{d19}Let $G$ be a group (abstract, profinite, algebraic, \ldots) acting
on a finite-dimensional vector space $V$ over a field $k$ of characteristic
$0$. The $k$-algebra $\left(  \bigotimes^{\ast}V\right)  ^{G}$ is generated by
$G$-invariant tensors of degree $2$ in each of the following cases:

\begin{enumerate}
\item $G=\mathrm{Sp}(\phi)$ with $\phi$ a nondegenerate skew-symmetric form on
$V;$

\item $G=O(\phi)$ with $\phi$ a nondegenerate symmetric form on $V;$

\item $G=\GL(W)$ and $V=W\oplus W^{\vee}$;

\item $T$ is a torus and the weights $\xi_{1},\ldots,\xi_{2m}$ of $T$ on $V$
can be numbered in such a way that the $\mathbb{Z}{}$-module of relations
among the $\xi_{i}$ is generated by the relations $\xi_{i}+\xi_{i+1}=0,\quad
i=1,\ldots,m$.
\end{enumerate}

\noindent See \cite{milne1999a}, 3.6, 3.8.
\end{plain}

\begin{plain}
\label{d20}Let $G$ be a group acting on a finite-dimensional vector space $V$
over a field $k$. If the $k$-algebra $\left(  \bigotimes^{\ast}V\right)  ^{G}$
is generated by $G$-invariant tensors of degree $2$, then the same is true of
$\left(  \bigwedge^{\ast}V\right)  ^{G}$.
\end{plain}

\begin{question}
\label{d21}Let $f\colon A\rightarrow S$ be an abelian scheme over a connected
normal scheme of finite type over an algebraically closed field $k$, and let
$\gamma$ be a global section of $R^{2n}f_{\ast}\mathbb{\mathbb{Q}{}}_{\ell}%
{}(n)$. If $\gamma_{s}$ is weakly Lefschetz for one closed $s\in S$, then is
it weakly Lefschetz for all closed $s\in S$?

After replacing $k$ with a finitely generated subfield, we have a diagram%
\[
\begin{tikzcd}
\mathcal{D}^{1}(A/S)\arrow{d}\arrow{r}&\mathcal{D}^{1}(A/S)\otimes{\mathbb{Q}_\ell}\arrow{r}{\simeq}\arrow[hook]{d}
&H^{0}(S,R^{2}f_{\ast}{\mathbb{Q}_\ell}(1))\arrow[hook]{d}\\
\mathcal{D}^{1}(A_{s})\arrow{r}&\mathcal{D}^{1}(A_{s})\otimes{\mathbb{Q}_\ell}\arrow{r}{\simeq}
&H^2(A_{s},\mathbb{Q}_\ell(1))^{\pi_{1}(s)}
\end{tikzcd}
\]
(see the proof of \ref{c9}), which shows that the answer is positive for
$n=1$. Let $V=H^{1}(A_{\bar{\eta}},\mathbb{Q}{}_{\ell})$. The question comes
down the following. Let $\gamma\in H^{0}(S,R^{2n}f_{\ast}\mathbb{Q}{}_{\ell
}(n))=(\bigwedge^{2n}V)(n)^{\pi_{1}(S)}$. Suppose that, when regarded as an
element of $(\bigwedge^{2n}V)(n)^{\pi_{1}(s)}$, $\gamma$ lies in the
$\mathbb{Q}{}_{\ell}$-algebra generated by $\bigwedge^{2}V(1)^{\pi_{1}(s)}$.
Does this imply that $\gamma$ lies in the $\mathbb{Q}{}_{\ell}$-algebra
generated by $\bigwedge^{2}V(1)^{\pi_{1}(S)}$? The answer is surely negative
in general, but there may be useful conditions on $S$ that ensure that the
answer is positive.
\end{question}

\subsection{Lefschetz groups in families}

\begin{plain}
\label{d18}Let $A$ be an abelian variety over a separably closed field $k$,
and let $\ell$ be a prime number $\neq\mathrm{char}(k)$. Let $C_{\ell}(A)$
denote the centralizer of $\End^{0}(A)$ in $\End_{\mathbb{Q}{}_{\ell}}%
(V_{\ell}A)$. This is a semisimple algebra over $\mathbb{Q}{}_{\ell}$ with an
involution $\dagger$ (defined by any Rosati involution). The Lefschetz group
of $A$ (relative to $H_{\ell}$) is the algebraic group $L(A)$ over
$\mathbb{Q}{}_{\ell}$ with%
\[
L(A)(\mathbb{Q}{}_{\ell})=\{a\in C_{\ell}(A)\mid a^{\dagger}a\in\mathbb{Q}%
{}_{\ell}^{\times}\}
\]
(see \cite{milne1999b}, 4.4). When $k=\mathbb{F}{}$,%
\[
C_{\ell}(A)=C_{0}(A)\otimes_{\mathbb{Q}{}}\mathbb{Q}{}_{\ell},
\]
where $C_{0}(A)$ is the centre of $\End^{0}(A)$, so $C_{0}(A)=\mathbb{Q}%
{}\{\pi\}$.
\end{plain}

\begin{definition}
\label{c10a}Let $f\colon A\rightarrow S$ be an abelian scheme over a connected
normal scheme of finite type over an algebraically closed field $k$, and let
$s\in S(k)$. When we identify $V_{\ell}(A_{s})$ with $V_{\ell}(A_{\eta})$, we
have $C_{\ell}(A_{s})\subset C_{\ell}(A_{\eta})$. We say that $f$ is
\emph{general} if the $\mathbb{Q}_{\ell}$-algebras $C_{\ell}(A_{s})$, $s\in
S(k)$, generate $C_{\ell}(A_{\bar{\eta}})$.
\end{definition}

\begin{example}
\label{c10c}Let $k=\mathbb{F}{}$. For $s\in S(k)$, we have%
\[
\begin{tikzcd}
\End^{0}(A_{s})\arrow[hookleftarrow]{r}&\End^{0}(A/S)\arrow{r}{\simeq}
&\End^{0}(A_{\eta}),
\end{tikzcd}
\]
so $C_{0}(A_{s})\subset C_{\ell}(A_{\eta})$. So $f$ is general if the
$\mathbb{Q}{}$-algebras $C_{0}(A_{s})$ generate the centralizer of
$\End^{0}(A_{\bar{\eta}})$ in $\End_{\mathbb{Q}{}_{\ell}}(V_{\ell}(A_{\eta}))$.
\end{example}

\begin{example}
\label{c10d}Let $f$ be the universal elliptic curve over the affine line
$(k=\mathbb{F}{}$). Then $C_{\ell}(A_{\bar{\eta}})=\End(V_{\ell}A_{\bar{\eta}%
})\approx M_{2}(\mathbb{Q}{}_{\ell})$. On the other hand, $C_{0}(A_{s})$ is
either $\mathbb{Q}{}$ or $F$, where $F$ is an imaginary quadratic number
field, and all quadratic imaginary number fields occur. Therefore, $f$ is general.
\end{example}

\begin{example}
\label{c10e}Let $f\colon A\rightarrow S$ be a constant abelian variety
($k=\mathbb{F}{})$. Then $\End(A_{s})=\End(A_{\eta})$ for all closed $s\in S$.
Therefore $C(A_{s})=C(A_{\eta})$ for all closed $s$, and $f$ is again general.
\end{example}

\begin{proposition}
\label{c10b}If $f\colon A\rightarrow S$ is general, then the algebraic group
$L(A_{\bar{\eta}})$ is generated by its subgroups $L(A_{s})$, $s\in S(k)$.
\end{proposition}

\begin{proof}
This follows from the above description on $L(A)$.
\end{proof}

\begin{proposition}
\label{c10g}Let $S$ be a smooth projective curve over $\mathbb{F}{}$ and
$f\colon A\rightarrow S$ an abelian scheme such that the $k(\eta)/k$-trace of
$A_{\eta}$ is zero. Then $f$ is general.
\end{proposition}

\begin{proof}
This follows from applying a Chebotarev density theorem to a model of $f$ over
a finite subfield of $\mathbb{F}$.
\end{proof}

\begin{question}
\label{c10f}Are all abelian schemes general?
\end{question}

\begin{remark}
\label{c10h}There is an extensive literature on Mumford--Tate groups and their
variation in families (see \cite{milne2013b}, \S 6, for a summary), much of
which carries over to Lefschetz groups.
\end{remark}

\subsection{Lefschetz classes in families}

\begin{theorem}
\label{c13} Let $f\colon A\rightarrow S$ be an abelian scheme over a connected
normal scheme of finite type over an algebraically closed field $k$, and let
$\delta$ be a global section of $R^{2n}f_{\ast}\mathbb{A}(n)$ such that
$\delta$ is fixed by $L(A_{\bar{\eta}})$. If $\delta_{s}$ is Lefschetz for one
closed $s\in S$, then it is Lefschetz for all closed $s\in S$.
\end{theorem}

\begin{proof}
Let $\mathcal{D}^{\ast}(A/S)$ denote the $\mathbb{Q}$-subalgebra of
\[
H^{0}(S,R^{2\ast}f_{\ast}\mathbb{A}(\ast))\simeq H_{\mathbb{A}}^{2n}%
(A_{\bar{\eta}})(n)^{\pi_{1}(S)}%
\]
generated by the image of $\mathrm{NS}(A/S)$. Consider the diagram%
\[
\begin{tikzcd}
\mathcal{D}^{n}(A_{\eta})\arrow{r}{\otimes}
&H^{2n}(A_{\eta},\mathbb{A}(n))^{L(A_\eta)}\\
\mathcal{D}^{n}(A/S)\arrow{u}{a}\arrow{r}{b}\arrow[hook]{d}
&H^{0}(S,R^{2n}f_{\ast}\mathbb{A}(n))^{L(A_{\overline\eta})}\arrow{u}{c}\arrow[hook]{d}\\
\mathcal{D}^{n}(A_{s})\arrow{r}{\otimes}
&H^{2n}(A_{s},\mathbb{A}{}(n))^{L(A_{s})}.%
\end{tikzcd}
\]
For the top and bottom arrows, see \ref{l5}. The map $a$ is surjective when
$n=1$ (see the proof of Theorem \ref{c9}), and so it is surjective for all
$n$. The maps $b$ and $c$ are injective, from which it follows that $a$ and
$c$ are isomorphisms and that $b$ induces an isomorphism
\[
\mathcal{D}^{n}(A/S)\otimes\mathbb{A}\rightarrow H^{0}(S,R^{2n}f_{\ast
}\mathbb{A}(n))^{L(A_{\overline{\eta}})}.
\]
On applying Lemma \ref{c6} to the bottom square, we find that%
\[
\mathcal{D}^{n}(A/S)=\mathcal{D}{}^{n}(A_{s})\cap H^{0}(S,R^{2n}f_{\ast
}\mathbb{A}(n))^{L(A_{\overline{\eta}})}.
\]
If the element $\delta$ is such that $\delta_{s}\in\mathcal{D}^{n}(A_{s})$ for
one $s$, then it lies in $\mathcal{D}{}^{n}(A/S)$, and so $\delta_{s}%
\in\mathcal{D}^{n}(A_{s})$ for all $s.${}
\end{proof}

\begin{corollary}
\label{g13c}Let $f\colon A\rightarrow S$ be as in the statement of the
theorem. Assume that $f$ is general, and let $\delta$ be a global section of
$R^{2n}f_{\ast}\mathbb{A}(n)$. If $\delta$ is weakly Lefschetz for all closed
$s\in S$ and Lefschetz for one closed $s\in S$, then it is Lefschetz for all
closed $s.$
\end{corollary}

\begin{proof}
The hypotheses imply that $\delta$ is fixed by $L(A_{\bar{\eta}})$.
\end{proof}

\subsection{A variational approach to proving Conjecture B for Weil classes.}

The key to proving the conjectures is to prove that Conjecture B holds for
split Weil classes. Not having a proof, we simply make some remarks.

Let $(A_{1},\nu_{1},\lambda_{1})$ be a split Weil triple over $\mathbb{Q}%
{}^{\mathrm{al}}$ (relative to a CM-algebra $E$) with good reduction at $w$.

\begin{plain}
\label{w14}If $(A_{1},\nu_{1},\lambda_{1})$ is of the form $(B\otimes
\mathcal{O}{}_{E},\ldots)$, then the Weil classes on $A_{1}$ are $w$-Lefschetz
if they are weakly $w$-Lefschetz (because they are algebraic).
\end{plain}

\begin{plain}
\label{w15}Consider the universal Weil family $f\colon A\rightarrow S$ over
$\mathbb{Q}{}^{\mathrm{al}}$ containing $(A_{1},\nu_{1},\lambda_{1})$. If the
Weil classes on $A_{1}$ are weakly $w$-Lefschetz, and $(A_{1},\nu_{1}%
,\lambda_{1})_{0}$ lifts (up to isogeny) to a triple in the family whose Weil
classes are $w$-Lefschetz, then the Weil classes on $A_{1}$ are $w$-Lefschetz (obviously).
\end{plain}

\begin{plain}
\label{c15} Assume that the Weil classes on $A_{1}$ are weakly $w$-Lefschetz.
If $B$ and the family $\bar{f}\colon\bar{A}\rightarrow C$ in \ref{w7} can be
chosen so that the Weil classes $\gamma$ on $\bar{A}$ (i.e, the global
sections of $W_{E}(\bar{A}/C)$) are fixed by $L(\bar{A}_{\bar{\eta}})$, then
$\gamma_{s}$ is Lefschetz for all closed $s\in C$. In particular, the Weil
classes on $A_{1}$ are $w$-Lefschetz.
\end{plain}

\subsection{CM abelian varieties}

Let $\mathbb{Q}{}^{\mathrm{al}}$ be the algebraic closure of $\mathbb{Q}{}$ in
$\mathbb{C}{}$, and let $w$ be a prime of $\mathbb{Q}{}^{\mathrm{al}}$ lying
over $p$.

\begin{theorem}
\label{r0} Let $A$ be an abelian variety over $\mathbb{Q}{}^{\mathrm{al}}$ of
CM-type. There exist abelian varieties $A_{\Delta}$ of split Weil type and
homomorphisms $f_{\Delta}\colon A\rightarrow A_{\Delta}$ such that every Hodge
class $\gamma$ on $A$ can be written as a sum $\gamma=\sum f_{\Delta}^{\ast
}(\gamma_{\Delta})$ with $\gamma_{\Delta}$ a Weil class on $A_{\Delta}$. If
$\gamma$ is weakly $w$-Lefschetz on $A$, then the $\gamma_{\Delta}$ can be
chosen to be weakly $w$-Lefschetz on $A_{\Delta}$.
\end{theorem}

\begin{proof}
Let $E_{0}$ be the centre of $\End(A_{0})$ and $L_{0}$ its Lefschetz group.
Then the $\mathbb{Q}$-vector space of weakly $w$-Lefschetz classe is
$B^{p}(A)^{L_{0}}$. This is equal to the sum $\sum_{\Delta}f_{\Delta}^{\ast
}(W_{K}(A_{\Delta}))$, where $\Delta$ runs over the classes $\Delta$
satisfying (\ref{eq2}) and such that the elements of
\[
\bigoplus_{t\in T}H^{2p}(A_{\Delta})_{\Delta\times\{t\}}%
\]
are fixed by $L_{0}$. See \ref{a6}.
\end{proof}

\begin{corollary}
\label{r0c}If Conjecture B holds for split Weil classes then it holds for all
Hodge classes on CM abelian varieties.
\end{corollary}

\begin{proof}
The pullback of a $w$-Lefschetz class by a morphism of abelian varieties is
$w$-Lefschetz (because the pullback of a Lefschetz class is Lefschetz).
\end{proof}

\begin{remark}
\label{r0r}It is not true that all Weil classes on abelian varieties of Weil
type over $\mathbb{Q}{}^{\mathrm{al}}$ are weakly $w$-Lefschetz because that
would imply that all Hodge classes on CM abelian varieties specialize to
Lefschetz classes, which is false.
\end{remark}

\section{Conjecture A implies C and D\label{CAR}}

For CM abelian varieties, there is a proof (\cite{milne2009}, \S 4), using
tannakian categories, that Conjecture B implies Conjecture D, hence also
Conjecture C. In the second subsection below, we discuss an elementary
approach to proving that Conjecture B implies Conjecture C, and in the third
subsection we explain the tannakian proof. Recall that Conjecture A implies
Conjecture B.

\subsection{Preliminaries}

\begin{plain}
\label{m1}Recall that a representation of an affine group scheme over $R$ on
an $R$-module $V$ corresponds to co-action $\rho\colon V\rightarrow
V\otimes\mathcal{O}{}(G)$ of the Hopf algebra $\mathcal{O}(G)$ of $G$ on $V$.
We define%
\[
V^{G}=\{v\in V\mid\rho(v)=v\otimes1\}.
\]
Its definition commutes with base change. See \cite{milne2017c}, 4.i.
\end{plain}

\begin{plain}
\label{m2}Let $Q$ be a field and $R$ a $Q$-algebra. Let $M$ be an $R$-module.
Recall that a $Q$-structure on $M$ is a $Q$-submodule $V$ of $M$ such that the
map $v\otimes r\mapsto rm\colon V\otimes_{Q}R\rightarrow M$ is an isomorphism.
We write $V\overset{\otimes}{\longrightarrow}M$ to signify that $V$ is a
$Q$-structure on $M$.
\end{plain}

\begin{plain}
\label{m2a}Let $A$ and $B$ be abelian varieties such that $\Hom(A,B)=0$. Then%
\[
\mathcal{D}{}^{\ast}(A\times B)\simeq\mathcal{D}{}^{\ast}(A)\otimes
\mathcal{D}{}^{\ast}(B).
\]
The corresponding statement for Hodge classes, Tate classes, rational Tate
classes\ldots\ is false. However, for example, $\mathcal{B}{}^{\ast
}(A)=f^{\ast}\left(  \mathcal{B}{}^{\ast}(A\times B)\right)  $, where $f\colon
A\rightarrow A\times B$ is the map $a\mapsto(a,0).$
\end{plain}

\subsection{Conjecture B implies Conjecture C}

\begin{plain}
\label{m3}Let $A$ be a CM abelian variety over $\mathbb{Q}{}^{\mathrm{al}}$
and $A_{0}$ its reduction over $\mathbb{F}{}$. The inclusion $\End^{0}%
(A)\hookrightarrow\End^{0}(A_{0})$ maps the centre $C(A)$ of $\End^{0}(A)$
onto a $\mathbb{Q}{}$-subalgebra of $\End^{0}(A_{0})$ containing its centre
$C(A_{0})$, and hence it defines an inclusion $L(A_{0})\hookrightarrow L(A)$.
We have a commutative diagram%
\[
\begin{tikzcd}
\MT(A)\arrow[hook]{r}&L(A)\\
P(A_{0})\arrow[hook]{r}\arrow[hook]{u}
&L(A_{0}),\arrow[hook]{u}
\end{tikzcd}
\]
where $P(A_{0})$ is the smallest algebraic subgroup of $L(A_{0})$ containing a
Frobenius element for $A_{0}$ (cf. \cite{milne2009}, 3.2).

Let $\mathcal{H}^{r}(A)=H_{\mathbb{A}}^{2r}(A)(r)$. Recall (\cite{deligne1982}%
) that%
\[
\mathcal{B}{}^{r}(A)\overset{\otimes}{\longrightarrow}\mathcal{H}{}%
^{r}(A)^{\MT(A)},
\]
and (\ref{l5})%
\[
\mathcal{D}{}^{r}(A_{0})\overset{\otimes}{\longrightarrow}\mathcal{H}{}%
^{r}(A_{0})^{L(A_{0})}.
\]
Using the projection $\mathcal{H}{}^{r}(A)\rightarrow\mathcal{H}{}^{r}(A_{0}%
)$, we identify $\mathcal{B}^{r}(A)$ with its image in $\mathcal{H}{}%
^{r}(A_{0})$, and then%
\begin{equation}
\mathcal{B}{}^{r}(A)\overset{\otimes}{\longrightarrow}\mathcal{H}{}^{r}%
(A_{0})^{\MT(A)} \label{ep8}%
\end{equation}

\end{plain}

\begin{proposition}
\label{m4}Conjecture $B$ for $A$ implies that%
\[
\mathcal{B}^{r}(A)\cap\mathcal{D}^{r}(A_{0})\overset{\otimes}{\longrightarrow
}\mathcal{H}{}^{r}(A_{0})^{\MT(A)\cdot L(A_{0})}%
\]
(intersection in $\mathcal{H}{}^{r}(A_{0})$).
\end{proposition}

\begin{proof}
On taking fixed points with respect to $L(A_{0})$ in (\ref{ep8}), we find that%
\[
\mathcal{B}{}^{r}(A)^{L(A_{0})}\overset{\otimes}{\longrightarrow}\mathcal{H}%
{}^{r}(A_{0})^{\MT(A)\cdot L(A_{0})}\text{.}%
\]
If $a\in\mathcal{B}{}^{r}(A)$ is fixed by $L(A_{0})$, then it is weakly
$w$-Lefschetz, hence $w$-Lefschetz by Conjecture $B$. Therefore $\mathcal{B}%
{}^{r}(A)^{L(A_{0})}=\mathcal{B}{}^{r}(A)\cap\mathcal{D}{}^{r}(A_{0})$.
\end{proof}

Assume that Conjecture B holds for $A$. From the split-exact sequence%
\[
\begin{tikzcd}[column sep=7ex]
0\arrow{r}
&\mathcal{H}^{r}(A_{0})\arrow{r}{x\mapsto(x,x)}
&\mathcal{H}^{r}(A_{0})\oplus\mathcal{H}^{r}(A_{0})
\arrow{r}{(x,y)\mapsto x-y}
&\mathcal{H}^{r}(A_{0})\arrow{r}
&0,\end{tikzcd}
\]
we get the bottom row of the following exact commutative diagram,%
\[
\begin{tikzcd}[column sep=small]
0\arrow{r}
&\mathcal{B}^{r}(A)\cap\mathcal{D}^{r}(A_{0})\arrow{r}\arrow{d}{\otimes}
&\mathcal{B}^{r}(A)\oplus\mathcal{D}^{r}(A_{0})\arrow{r}\arrow{d}{\otimes}
&\mathcal{B}^{r}(A)+\mathcal{D}^{r}(A_{0})\arrow{r}\arrow[dashed]{d}&0\\
0\arrow{r}&\mathcal{H}^{r}(A_{0})^{\MT(A)\cdot L(A_{0})}\arrow{r}
&\mathcal{H}^{r}(A_{0})^{\MT(A))}
\oplus\mathcal{H}^{r}(A_{0})^{L(A_{0})}\arrow{r}
&\mathcal{H}^{r}(A_{0})^{\MT(A)\cap L(A_{0})}
\end{tikzcd}
\]
From the diagram, we see that the dashed arrow exists and that%
\[
\left(  \mathcal{B}{}^{r}(A)+\mathcal{D}^{r}(A_{0})\right)  \otimes
\mathbb{A}\hookrightarrow\mathcal{H}{}^{r}(A_{0})^{\MT(A)\cdot\cap
L(A_{0})}.
\]
As $P(A_{0})\subset\MT(A)\cap L(A_{0})$, we obtain an injection%
\[
\left(  \mathcal{B}{}^{r}(A)+\mathcal{D}^{r}(A_{0})\right)  \otimes
\mathbb{A}{}_{f}\hookrightarrow\mathcal{T}{}^{r}(A).
\]
This will not usually be an isomorphism.

Let $K$ be a CM subfield of $\mathbb{C}{}$ that is finite and Galois over
$\mathbb{Q}{}$. We say that $K$ splits a CM abelian variety $A$ if
$\End^{0}(A)\otimes_{\mathbb{Q}}K$ is a product of matrix algebras over $K$.
Let $A^{K}$ be a CM abelian variety over $\mathbb{Q}^{\mathrm{al}}$ split by
$K$ and so large that every simple abelian variety over $\mathbb{Q}%
{}^{\mathrm{al}}$ split by $K$ is isogenous to an abelian subvariety of $A$.
Then $A_{0}^{K}$ is an abelian variety over $\mathbb{F}{}$ such that every
simple abelian variety over $\mathbb{F}{}$ split by $K$ is isogenous to an
abelian subvariety of $A_{0}^{K}$. For $A^{K}$, we have%
\[
P(A_{0}^{K})=\MT(A^{K})\cap L(A_{0}^{K}),
\]
(\cite{milne1999b}, Theorem 6.1). Is the $\mathbb{Q}{}$-subalgebra of
$\mathcal{T}{}^{\ast}(A_{0}^{K})$ generated by $\mathcal{B}{}^{\ast}(A^{K})$
and $\mathcal{D}{}^{\ast}(A_{0}^{K})$ a $\mathbb{Q}{}$-structure on
$\mathcal{T}^{\ast}(A_{0}^{K})$? If so, we could define $\mathcal{R}{}^{\ast
}(A_{0})$ to be this $\mathbb{Q}{}$-algebra. For any other abelian variety $B$
over $\mathbb{F}{}$, there exists a  homomorphism $f\colon B\rightarrow
A_{0}^{K}$ with finite kernel for some $K$, and we could define
\[
\mathcal{R}^{\ast}(B)=f^{\ast}\mathcal{R}^{\ast}(A_{0})
\]
This would give an elementary construction of the family of rational Tate classes.

\begin{summary}
\label{m5}Assume that conjecture B holds for CM abelian varieties over
$\mathbb{Q}{}^{\mathrm{al}}$. Then there exists a unique family $\mathcal{R}%
{}^{\ast}(A)$ of $\mathbb{Q}{}$-structures on the $\mathbb{A}_{f}{}$-algebras
$\mathcal{T}{}^{\ast}(A)$, indexed by the abelian varieties over $\mathbb{F}$,
satisfying (R1) and (R2) and such that Hodge classes on CM abelian varieties
over $\mathbb{Q}{}^{\mathrm{al}}$ specialize to rational Tate classes.
\end{summary}

Without Conjecture B, the map $(\mathcal{B}{}^{r}(A)+\mathcal{D}{}^{r}%
(A_{0}))\otimes\mathbb{A}_{f}{}\rightarrow\mathcal{T}^{r}(A_{0}$) need not be injective.

\subsection{Conjecture B implies Conjecture D}

In this section, we assume that Conjecture B holds for all CM abelian
varieties over $\mathbb{Q}^{\mathrm{al}}$, and we construct the category of
motives $\Mot(\mathbb{F})$ over $\mathbb{F}$. This section is largely a review
of earlier work of the author.

\subsubsection{Statements}

\begin{plain}
\label{d1}Assuming Conjecture B for CM abelian varieties, we construct
commutative diagrams%

\begin{equation}
\begin{tikzcd} S&S_{\mathbb{Q}_l}&&\CM(\mathbb{Q}^{\mathrm{al}}) \arrow{d}{R} \arrow{rd}{\xi_l}\\ P\arrow{u}&P_{\mathbb{Q}_l} \arrow{u}&P_l\arrow{l}\arrow{lu} &\Mot(\mathbb{F})\arrow{r}{\eta_l}&\mathsf{R}_{l} (\mathbb{F}) \end{tikzcd}\quad
l=2,\ldots,p,\ldots, \label{ep3}%
\end{equation}
where

\begin{itemize}
\item $\CM(\mathbb{Q}^{\mathrm{al}})$ is the subcategory of $\Mot^{w}%
(\mathbb{Q}^{\mathrm{al}})$ of motives of CM-type;

\item $\Mot(\mathbb{F}{})$ is a tannakian category over $\mathbb{Q}{}$ with
fundamental group $P$;

\item $P\rightarrow S$ is the Shimura--Taniyama homomorphism (\ref{b3})

\item $R\colon\CM(\mathbb{Q}{}^{\mathrm{al}})\rightarrow\Mot(\mathbb{F}{})$ is
a quotient functor bound by $P\rightarrow S$;

\item $\xi_{l}\colon\CM(\mathbb{Q}^{\mathrm{al}})\rightarrow\mathsf{R}%
_{l}(\mathbb{F}{})$ is the realization functor (\ref{b16}, \ref{b17}).
\end{itemize}
\end{plain}

\subsubsection{A construction}

Let $\LCM(\mathbb{Q}{}^{\mathrm{al}})$ denote the tannakian subcategory of
$\LMot(\mathbb{Q}{}^{\mathrm{al}})$ generated by the abelian varieties of
CM-type. There are canonical exact tensor functors $J\colon\LCM(\mathbb{Q}%
{}^{\mathrm{al}})\rightarrow\CM(\mathbb{Q}{}^{\mathrm{al}})$ and
$R\colon\LCM(\mathbb{Q}{}^{\mathrm{al}})\rightarrow\LMot(\mathbb{F}{})$ giving
rise to homomorphisms $S\hookrightarrow T$ and $L\hookrightarrow T$ of
(commutative) fundamental groups. We shall shall construct quotient functors
$q\colon\CM(\mathbb{Q}^{\mathrm{al}})\rightarrow\Mot^{\prime}(\mathbb{F})$ and
$q^{\prime}\colon\LMot(\mathbb{F}{})\rightarrow\Mot^{\prime}(\mathbb{F}{})$
with the following properties:

\begin{enumerate}
\item the diagram at left commutes and corresponds to the diagram of
fundamental groups at right%
\[
\begin{tikzcd}
\CM(\mathbb{Q}^{\mathrm{al}})\arrow{d}{q}
&\LCM(\mathbb{Q}^{\mathrm{al}})\arrow{l}[swap]{J}\arrow{d}{R}
&S\arrow[hook]{r}&T\\
\Mot^{\prime}(\mathbb{F})
&\LMot(\mathbb{F})\arrow{l}[swap]{q^{\prime}}
&P\arrow[hook]{r}\arrow[hook]{u}&L\arrow[hook]{u}
\end{tikzcd}
\]

\item the functors $\xi_{l}\colon\CM(\mathbb{Q}{}^{\mathrm{al}})
\rightarrow\mathsf{R}_{l}(\mathbb{F}{})$ factor through $q$.
\end{enumerate}

The functors $R$ and $J$ are both quotient functors, and so correspond to
$\mathbb{Q}{}$-valued functor $\omega^{R}$ and $\omega^{J}$ on
$\LCM(\mathbb{Q}{}^{\mathrm{al}})^{L}$ and $\LCM(\mathbb{Q}{}^{\mathrm{al}%
})^{S}$ respectively (see \ref{t2}). Conjecture B for CM abelian varieties
says exactly that these two functors restrict to the same fibre functor
$\omega_{1}$ on $\LCM(\mathbb{Q}^{\mathrm{al}})^{L\cdot S}$ and that
$\omega_{1}$ is a $\mathbb{Q}{}$-structure on the ad\'{e}lic fibre functor on
$\LCM(\mathbb{Q}{}^{\mathrm{al}})^{L\cdot S}$ defined by the standard Weil
cohomology theories. As $P=S\cap L$ (\cite{milne1999b}, 6.1), the sequence%
\[
0\rightarrow S/P\rightarrow T/L\rightarrow T/(L\cdot S)\rightarrow0
\]
is exact. Therefore $J|\colon\LCM(\mathbb{Q}^{\mathrm{al}})^{L}\rightarrow
\CM(\mathbb{Q}{}^{\mathrm{al}})^{P}$ is a quotient functor and\newline$\left(
\LCM(\mathbb{Q}^{\mathrm{al}})^{L}\right)  ^{S/P}=\LCM(\mathbb{Q}%
^{\mathrm{al}})^{L\cdot S}$:%
\[
\begin{tikzcd}[column sep=scriptsize]
\CM(\mathbb{Q}^{\mathrm{al}})^P\arrow{d}
&\LCM(\mathbb{Q}^{\mathrm{al}})^L\arrow{l}[swap]{J|}\arrow{d}
&\LCM(\mathbb{Q}^{\mathrm{al}})^{L\cdot S}\arrow{l}\arrow{d}
&S/P\arrow{r}&T/L\arrow{r}&T/L\cdot S\\
\CM(\mathbb{Q}^{\mathrm{al}})\arrow{d}{q}
&\LCM(\mathbb{Q}^{\mathrm{al}})\arrow{l}[swap]{J}\arrow{d}{R}
&\LCM(\mathbb{Q}^{\mathrm{al}})^S\arrow{l}&S\arrow{r}\arrow{u}
&T\arrow{u}\arrow{r}&T/S\arrow{u}\\
\Mot^{\prime}(\mathbb{F})
&\LMot(\mathbb{F})\arrow{l}[swap]{q^{\prime}}
&&P\arrow{r}\arrow{u}&L\arrow{u}
\end{tikzcd}
\]
From $\omega_{1}$ and the equality $\omega_{1}=\omega^{R}|$, we get a fibre
functor $\omega_{0}$ on $\CM(\mathbb{Q}^{\mathrm{al}})^{P}$ (see \ref{t3})
such that

\begin{enumerate}
\item $\omega_{0}|\LCM(\mathbb{Q}{}^{\mathrm{al}})^{L}=\omega^{R}$,

\item $\omega_{0}$ is a $\mathbb{Q}{}$-structure on $x_{l}$ for $l=2,\ldots
,p,\ldots.$
\end{enumerate}

\noindent We define $\Mot^{\prime}(\mathbb{F})$ to be the quotient
$\CM(\mathbb{Q}^{\mathrm{al}})/\omega_{0}$. Because of (a), the functor
$q\circ J$ factors through $R$, say, $q\circ J=q^{\prime}\circ R$. The triple
($\Mot^{\prime}(\mathbb{F}{}),q,q^{\prime})$ has the properties (a) and (b).
See \cite{milne2009}, \S 4, for more details.

\section{Extending the reduction functor}

\label{ER}

In this section, \textit{we assume that Conjecture B holds for all CM abelian
varieties over} $\mathbb{Q}^{\mathrm{al}}$, and we investigate whether the
reduction functor $R\colon\CM(\mathbb{Q}{}^{\mathrm{al}})\rightarrow
\Mot(\mathbb{F}{})$ extends to $\Mot^{w}(\mathbb{Q}{}^{\mathrm{al}})$.

\subsection{Statements}

Let $\Mot^{\prime}(\mathbb{Q}{}^{\mathrm{al}})$ be a tannakian subcategory of
$\Mot^{w}(\mathbb{Q}{}^{\mathrm{al}})$ containing $\CM(\mathbb{Q}%
^{\mathrm{al}})$, and consider the following statements.

\begin{theorem}
\label{r3}The reduction functor $R\colon\CM(\mathbb{Q}{}^{\mathrm{al}%
})\rightarrow\Mot(\mathbb{F}{})$ extends uniquely to a functor $R\colon
\Mot^{\prime}(\mathbb{Q}^{\mathrm{al}})\rightarrow\Mot(\mathbb{F}{})$ with the
following properties:

\begin{enumerate}
\item if $hA\in$ $\ob\Mot^{\prime}(\mathbb{Q}^{\mathrm{al}})$, then
$R(hA)=hA_{0}$, and

\item the diagrams
\[
\begin{tikzcd}
\Mot^{\prime}(\mathbb{Q}^{\mathrm{al}}) \arrow{d}{R}
\arrow{rd}{\xi_l}\\
\Mot(\mathbb{F})\arrow{r}{\eta_l} &R_{l}(\mathbb{F}) \end{tikzcd}
\]
commute for all prime numbers $l$.
\end{enumerate}
\end{theorem}

\noindent If $\Mot^{\prime}(\mathbb{Q}^{\mathrm{al}})$ is generated by the
abelian varieties it contains, then the uniqueness is obvious.

\begin{corollary}
\label{r1}Let $A$ be an abelian variety over $\mathbb{Q}^{\mathrm{al}}$. If
$hA$ lies in $\Mot^{\prime}(\mathbb{Q}^{\mathrm{al}})$, then Hodge classes on
$A$ specialize to rational Tate classes on $A_{0}$.
\end{corollary}

\begin{proof}
Obvious from the definitions.
\end{proof}

\begin{corollary}
\label{r11}Conjecture A holds for all abelian varieties over $\mathbb{Q}%
^{\mathrm{al}}$ such that $hA\in\Mot^{\prime}(\mathbb{Q}{}^{\mathrm{al}})$.
\end{corollary}

\begin{proof}
Obvious from Corollary \ref{r1}.
\end{proof}

For an abelian motive $M$ over $\mathbb{Q}^{\mathrm{al}}\subset\mathbb{C}{}$,
we let $G_{M}$ denote the Mumford--Tate group of the rational Hodge structure
$\omega_{B}(M)$.

\begin{corollary}
\label{r11c}Let $M$ be a motive in $\Mot^{\prime}(\mathbb{Q}^{\mathrm{al}})$.
The Galois representation attached to any model of $M$ over a sufficiently
large algebraic number field in $\mathbb{Q}^{\mathrm{al}}$ takes values in
$G_{M}$ and is strictly compatible.
\end{corollary}

\begin{proof}
The first part of the statement is obvious, and the second follows from the
properties of the motive $R(M)$. See \ref{d9c}.
\end{proof}

When $\Mot^{\prime}(\mathbb{Q}^{\mathrm{al}})=\CM(\mathbb{Q}{}^{\mathrm{al}}%
)$, Theorem \ref{r3} says nothing new. When $\Mot^{\prime}(\mathbb{Q}%
{}^{\mathrm{al}})$ is the category generated by the abelian varieties with
very good reduction, we prove this in \ref{r5} below. When $\Mot^{\prime
}(\mathbb{Q}{}^{\mathrm{al}})$ consists of the abelian motives with visibly
good reduction, we suggest two approaches to proving it. When $\Mot^{\prime
}(\mathbb{Q}{}^{\mathrm{al}})=\Mot^{w}(\mathbb{Q}{}^{\mathrm{al}})$, it is
left as an exercise for the reader.

\subsubsection{Notes}

\begin{plain}
\label{r11a}For a collection $\mathscr{s}{}$ of abelian varieties over
$\mathbb{Q}{}^{\mathrm{al}}$ with good reduction at $w$, we define
$\Mot^{\mathscr{s}}(\mathbb{Q}^{\mathrm{al}})$ to be the tannakian subcategory
of $\Mot^{w}(\mathbb{Q}{}^{\mathrm{al}})$ generated by $\mathscr{s}$. Clearly
Theorem \ref{r3} holds for $\Mot^{\mathscr{s}}(\mathbb{Q}^{\mathrm{al}})$ if
and only if the Hodge classes on the abelian varieties in $\mathscr{s}{}$
specialize to rational Tate classes.
\end{plain}

\begin{plain}
\label{r11b}We would like to prove Theorem \ref{r3} in the general case using
as little of the theory of Shimura varieties as possible. One of the goals of
this article is to \textit{recover} the theory of Shimura varieties from the
theory of motives, not merely enhance it.
\end{plain}

\begin{plain}
\label{r11d}Let $A$ be an abelian variety over a number field $K$.
\cite{kisinZ2025} show that, after replacing $K$ with a finite extension, the
Weil--Deligne representation attached to $A$ takes values in the Mumford-Tate
group of $A$ and is strictly compatible. In other words, a version of
Corollary \ref{r11c} holds for \textit{all} abelian varieties over
$\mathbb{Q}{}^{\mathrm{al}}$, not just those with good reduction at $w$. This
suggests, as noted elsewhere, that many statements concerning abelian
varieties with good reduction at $w$ should extend mutatis mutandis to all
abelian varieties over $\mathbb{Q}{}^{\mathrm{al}}$.
\end{plain}

\subsection{CM lifts}

Up to isogeny, every abelian variety over $\mathbb{F}{}$ lifts to a CM abelian
variety in characteristic zero (\ref{x24}). There is the following more
precise conjecture.

\begin{conjecture}
\label{r12}Let $A$ be an abelian variety over $\mathbb{Q}^{\mathrm{al}}$ with
good reduction at $w$ and let $\gamma$ be a Hodge class on $A$. There exist a
CM abelian variety $A^{\prime}$ over $\mathbb{Q}{}^{\mathrm{al}}$ and a Hodge
class $\gamma^{\prime}$ on $A^{\prime}$ such that $(A,\gamma)_{0}%
\sim(A^{\prime},\gamma^{\prime})_{0}$, i.e., such that there exists an isogeny
$A_{0}\rightarrow A_{0}^{\prime}$ sending $\gamma_{0}$ to $\gamma_{0}^{\prime
}$.
\end{conjecture}

\begin{plain}
\label{r12a}If Conjecture \ref{r12} holds for all $\gamma$ on the abelian
variety $A$, then $A$ satisfies Conjecture A. Indeed, the condition implies
that $\gamma_{0}$ is a rational Tate class on $A_{0}$, and intersections of
rational Tate classes of complementary dimension are rational numbers.
\end{plain}

\begin{plain}
\label{r12c}An abelian motive $M$ over $\mathbb{Q}{}^{\mathrm{al}}$ is said to
have \emph{visibly good reduction} if it can be expressed in the form
$h(A,e,m)$ with $A$ an abelian variety with good reduction at $w$. We write
$\Mot^{\text{vis}}(\mathbb{Q}^{\mathrm{al}})$ for the category of abelian
motives over $\mathbb{Q}{}^{\mathrm{al}}$ (tannakian subcategory of
$\Mot^{w}(\mathbb{Q}{}^{\mathrm{al}})$).
\end{plain}

\begin{plain}
\label{r12b}Conjecture \ref{r12} implies that Conjecture B holds for all
abelian varieties over $\mathbb{Q}^{\mathrm{al}}$ with good reduction at $w$.
The same argument as in \S 5 then allows us to define $\Mot(\mathbb{F}{})$ to
be $\Mot^{\text{vis}}(\mathbb{Q}^{\mathrm{al}})/\omega$ for a suitable fibre
functor $\omega$ on $\Mot^{\mathrm{vis}}(\mathbb{Q}{}^{\mathrm{al}})^{P}$, and
the functor $\CM(\mathbb{Q}^{\mathrm{al}})\rightarrow\Mot^{\text{vis}%
}(\mathbb{Q}^{\mathrm{al}})$ induces an equivalence $\CM(\mathbb{Q}%
^{\mathrm{al}})/\omega\rightarrow\Mot^{\text{vis}}(\mathbb{Q}^{\mathrm{al}%
})/\omega$. Therefore Conjecture \ref{r12} implies Theorem \ref{r3} for
abelian motives with visibly good reduction.
\end{plain}

\begin{plain}
\label{r10}There is a converse to the last statement. Let $\Sh_{p}(G,X)$ be a
Shimura variety abelian type with rational weight satisfying the condition to
have good reduction at $p$. From Theorem $\ref{r3}$ for $\Mot^{w}(\mathbb{F}%
{})$, we obtain an integral canonical model of the Shimura variety and a
description of it as a moduli variety for abelian motives (see the next
section). Proceeding as in \cite{langlandsR1987}, we then attach to each point
of $\Sh_{p}(\mathbb{F}{})$ an admissible homomorphism $\mathfrak{P}%
\rightarrow\mathfrak{G}_{G}$. Now a cohomological argument (assuming
$G^{\mathrm{der}}$ is simply connected) shows that $\varphi$ is the
homomorphism attached to a special point (op.\ cit.\ 5.3). In this way, we see
that every point of $\Sh_{p}(\mathbb{F}{})$ lifts to a special point. Cf.~\S 4
of \cite{milne1992}, especially Theorem 4.6.
\end{plain}

\begin{plain}
\label{r10a}There are many results in the literature concerning Conjecture
\ref{r12}. For example, \cite{kisinZ2021} prove that every point in the $\mu
$-ordinary locus of the special fibre of a Shimura variety lifts to a special point.
\end{plain}

\subsection{Nifty abelian varieties}

Recall that an abelian variety $A$ over $\mathbb{Q}{}^{\mathrm{al}}$ with good
reduction at $w$ is nifty if $\MT(A)\cdot L(A_{0})=L(A)$.

\begin{proposition}
Hodge classes on nifty abelian varieties specialize to rational Tate classes.
\end{proposition}

\begin{proof}
Omitted for the moment.
\end{proof}

\subsection{Abelian varieties with very good reduction}

We say that an abelian variety $A$ has \emph{very good reduction} at $w$ if it
has good reduction at $w$ and the adjoint group of $\MT(A)$ is unramified at
$p$. Note that products of abelian varieties with very good reduction have
very good reduction, and that all CM abelian varieties over $\mathbb{Q}%
{}^{\mathrm{al}}$ have very good reduction.

Let $\Mot^{\mathrm{vg}}(\mathbb{Q}^{\mathrm{al}})$ denote the category of
abelian motives over $\mathbb{Q}{}^{\mathrm{al}}$ generated by the abelian
varieties with very good reduction. We explain in this subsection how to
extend $R$ to $\Mot^{\mathrm{vg}}(\mathbb{Q}^{\mathrm{al}})$.

Let $A$ have very good reduction at $w$, and let $(G,h)$ be the Mumford-Tate
group of $A$. Let $X$ be the conjugacy class of $h$. Then $(G,X)$ is a Shimura
datum, and so, for every compact open $K\subset G(\mathbb{A}{}_{f})$, we have
a variety $\Sh_{K}(G,X)$ over $\mathbb{C}$. We assume that $K=K_{p}\times
K^{p}$ with $K_{p}$ a hyperspecial subgroup of $G(\mathbb{Q}{}_{p}$) and
$K^{p}$ a sufficiently small subgroup of $G(\mathbb{A}_{f}^{p})$. The action
of $G$ on the $\mathbb{Q}{}$-vector space $V\overset{\df}{=}H_{1}%
(A,\mathbb{Q}{})$ allows us to realize $\Sh_{K}(G,X)$ as the solution to a
moduli problem over $\mathbb{C}{}$. The moduli problem is defined over
$\mathbb{Q}{}^{\mathrm{al}}$, and descent theory shows that $\Sh_{K}(G,X)$ has
a canonical model over $\mathbb{Q}{}^{\mathrm{al}}$, which we denote
$S_{\mathbb{Q}{}^{\mathrm{al}}}$, and which is the solution to a moduli
problem over $\mathbb{Q}{}^{\mathrm{al}}$. Specifically, $S_{\mathbb{Q}%
{}^{\mathrm{al}}}$ parametrizes triples $(B,\mathfrak{k}{},\lambda)$ where $B$
is an abelian variety over $\mathbb{Q}{}^{\mathrm{al}}$, $\mathfrak{k}{}$ is a
family of Hodge tensors on $B$ and its powers, and $\lambda$ is a level
structure on $B$. We choose $\mathfrak{k}{}$ to be the family of all Hodge
classes on $A$ and its powers. We assume that $K$ has been chosen small enough
to force $B$ to have good reduction at $w$. Results of Kisin and Vasiu, show
that $S_{\mathbb{Q}{}^{\mathrm{al}}}$ extends to a smooth canonical model $S$
over $\mathcal{O}{}_{w}$, and that the point of $S(\mathbb{F)}$ defined by
$(A,\mathfrak{k}{},\lambda)$ is isogenous to the reduction of a special point.
Specifically, this means that there exists an abelian variety $B$ over
$\mathbb{Q}^{\mathrm{al}}$ and an isogeny $A_{\mathbb{F}{}}\rightarrow
B_{\mathbb{F}{}}$ such that

\begin{enumerate}
\item $B$ is of CM-type;

\item let $\gamma$ be a Hodge class on $A_{\mathbb{Q}{}^{\mathrm{al}}}$ and
$\gamma^{\prime}$ the corresponding Hodge class on $B$; then, under the
isogeny under the isogeny, $\gamma_{l}$ maps to $\gamma_{l}^{\prime}$ for all
$\gamma$.
\end{enumerate}

\noindent Therefore, there exists an exact $\mathbb{Q}{}$-linear tensor
functor $\langle A\rangle^{\otimes}\rightarrow\Mot(\mathbb{F}{})$ such that
the diagram%
\[
\begin{tikzcd}
\langle A\rangle^{\otimes}\arrow{d}{R}\arrow{rd}{\xi_l}\\
\Mot(\mathbb{F})\arrow{r}{\eta_l}&R_l(\mathbb{F})
\end{tikzcd}
\]
commutes. Let $A^{\prime}$ be a second abelian variety over $\mathbb{Q}%
{}^{\mathrm{al}}$ with very good reduction at $w$. On repeating the argument
for $A\times A^{\prime}$, we can extend the above diagram to a diagram
\[
\begin{tikzcd}
\langle A\times A^{\prime}\rangle^{\otimes}\arrow{d}{R}\arrow{rd}{\xi_l}\\
\Mot(\mathbb{F})\arrow{r}{\eta_l}&\Mot(\mathbb{F}).
\end{tikzcd}
\]
As the set of isogeny classes of abelian varieties over $\mathbb{Q}%
{}^{\mathrm{al}}$ with very good reduction is countable, this will eventually
lead to a functor $R\colon\Mot^{\mathrm{vg}}(\mathbb{Q}^{\mathrm{al}%
})\rightarrow\Mot(\mathbb{F}{})$ such that the diagrams (\ref{ep1}) commute
(by the axiom of dependent choice). We have proved the following statement.

\begin{theorem}
\label{r5}The reduction functor $R\colon\CM(\mathbb{Q}{}^{\mathrm{al}%
})\rightarrow\Mot(\mathbb{F}{})$ extends uniquely to $\Mot^{\mathrm{vg}%
}(\mathbb{Q}{}^{\mathrm{al}})$, and makes the diagrams (\ref{ep1}) commute.
\end{theorem}

\begin{corollary}
\label{r6}Let $A$ be an abelian variety over $\mathbb{Q}^{\mathrm{al}}$. If
$A$ has very good reduction at $w$, then Hodge classes on $A$ specialize to
rational Tate classes on $A_{0}$.
\end{corollary}

\begin{corollary}
\label{r7} Conjecture A holds for abelian varieties with very good reduction.
\end{corollary}

\begin{proof}
Obvious from Corollary \ref{r6}.
\end{proof}

\begin{corollary}
\label{r7a}Let $M$ be a motive in $\Mot^{\mathrm{vg}}(\mathbb{Q}^{\mathrm{al}%
})$. The Galois representation attached to any model of $M$ over a
sufficiently large algebraic number field in $\mathbb{Q}{}^{\mathrm{al}}$
takes values in $G_{M}$ and is strictly compatible.
\end{corollary}

\begin{remark}
\label{r8}The hypothesis that we have made in this section that
$\MT(A)^{\mathrm{ad}}$ be unramified at $p$, i.e., quasi-split over
$\mathbb{Q}{}_{p}$ and splits over an unramified extension, is unnecessarily
strong. For example, if Conjecture 1 of \cite{kisinM2022} holds, then we can
replace it with the requirement that $\MT(A)$ be quasi-split over
$\mathbb{Q}{}_{p}$. See also Reimann 1997, B3.12.
\end{remark}

\subsection{Abelian motives with visibly good reduction}

In this subsection, we investigate the following statement, which implies that
Theorem \ref{r3} holds for abelian motives over $\mathbb{Q}{}^{\mathrm{al}}$
with visibly good reduction at $w$ (cf.~\ref{r12c}, \ref{r12b}).

\begin{plain}
\label{r9}For every abelian variety $A$ over $\mathbb{Q}{}^{\mathrm{al}}$ with
good reduction at $w$, Hodge classes on $A$ specialize to rational Tate classes.
\end{plain}

\subsubsection{First approach}

\begin{plain}
\label{r20}As in \S 6 of \cite{deligne1982}, embed $(A,\gamma)$ in a family of
abelian varieties with additional structure over $\mathbb{C}{}$. In
particular, $\gamma$ extends to a global section of the family. Now the family
is defined over a number field, and specializes to a family over $\mathbb{F}%
{}$. Complete the proof by showing that $A_{0}$ lifts to a CM abelian variety
in the family (this seems to be more general than known, or even conjectured, results).
\end{plain}

\subsubsection{Second approach}

We saw earlier (Theorem \ref{r22}) that \ref{r9} holds under some hypotheses.

\subsection{Abelian motives with good reduction}

As mentioned earlier, all statements in this article for abelian varieties
over $\mathbb{Q}^{\mathrm{al}}$ with good reduction at $w$ should hold mutatis
mutandis also for those with bad reduction. In particular, the second approach
in the last subsection should yield a proof of Theorem \ref{r3} for
$\Mot^{w}(\mathbb{Q}^{\mathrm{al}})$.

\section{Consequences of Conjecture A for motives}

\label{CM} In this section, we assume the conjectures A, B, C, D with
$\mathscr{s}$ the collection of CM abelian varieties, and we investigate some
of the applications of the conjectures.

\subsection{The category of motives over $\mathbb{F}$}

\begin{quote}
We define the category of motives\footnote{The reader may ask why we call this
the category of motives over $\mathbb{F}{}$ rather than the category of
abelian motives. Conjecturally, the two are the same. Specifically, when we
assume the Tate and standard conjectures over $\mathbb{F}$, the simple objects
of $\Mot(\mathbb{F})$ are classified by the conjugacy classes of Weil numbers,
and the abelian motives exhaust the possible Weil numbers (\cite{milne1994a},
2.7).} over $\mathbb{F}{}$ and show that it has most of the properties that
Grothendieck's category of numerical motives would have if the Tate and
standard conjectures were known over $\mathbb{F}{}$.
\end{quote}

\begin{plain}
\label{d4}Let $\Mot(\mathbb{F})$ denote the category of motives based on the
abelian varieties over $\mathbb{F}{}$ using the rational Tate classes as
correspondences. Specifically, its objects are triples $(A,e,m)$, where

\begin{itemize}
\item $A$ is a variety over $\mathbb{F}{}$ each of whose components admits the
structure of an abelian variety,

\item $e\in\mathcal{R}{}^{\dim A}(A\times A)$ is such that $e^{2}=e$, and

\item $m\in\mathbb{Z}{}$
\end{itemize}

\noindent and the Homs are given by%
\[
\Hom((A,e,m),(B,f,n))\mathcal{=}f\cdot\mathcal{\mathcal{R}{}}^{\dim
A+n-m}(A\times B)\cdot e.
\]
We use that the K\"{u}nneth components of the diagonal are rational Tate
classes to modify the commutativity constraint.
\end{plain}

\begin{plain}
\label{d5}The category $\Mot(\mathbb{F})$ is a tannakian category over
$\mathbb{Q}$ with canonical functors
\[
\eta_{l}\colon\Mot(\mathbb{F}{})\rightarrow R_{l}(\mathbb{F}),\qquad l\text{ a
prime number}.
\]
Its band is the Weil-number protorus $P$. The category $\Mot(\mathbb{F}{})$
has a canonical structure of a Tate triple.
\end{plain}

\begin{plain}
\label{d6}There is a canonical functor $R$ making the diagrams in \ref{d1}
commute. There is a unique polarization on $\Mot(\mathbb{F}{})$ compatible
with the canonical polarization on $\CM(\mathbb{Q}{}^{\mathrm{al}})$. This can
be proved as in \cite{milne2002}.
\end{plain}

\begin{plain}
\label{d8}Grothendieck's standard conjectures hold for abelian varieties over
$\mathbb{F}$ and rational Tate classes. For the Lefschetz standard conjecture,
this follows from the fact that the inverse Lefschetz operator is even
Lefschetz. For the Hodge standard conjecture, it is a restatement of \ref{d6}
\end{plain}

\begin{plain}
\label{d8a}The functor $\Mot(\mathbb{F}{})\rightarrow\Mot(\mathbb{F}%
{};\mathbb{Q}{}_{l})$ defines an equivalence%
\[
\Mot(\mathbb{F}{})_{(\mathbb{Q}{}_{l})}\rightarrow\Mot(\mathbb{F}{}%
;\mathbb{Q}{}_{l})\text{.}%
\]
Thus, we can regard the tannakian category $\Mot(\mathbb{F}{})$ as being
(simultaneously) a $\mathbb{Q}{}$-structure on the tannakian categories being
a $\Mot(\mathbb{F})_{(\mathbb{Q}_{l})}\rightarrow V_{l}(\mathbb{F})$ are
equivalences of categories.
\end{plain}

\subsection{The category of motives over a finite field.}

\begin{plain}
\label{d9}In Grothendieck's motivic paradise, $\Mot(\mathbb{F}{}_{q})$ is a
tannakian category over $\mathbb{Q}$ with fundamental group $P(q)$ (see
\ref{b2a}), and $\Mot(\mathbb{F})$ is a tannakian category over $\mathbb{Q}$
with fundamental group the Weil-number protorus $P$ (see \cite{milne1994a}).
The functor $\Mot(\mathbb{F}{}_{q})\rightarrow\Mot(\mathbb{F}{})$ identifies
$\Mot(\mathbb{F}{}_{q})$ with the category whose objects are pairs consisting
of an object of $\Mot(\mathbb{F})$ and an action of $P(q)$ on the object
consistent with the action of $P$ (see the author's book on Tannakian
Categories). To give such an action, it suffices to endow the object with a
suitable Frobenius endomorphism.
\end{plain}

\begin{plain}
\label{d9a}In particular, the preceding remark suggests the following
\textit{definition.} Every object $\Mot(\mathbb{F}{})$ is equipped with an
action of $P$. In particular, it has a germ of Frobenius endomorphisms. We
define $\Mot(\mathbb{F}{}_{q}$) to be the category whose objects are the pairs
$(M,\pi_{M})$, where $M$ is an object of $\Mot(\mathbb{F})$ and $\pi_{M}$ is a
Frobenius endomorphism representing the germ and such that $\rho(\pi_{M}%
)\cdot\overline{\rho(\pi_{M})}=q^{m}$. The resulting category $\Mot(\mathbb{F}%
_{q})$ has essentially all the properties that the category of numerical
motives has in Grothendieck's motivic paradise. (See \cite{milne1994a}.)
\end{plain}

\begin{plain}
\label{d9b}The category $\Mot(\mathbb{F}{}_{q}$) is semisimple. The Frobenius
element $\pi_{M}$ of a simple motive $M$ is a Weil $q$-number of weight $m$.
The isomorphism classes of simple motives are classified by the orbits of
$\Gamma\overset{\df}{=}\Gal(\mathbb{Q}{}^{\mathrm{al}}/\mathbb{Q}{})$ acting
on $W(q)$,%
\[
\Sigma(\Mot(\mathbb{F}{}_{q}))\simeq\Gamma\backslash W(q).
\]
Let $M$ be a simple motive over $\mathbb{F}{}_{q}$. Then $E\overset{\df}{=}%
\End(M)$ is a simple $\mathbb{Q}{}$-algbra with centre $\pi_{M}$, and its
invariant at a prime $v$ of $\mathbb{Q}{}[\pi_{M}]$ is given by%
\[
\renewcommand{\arraystretch}{1.5}\inv_{v}(E)=\left\{
\begin{array}
[c]{l}%
1/2\text{ if }v\text{ is real and }M\text{ has odd weight,}\\
\dfrac{\ord_{v}(\pi_{X})}{\ord_{v}(q)}\cdot\lbrack\mathbb{Q}{}[\pi_{X}%
]_{v}\colon\mathbb{Q}{}_{p}]\text{ if }v|p,\\
0\text{ otherwise.}%
\end{array}
\right.
\]

\end{plain}

\subsection{Abelian motives in characteristic zero}

Assume now that the reduction functor $R$ extends to $\Mot^{w}(\mathbb{Q}%
{}^{\mathrm{al}})$. This allows us to prove results, otherwise unknown, about
abelian motives in characteristic zero (in the sense of \cite{deligne1982}).

Let $P$ be the Weil-number protorus and let $G=\mathcal{A}ut^{\otimes}%
(\omega_{B}|\Mot^{w}(\mathbb{Q}{}^{\mathrm{al}}))$ --- they are both affine
group schemes over $\mathbb{Q}{}$. The reduction functor%
\[
R\colon\Mot^{w}(\mathbb{Q}{}^{\mathrm{al}})\rightarrow\Mot(\mathbb{F}{})
\]
defines a morphism of the associated bands%
\[
bP\rightarrow bG.
\]
We make this explicit. The choice of a $\mathbb{Q}^{\mathrm{al}}$-valued fibre
functor on $\Mot(\mathbb{F})$ and an isomorphism $\omega_{B}\otimes
\mathbb{Q}{}^{\mathrm{al}}\rightarrow\omega\circ R$ defines a homomorphism%
\[
P_{\mathbb{Q}{}^{\mathrm{al}}}\rightarrow G_{\mathbb{Q}{}^{\mathrm{al}}}%
\]
that is independent of the choices up to an inner automorphism of
$G_{\mathbb{Q}^{\mathrm{al}}}$. In this way we get a conjugacy class of
homomorphisms $P_{\mathbb{Q}{}^{\mathrm{al}}}\rightarrow G_{\mathbb{Q}%
{}^{\mathrm{al}}}$ stable under $\Gal(\mathbb{Q}^{\mathrm{al}}/K)$. Let $clG$
denote the scheme of conjugacy classes in $G$ (quotient of $G$ by its action
on itself by inner automorphisms). The composite $P_{\mathbb{Q}{}%
^{\mathrm{al}}}\rightarrow G_{\mathbb{Q}{}^{\mathrm{al}}}\rightarrow
(clG)_{\mathbb{Q}{}^{\mathrm{al}}}$ is independent of all choices and is
defined over $\mathbb{Q}$,%
\[
\begin{tikzcd}
P_{\mathbb{Q}^{\mathrm{al}}}\arrow{r}
&G_{\mathbb{Q}^{\mathrm{al}}}\arrow{r}
&(clG)_{\mathbb{Q}^{\mathrm{al}}}\\
P\arrow{u}\arrow{rr}&& clG\arrow{u}.
\end{tikzcd}
\]

We need a variant of this. Let $L\subset\mathbb{Q}{}^{\mathrm{al}}$ be a
number field, and let $\Mot^{w}(L)$ be the category of abelian motives over
$L$ with good reduction at $w|L$, i.e., which satisfy the N\'{e}ron condition.
Let $\mathbb{F}{}_{q}\subset\mathbb{F}{}$ be the residue field at $w|L$. The
functor $R\colon\Mot^{w}(\mathbb{Q}{}^{\mathrm{al}})\rightarrow\Mot(\mathbb{F}%
{})$ defines an exact tensor functor%
\[
R\colon\Mot^{w}(L)\rightarrow\Mot(\mathbb{F}_{q})\text{,}%
\]
and hence a morphism of bands%
\[
P(q)\rightarrow G,\quad G\overset{\df}{=}\mathcal{A}{}ut(\omega_{B}%
\mid\Mot^{w}(L)).
\]
As before, this defines a morphism of schemes over $\mathbb{Q}$,
\[
P(q)\rightarrow clG.
\]

Let $M\in\ob\Mot^{w}(L)$, and let $G_{M}=\mathcal{A}{}ut^{\otimes}(\omega
_{B}|\langle M_{\mathbb{Q}{}^{\mathrm{al}}}\rangle^{\otimes})$. Let $M_{0}$
denote the specialization of $M$ in $\Mot(\mathbb{F}{}_{q})$. Let $\ell\neq p$
be a prime number. After possibly replacing $L$ with a finite extension, we
obtain a representation%
\[
\rho_{\ell}\colon\Gal(\mathbb{Q}{}^{\mathrm{al}}/L)\rightarrow G_{M}%
(\mathbb{Q}{}_{\ell}).
\]
Let $\gamma_{\ell}$ denote the image of $\mathrm{Frob}_{w}$ in $(clG_{M}%
)(\mathbb{Q}_{\ell})$.

\begin{theorem}
\label{d9c}There exists a $\gamma\in(clG_{M})(\mathbb{Q})$ that maps to
$\gamma_{\ell}\in(clG_{M})(\mathbb{Q}_{\ell})$, all $\ell\neq p$.
\end{theorem}

\begin{proof}
We have a morphism (over $\mathbb{Q}{}$),%
\[
P(q)\rightarrow clG\rightarrow clG_{M}.
\]
We can take $\gamma$ to be the image of the universal element $\pi
_{\text{univ}}\in$ $P(q)(\mathbb{Q})$.
\end{proof}

\begin{note}
Theorem \ref{d9c} is proved for abelian varieties in \cite{kisinZ2021} after
earlier work of Laskar and Noot. See also \cite{commelin2019}.
\end{note}

\subsection{Integral motives}

Let $k=\mathbb{F}{}_{q}$ or $\mathbb{F}{}$. For the definition of the
categories $\mathsf{R}^{+}(k;\mathbb{\hat{Z})}$ and $\mathsf{R}^{+}%
(k;\mathbb{A}{}_{f})$, we refer the reader to \cite{milneR2004}. We let
$\Mot^{+}(k)$ denote the subcategory of $\Mot(k)$ of effective motives
(triples $(A,e,m)$ with $m\geq0$).

\begin{definition}
\label{d10}The category of effective integral motives $\Mot{}^{+}%
(k,\mathbb{Z})$ over $k$ is the full subcategory of the fibre product category%
\[
\mathsf{R}{}^{+}(k;\mathbb{\hat{Z}}) \underset{\mathsf{R}^{+}(k;\mathbb{A}%
_{f})}{\times}\Mot^{+}(k)
\]
whose objects $(X_{f},X_{0},x_{f})$ are those for which the prime-to-$p$
torsion subgroup of $X_{f}$ is finite.
\end{definition}

Thus, an effective integral motive is a triple $(X_{f},X_{0},x_{f})$
consisting of

\begin{enumerate}
\item an object $X_{f}=(X_{l})_{l}$ of $\mathsf{R}{}^{+}(k;\mathbb{Z}{})$ such
that $X_{l}$ is torsion-free for almost all $l$,

\item as effective motive $X_{0}$, and

\item an isomorphism $x_{f}\colon(X_{f})_{\mathbb{Q}{}}\rightarrow\omega
_{f}(X_{0})$ in $\mathsf{R}{}^{+}(k;\mathbb{A}{}_{f}$).
\end{enumerate}

For $M$ in $\mathsf{R}{}_{p}^{+}(\mathbb{F}{}_{q})$, let $r(M)$ denote the
rank of $M$ and $s(M)$ the sum of the slopes of $M$. Thus, if%
\[
P_{M}(T)=T^{h}+\cdots+c,
\]
then $r(M)=h$ and $s(M)=\ord_{p}(c)/\ord_{p}(q)$.

\begin{theorem}
\label{d11}Let $X$ and $Y$ be effective motives over $\mathbb{F}{}_{q}$ (i.e.,
objects of $\Mot^{+}(\mathbb{F}{}_{q})$). The group $\Ext^{1}(X,Y)$ is finite,
and%
\[
\lim_{s\rightarrow0}\frac{\zeta(X^{\vee}\otimes Y)}{(1-q^{-s})^{\rho(X,Y)}%
}=q^{-\chi(X,Y)}\frac{[\Ext^{1}(X,Y)]\cdot D(X,Y)}{\left[
\Hom(X,Y)_{\mathrm{tors}}\right]  \cdot\left[  \Hom(Y,X)_{\mathrm{tors}%
}\right]  },
\]
where

\begin{itemize}
\item $\chi(X,Y)=s(X_{p})r(Y_{p}),$

\item $D(X,Y)$ is the discriminant of the pairing%
\[
\Hom(Y,X)\times\Hom(X,Y)\rightarrow
\End(Y)\xrightarrow{\mathrm{trace}}\mathbb{Z}{}.
\]

\end{itemize}
\end{theorem}

\begin{proof}
See \cite{milneR2004}, 10.1.
\end{proof}

\begin{aside}
\label{d12}Compare \ref{d11} with the following result (\cite{milne1968a}). If
$A$ and $B$ are abelian varieties over $\mathbb{F}_{q}$, then%
\[
q^{\dim(A)\dim(B)}\prod_{a_{i}\neq b_{j}}\bigg(  1-\frac{a_{i}}{b_{j}}\bigg)
=[\Ext^{1}(A,B)]\cdot D(A,B)
\]
where

\begin{itemize}
\item $(a_{i})_{1\leq i\leq2\dim A}$ and $(b_{i})_{1\leq i\leq2\dim B}$ are
the roots of the characteristic polynomials of the Frobenius endomorphisms of
$A$ and $B$,

\item $D(A,B)$ is the discriminant of the pairing%
\[
\Hom(B,A)\times\Hom(A,B)\rightarrow
\End(B)\xrightarrow{\textrm{trace}}\mathbb{Z}.
\]

\end{itemize}
\end{aside}

\subsection{Almost rational Tate classes}

\begin{plain}
\label{d31}Assuming Conjecture B for CM abelian varietes over $\mathbb{Q}%
{}^{\mathrm{al}}$, we have shown how to construct tannakian categories of
abelian motives $\Mot(\mathbb{F}{}_{p^{n}})$ for all $p$ and $n$. We now
explain how to obtain tannakian categories of abelian motives $\Mot(k)$ for
all fields $k$. For simplicity, we take $k$ to be algebraically closed.
\end{plain}

\begin{plain}
\label{d32}Let $A$ be an abelian variety over $k$. An \emph{almost-RT class}
of codimension $n$ on $A$ is an element $\gamma\in H_{\mathbb{A}{}}%
^{2n}(A)(n)$ such that there exists a cartesian square%
\[
\begin{tikzcd}
X\arrow{d}{f}&A\arrow{l}\arrow{d}\\
S&\Spec(k)\arrow{l}
\end{tikzcd}
\]
and a global section $\tilde{\gamma}$ of $R^{2n}f_{\ast}\mathbb{A}{}(n)$
satisfying the following conditions

\begin{itemize}
\item $S$ is a connected normal scheme of finite type over $\Spec\mathbb{Z}{}$;

\item $f\colon X\rightarrow S$ is an abelian scheme over $S$;

\item the fibre of $\tilde{\gamma}$ over $\Spec(k)$ is $\gamma$, and the
specialization of $\tilde{\gamma}$ at $s$ is rational Tate for all closed
points $s$ in a dense open subset $U$ of $S$.
\end{itemize}

\noindent Note that the residue field $\kappa(s)$ at a closed point of $S$ is
finite, so it makes sense to require $\tilde{\gamma}_{s}$ to be rational Tate.
\end{plain}

\begin{plain}
\label{d33}Let $\Mot(k)$ denote the category of motives based on the abelian
varieties over $k$ using the almost-RT classes as correspondences. Then
$\Mot(k)$ is a tannakian category over $\mathbb{Q}{}$ with many of the
properties anticipated for Grothendieck's category of abelian motives.
\end{plain}

\begin{question}
\label{d34}Does the Tate conjecture for almost-RT classes on abelian varieties
hold over finitely generated fields?
\end{question}

\begin{question}
\label{d35}Let $f\colon A\rightarrow S$ be an abelian scheme over a connected
normal scheme $S$ of finite type over $\mathbb{Z}{}$. Is the set of closed
$s\in S$ such that $\gamma_{s}$ is rational Tate closed?
\end{question}

\begin{proposition}
\label{r21}Assume that \ref{d34} and \ref{d35} have positive answers. Let
$f\colon A\rightarrow S$ be an abelian scheme over a connected normal scheme
of finite type over $\mathbb{F}$, and let $\gamma$ be a global section of
$R^{2n}f_{\ast}\mathbb{A}(n)$. If $\gamma_{s}\in H_{\mathbb{A}{}}^{2n}%
(A_{s})(n)$ is a rational Tate class for one $s\in S(\mathbb{F}{})$, then it
is a rational Tate class for all $s\in S(\mathbb{F})$.
\end{proposition}

\begin{proof}
As in the proof of \ref{c8}, we may replace $\mathbb{F}$ with a finite
subfield. Consider the diagram
\[
\begin{tikzcd}
\mathcal{R}^n(A_{\eta})\arrow{r}{\otimes}
&H^{2n}_{\mathbb{A}}(A_{\eta})(n))^{\pi_{1}(\eta)}\\
\mathcal{R}^n(A/S)\arrow{r}{b}\arrow{u}{\simeq}\arrow[hook]{d}
&H^{0}(S,R^{2n}f_{\ast}\mathbb{A}(n))\arrow{u}{c}\arrow[hook]{d}\\
\mathcal{R}^n(A_{s})\arrow{r}{\otimes} &H^{2n}_{\mathbb{A}}(A_{s})(n)^{\pi_{1}(s)}.
\end{tikzcd}
\]
Because the Tate conjecture holds for rational Tate classes,%
\[
\mathcal{R}^{n}(A_{s})\otimes_{\mathbb{Q}{}}\mathbb{A}{}%
\overset{\lsimeq}{\longrightarrow}H_{\mathbb{A}{}}^{2n}(A_{s})(n)^{\pi_{1}%
(s)}.
\]
Let $\mathcal{R}^{n}(A_{\eta})$ denote the space of almost-RT classes on
$A_{\eta}$. Because we are assuming the Tate conjecture for almost-RT classes
(\ref{d34}),
\[
\mathcal{R}^{n}(A_{\eta})\otimes_{\mathbb{Q}{}}\mathbb{A}{}%
\overset{\lsimeq}{\longrightarrow}H_{\mathbb{A}{}}^{2n}(A_{\eta})(n)^{\pi
_{1}(\eta)}.
\]
Define $R^{n}(A/S)$ to be the space of global sections $\gamma$ of
$R^{2n}f_{\ast}\mathbb{A}(n)$ such that $\gamma_{s}\in H_{\mathbb{A}{}}%
^{2n}(A_{s})(n)$ is a rational Tate class for all $s\in S(\mathbb{F}{}_{q})$.
The map%
\[
\mathcal{R}^{n}(A/S)\rightarrow\mathcal{R}^{n}(A_{\eta})
\]
is injective, and, because of (\ref{d35}), it is surjective. Now we can apply
Lemma \ref{c6} to obtain the equality,%
\[
\mathcal{R}^{n}(A/S)=\mathcal{R}^{n}(A_{s})\cap H^{0}(S,R^{2n}f_{\ast
}\mathbb{A}(n))
\]
(intersection inside $H_{\mathbb{A}}^{2n}(A_{s})(n)$). Thus, if $\gamma_{s}$
is rational Tate for one $s$, it lies in $\mathcal{R}{}^{n}(A/S)$, which means
that $\gamma_{s}$ is rational Tate for all $s$.
\end{proof}

\begin{theorem}
\label{r22}Assume that \ref{d34} and \ref{d35} have positive answers. All
Hodge classes on abelian varieties over $\mathbb{Q}{}^{\mathrm{al}}$ with good
reduction at $w$ specialize to rational Tate classes.
\end{theorem}

\begin{proof}
First prove this for split Weil classes (see the proof of Theorem \ref{c15}).
Then deduce it for Hodge classes on CM abelian varieties (apply \ref{a6}).
Finally, deduce the general case by the argument in \S 6 of \cite{deligne1982}.
\end{proof}

\begin{corollary}
\label{r23}Under the assumptions of the theorem, all Hodge classes on abelian
varieties over fields of characteristic zero are almost-RT.
\end{corollary}

\begin{proof}
For $\mathbb{Q}^{\mathrm{al}}$, this follows from Theorem \ref{r22}. For the
general case, specialize first to $\mathbb{Q}{}^{\mathrm{al}}$.
\end{proof}

\subsection{Comparison with the constructions in \cite{langlandsR1987}}

Recall that we have canonically-defined tannakian categories and quotient
functors,%
\[
\CM(\mathbb{Q}^{\mathrm{al}})\xleftarrow{J}\LCM(\mathbb{Q}^{\mathrm{al}%
})\xrightarrow{R}\LMot(\mathbb{F}).
\]
Let $\omega^{R}$ be the functor on $\LCM(\mathbb{Q}{}^{\mathrm{al}})^{L}$
defined by $R$, so%
\[
\omega^{R}(X)=\Hom(\1,R(X)).
\]
We have the following statement.

\begin{theorem}
\label{d22}There exists a unique $\mathbb{Q}{}$-valued fibre functor
$\omega_{0}$ on $\CM(\mathbb{Q}{}^{\mathrm{al}})^{P}$ such that%
\[
\omega_{0}(J(X))=\omega^{R}(X)
\]
for all $X$ in $\LCM(\mathbb{Q}{}^{\mathrm{al}})$. Moreover, $\omega_{0}$
provides a $\mathbb{Q}{}$-structure for $\omega_{\mathbb{A}}$.
\end{theorem}

Because $\LCM(\mathbb{Q}{}^{\mathrm{al}})^{L}\rightarrow\CM(\mathbb{Q}%
{}^{\mathrm{al}})^{P}$ is a quotient functor, the uniqueness is obvious. That
$\omega_{0}$ is a $\mathbb{Q}$-structure on $\omega_{\mathbb{A}}$ follows from
the fact that $\omega^{R}$ is a $\mathbb{Q}$-structure on $\omega_{\mathbb{A}%
}$. The proof of the existence requires Conjecture B (in fact, is equivalent
to it).

Using cohomology, it is possible to prove only the following weaker result.

\begin{theorem}
\label{d23}There exists a $\mathbb{Q}$-valued fibre functor $\omega$ on
$\CM(\mathbb{Q}^{\mathrm{al}})^{P}$ such that $\omega\otimes\mathbb{Q}{}%
_{l}\approx\omega_{l}$ for all $l$. Any two become isomorphic on any algebraic
subcategory of $\CM(\mathbb{Q}^{\mathrm{al}})$. The set of isomorphism classes
of such $\omega$ is a principal homogeneous space for $\varprojlim
_{\mathcal{F}{}}C(K)$, where $\mathcal{F}{}$ is the set of CM-subfields of
$\mathbb{Q}{}^{\mathrm{al}}$ finite over $\mathbb{Q}{}$ and $C(K)$ is the
ideal class group of $K$.
\end{theorem}

\begin{proof}
As $\CM(\mathbb{Q}^{\mathrm{al}})$ has a canonical fibre functor $\omega_{B}$,
the isomorphism classes of $\mathbb{Q}$-valued fibre functors on
$\CM(\mathbb{Q}^{\mathrm{al}})^{P}$ are classified by the cohomology group
$H^{1}(\mathbb{Q}{},S/P)$. The proof of the existence of $\omega$ occupies a
large part of the article \cite{langlandsR1987}. For the rest, see Theorem 4.1
of \cite{milne2003}.
\end{proof}

We can now choose a $\mathbb{Q}{}$-valued subfunctor $\omega_{0}$ of
$\omega_{\mathbb{A}{}}$ such that $\omega_{0}\otimes_{\mathbb{Q}{}}%
\mathbb{A}{}=\omega_{\mathbb{A}{}}$. We define $\Mot(\mathbb{F}{})$ to be the
quotient $q\colon\CM(\mathbb{Q}^{\mathrm{al}})\rightarrow\Mot(\mathbb{F}{})$
of $\CM(\mathbb{Q}^{\mathrm{al}})$ corresponding to the functor $\omega_{0}$
on $\CM(\mathbb{Q}^{\mathrm{al}})^{P}$ (see \cite{milne2007d}). As
$\CM(\mathbb{Q}^{\mathrm{al}})$ is semisimple, this has an explicit
description (ibid.~2.12 et seq.). Apart from involving a choice, this
definition does not give an object with the wished for properties.

\subsection{Comparison with Grothendieck's categories of motives}

Let $\CM_{\mathrm{num}}(\mathbb{Q}{}^{\mathrm{al}})$ and $\Mot_{\mathrm{num}%
}(\mathbb{F}{})$ be the categories of CM and abelian motives defined using
algebraic cycles modulo numerical equivalence.

\begin{proposition}
\label{d24}The following statements are equivalent:

\begin{enumerate}
\item the functor $\CM_{\mathrm{num}}(\mathbb{Q}^{\mathrm{al}})\rightarrow
\Mot(\mathbb{F}{})$ factors through the functor $\CM_{\mathrm{num}}%
(\mathbb{Q}^{\mathrm{al}})\rightarrow\Mot_{\mathrm{num}}(\mathbb{F})$;

\item an object $M$ of $\Mot(\mathbb{F)}$ is trivial if and only if the
Frobenius element $\pi_{M}=1$;

\item the Tate conjecture holds for all abelian varieties over $\mathbb{F}{}$.
\end{enumerate}
\end{proposition}

\begin{proof}
For the equivalence of (a) and (b), see Milne 2007d. For the equivalence of
(b) and (c), see Geisser 1998.
\end{proof}

\subsection{Comparison with Andr\'{e}'s categories}

\begin{plain}
\label{d27}Fix a prime number $\ell\neq p$, and let $\Mot(\mathbb{F}%
;\mathbb{Q}{}_{\ell})$ denote the $\mathbb{Q}_{\ell}$-linear category based on
abelian varieties and using Tate classes as correspondencees. For some
countable subfield $Q$ of $\mathbb{Q}_{\ell}$, Andr\'{e} defines a $Q$-linear
category $\Mot(\mathbb{F};Q)$ based on abelian varieties and using motivated
classes as correspondences. The are canonical exact tensor functors%
\[
\Mot(\mathbb{F}{})\rightarrow\Mot(\mathbb{F};Q{})\rightarrow\Mot(\mathbb{F}%
;\mathbb{Q}{}_{\ell})
\]
such that%
\begin{align*}
\Mot(\mathbb{F}{})_{(Q)}  &  \rightarrow\Mot(\mathbb{F};Q{})\\
\Mot(\mathbb{F};Q{})_{\mathbb{(\mathbb{Q}{}}_{\ell})}  &  \rightarrow
\Mot(\mathbb{F};\mathbb{Q}{}_{\ell}{})
\end{align*}
are equivalences of tensor categories.

To prove this, note that Andr\'{e} (2006a, 2006b)\nocite{andre2006a,
andre2006b} shows that motivated classes on abelian varieties satisfy the
conditions (R1), (R2), (R3), and (R4) of Conjecture C. Since the spaces
$\mathcal{R}{}^{\ast}(A)\otimes_{\mathbb{Q}}Q$ also satisfy these conditions,
the argument in the proof of Theorem \ref{f2} shows that the two families coincide.
\end{plain}

\begin{summary}
\label{d29}Every rational Tate class on an abelian variety over $\mathbb{F}$
becomes motivated over $Q$, and the space of rational Tate classes is a
$\mathbb{Q}$-structure on the $Q$-space of motivated classes.
\end{summary}

\section{Shimura varieties of abelian type}

\label{SVA}

In this section, we assume that there exists a commutative diagram
\begin{equation}
\begin{tikzcd} \LMot^{\mathscr{w}}(\mathbb{Q}^{\mathrm{al}})\arrow{d}{R}\arrow{r} &\Mot^{\mathscr{w}}(\mathbb{Q}^{\mathrm{al}})\arrow{d}{R}\arrow{rd} {\xi_f}\\ \LMot(\mathbb{F})\arrow{r}\arrow[bend right=20]{rr}[swap]{\xi_f} &\Mot(\mathbb{F})\arrow{r} &\Mot(\mathbb{F};\mathbb{A}_{f})\end{tikzcd} \label{ep1}%
\end{equation}
such that the functor%
\[
\Mot(\mathbb{F})_{(\mathbb{A}_{f})}\longrightarrow\Mot(\mathbb{F};\mathbb{A}%
{}_{f})
\]
is a tensor equivalence, and we investigate its applications to Shimura varieties.

\subsection{Introduction}

\label{g}

\begin{plain}
\label{g1}Since the 1970s, Deligne has been promoting the idea that Shimura
varieties with rational weight should be thought of as moduli schemes for
motives with additional structure. Indeed this is a powerful tool for
discovery, which has been used most prominently by Langlands in his work on
understanding the zeta functions of Shimura varieties. In his Corvallis
article (1979),\nocite{langlands1979} Langlands applied it to find a
conjectural description of the conjugate of a Shimura variety --- this is
needed to compute the factors at infinity of the zeta function. In his article
with Rapoport (1987),\nocite{langlandsR1987} Langlands applied it to find a
conjectural description of the points of the Shimura variety modulo $p$ ---
this is needed to compute the factors at the finite places of the zeta function.
\end{plain}

\begin{plain}
\label{g2}Recall that the Shimura varieties of abelian type are, by
definition, exactly those for which Deligne proved the existence of a
canonical model in his Corvallis article (1979).\nocite{deligne1979} Let
$(G,X)$ be a Shimura datum of abelian type with rational weight, and chose an
algebra $(V,t)$ over $\mathbb{Q}{}$ such that $G=\mathcal{A}{}ut(V,t)$%
.\footnote{See \cite{milne2020b}. By an algebra over $\mathbb{Q}$ we mean a
finite-dimensional $\mathbb{Q}$-vector space $V$ together with a linear map
$V\otimes V\to V$ (no conditions). Readers may prefer to take any $V$ and
family of tensors determining $G$.} Such a choice realizes $\Sh(G,X)$ as a
moduli scheme (in the category of complex analytic spaces) for polarizable
rational Hodge structures with algebra and level structure. It follows from
Theorem \ref{a1} that the pair $(G,X)$ is of abelian type if and only if the
Hodge structures in the moduli family are in the essential image of the (fully
faithful) Betti functor%
\[
\Mot(\mathbb{C})\rightarrow\Hdg_{\mathbb{Q}{}}.
\]
From the theorem of Borel (1972), it follows that, when $(G,X)$ is of abelian
type, this realization becomes a realization of $\Sh(G,X)$ as a moduli scheme
(in the category of algebraic schemes over $\mathbb{C}$) for abelian motives
with algebra and level structure. From this modular realization, it is
possible to read off a proof Langlands's conjecture on conjugates. Elementary
descent theory gives a proof of the existence of canonical models that is both
simpler and more natural than the original --- the Shimura variety is defined
over the number field because the moduli probem is defined over the number
field. Moreover, this approach provides a \textit{description} of the
canonical model as a moduli scheme, whereas Deligne's original approach (1979)
provides only a \textit{characterization} of it in terms of reciprocity laws
at the special points.
\end{plain}

\begin{plain}
\label{g3}When the existence of the diagram (\ref{ep1}, p.~\pageref{ep1}) is
assumed, the theory outlined above extends to characteristic $p$.
Specifically, suppose that $G$ is unramified at $p$, and suppose that the
$\mathbb{Q}$-algebra $(V,t)$ is chosen to satisfy the condition 3.2.3 of
\cite{kisin2020}. The moduli problem over the reflex field can be extended
over its ring of integers, and the corresponding moduli scheme is smooth. This
gives us a smooth integral model of $\Sh(G,X)$ and a modular interpretation of
its functor of points. The modular description of the points with coordinates
in $\mathbb{F}$ can be regarded as a categorification of the conjectural
description in Langlands and Rapoport (1987). To obtain their original
description is an exercise in tannakian theory. Besides the integral model of
the Shimura variety, one obtains in this way an integral model of the standard
principal bundle, and hence integral models of the automorphic vector bundles
on the Shimura variety.
\end{plain}

\subsection{Characteristic zero}

\begin{plain}
\label{g7}Recall that, for a field $k$ of characteristic zero, $\Mot(k)$
denotes the category of abelian motives over $k{}$ (defined using absolute
Hodge classes). Let $S$ be a connected smooth algebraic scheme over a field
$k$ of characteristic zero, and let $\eta$ be its generic point. We define an
abelian motive $M$ over $S$ to be an abelian motive $M_{\eta}$ over $k(\eta)$
such that the action of $\pi_{1}(\eta,\bar{\eta})$ on $\omega_{f}(M)$ factors
through $\pi_{1}(S,\bar{\eta})$. See \cite{milne1994b}, 2.37.\footnote{For
hints on how to extend the definition to nonsmooth schemes, see ibid.~2.45.}
\end{plain}

\begin{plain}
\label{g4}Let $(G,X)$ be a Shimura datum. In order for the Shimura variety
$\Sh(G,X)$ to be a moduli variety for motives, it is necessary that every
special point be CM. This is true when $(G,X)$ satisfies the conditions:

\begin{enumerate}
\item the central character $w_{X}$ is defined over $\mathbb{Q}$ and

\item the connected centre of $G$ is split by a CM field.
\end{enumerate}

\noindent See \cite{milne1988a}, A.3. \noindent From now on, we always assume
the condition (b).\footnote{It is the author's view that pairs $(G,X)$ failing
(b) are pathological and should be excluded, but Deligne disagrees.}
\end{plain}

\begin{plain}
\label{g5}Let $(G,X)$ be a Shimura datum such that $w_{X}$ is rational. Let
$K$ be a (small) compact open subgroup of $G(\mathbb{A}_{f})$, and let $(V,t)$
be an algebra such that $G=\mathcal{A}ut(V,t)$.\footnote{See \cite{milne2020b}%
. Here and elsewhere, the reader may prefer to take any vector space $V$ and
family of tensors determining $G$.} Then $\Sh_{K}(G,X)$ is the solution of a
moduli problem $\mathcal{H}_{K}$ on the category of smooth algebraic schemes
over $\mathbb{C}$. More precisely, for such a scheme $S$, $\Sh_{K}(G,X)(S)$
classifies certain triples $(\mathbb{V},t,\eta)$, where $\mathbb{V}$ is a
variation of Hodge structures on $S$, $t\colon\mathbb{V}\otimes\mathbb{V}%
\rightarrow\mathbb{V}$ is an algebra structure on $\mathbb{V}$, and $\eta$ is
a $K$-level structure. If the largest $\mathbb{R}$-split torus in $Z(G)$ is
already split over $\mathbb{Q}$, then $\Sh_{K}(G,X)$ is a fine moduli scheme.
See \cite{milne1994b}, 3.10, 3.11.
\end{plain}

\begin{plain}
\label{g6}The elements of $\mathcal{H}{}_{K}(\mathbb{C}{})$ (in \ref{g5}) are
the Betti realizations of abelian motives if and only if $(G,X)$ is of abelian
type.\footnote{This can be proved by comparing Deligne's characterization of
the Shimura varieties he considers with the characterization of the algebraic
quotients of $G$ in Theorem \ref{a1}.} When this is the case, $\Sh_{K}(G,X)$
is a moduli scheme over $\mathbb{C}{}$ for abelian motives with additional
structure, and a fine moduli scheme if $Z(G)$ satisfies the condition in
\ref{g5}. See \cite{milne1994b}, 3.13.
\end{plain}

\begin{plain}
\label{g8}Let $(G,X)$ be a Shimura datum of abelian type with rational weight.
From (\ref{g6}), it is possible to read off a proof Langlands's conjugation
conjecture, except with the Taniyama group in place of Langlands's group. See
\cite{milne1990}, 4.2. To complete the proof, one needs to use that the two
groups are equal (\ref{b13a}).
\end{plain}

\begin{plain}
\label{g9}Let $(G,X)$ be a Shimura datum of abelian type with rational weight,
and let $F$ be a number field such that the moduli problem in \ref{g6} is
defined over $F$. Then an elementary descent argument (\cite{milne1999c})
shows that $\Sh_{K}(G,X)$ has model over $F$ that is a solution to the moduli
problem. When $F\subset\mathbb{C}$ is the reflex field, we get the canonical
model of $\Sh_{K}(G,X)$ in the original sense of \cite{deligne1979};
otherwise, we get the canonical model in the sense of \cite{semplinerT2025}.
\end{plain}

\begin{plain}
\label{g10}Let $(G,X)$ be a Shimura datum, and let $G^{c}$ denote the quotient
of $G$ by the largest subtorus of $Z(G)$ that is split over $\mathbb{R}{}$ but
has no subtorus that is split over $\mathbb{Q}$. Let%
\[
P(G,X)=G(\mathbb{Q}{})\backslash X\times G^{c}(\mathbb{C}{})\times
G(\mathbb{A}{}_{f})/G(\mathbb{Q}{})^{-}\text{.}%
\]
It is principal bundle $G^{c}$-bundle with a flat connection, called the
standard principal bundle. There is a canonical equivariant morphism
$\gamma\colon P(G,X)\rightarrow X^{\vee}$. The automorphic vector bundles are
obtained as follows: start with a $G^{c}$-vector bundle on $X^{\vee}$, pull it
back to $P(G,X)$, and descend it to $\Sh(G,X)$. See \cite{milne1990}, III.

Now assume that $(G,X)$ is of abelian type and is a fine moduli scheme for
abelian motives with additional structure (\ref{g6}). The system%
\begin{equation}
\Sh_{K}(G,X)\longleftarrow P_{K}(G,X)\longrightarrow X^{\vee} \label{ep6}%
\end{equation}
can be re-constructed from the universal abelian motive over $\Sh_{K}(G,X)$
(cf.~\cite{milne1990}, 3.3). From this, we can read off

\begin{enumerate}
\item a description of the conjugate of the entire system (\ref{ep6}) by an
automorphism of $\mathbb{C}{}$, extending Langlands's description of the
conjugate of $\Sh_{K}(G,X)$ (but necessarily expressed in terms of the period torsor);

\item the existence of a canonical model of the system (\ref{ep6}), extending
the existence of the canonical model of $\Sh(G,X)$.
\end{enumerate}
\end{plain}

\begin{plain}
\label{g14}For an arbitrary Shimura variety $\Sh(G,X)$ of abelian type, not
necessarily with rational weight, there are morphisms%
\[
\Sh(G,X)\times\Sh(Z_{\ast},\epsilon)\rightarrow\Sh(G_{\ast},X_{\ast}),
\]
where $Z_{\ast}$ is a torus and $\Sh(G_{\ast},X_{\ast})$ is of abelian type
with rational weight. These can be used to deduce statements about $\Sh(G,X)$
from statements about $\Sh(G_{\ast},X_{\ast})$. See \cite{milne1994b}, 3.33--3.37.
\end{plain}

\begin{note}
For more details, see \cite{milne1990}, II,3; \cite{milne1994b}, \S 3;
\cite{milne2013b}.
\end{note}

\subsection{Mixed characteristic}

\begin{plain}
[Serre-Tate]\label{g15}Let $S=\Spec R$ be an artinian local scheme with
(closed) point $s$ such that $k\overset{\df}{=}\kappa(s)$ has characteristic
$p>0$. The functor $A\rightsquigarrow(A_{s},T_{p}A,\id)$ is an equivalence
from the category of abelian schemes over $S$ to the category of triples
$(A_{0},X,\varphi)$, where $A_{0}$ is an abelian variety over $k$, $X$ is a
$p$-divisible group over $S$, and $\varphi$ is an isomorphism $X_{s}%
\rightarrow T_{p}(A_{0}$).
\end{plain}

\begin{plain}
[Fontaine]\label{g16}Let $S$ be an artinian local scheme, as in \ref{g15}. The
functor sending a $p$-divisible group $X$ over $S$ to $(L,M)$, where $M$ is
the covariant Dieudonn\'{e} module of $X_{s}$ and $L$ is an $R$-submodule of
$M\otimes R$ such that $L\otimes_{R}k=VM/pM$, is an equivalence of categories.
\end{plain}

\begin{plain}
\label{g17}Let $S$ be an artinian local scheme, as in \ref{g15}, but with $k$
equal $\mathbb{F}_{q}$ or $\mathbb{F}$. On combining the last two statements,
we see that, to give an abelian scheme over $S$ is equivalent to giving an
abelian variety $A_{0}$ over $k$ and a lifting of the filtration on the
covariant Dieudonn\'{e} module of $A_{0}$. This suggests \textit{defining} an
abelian motive over $S$ to be an abelian motive $M$ over $k$ (object of
$\Mot(k))$ and a lifting of the filtration on the crystalline homology groups
of $M$.
\end{plain}

\begin{plain}
\label{g17a} There is a similar statement (and definition) when $S$ is the
spectrum of a complete noetherian local ring with residue field $\mathbb{F}%
_{q}$ or $\mathbb{F}$.
\end{plain}

\begin{plain}
\label{g18}More generally, we define an abelian motive over a perfectoid space
$S$ over $\mathbb{F}{}$ to be a triple $(M_{0},X,\varphi)$, where $M_{0}$ is
an abelian motive over $\mathbb{F}{}$, $X$ is a $p$-adic shtuka over $S$, and
$\varphi$ is an isomorphism from $X_{s}$ to the $p$-adic shtuka of $M_{0}$.
\end{plain}

\begin{plain}
\label{g19}Let $(G,X)$ be a Shimura datum of abelian type satisfying the
conditions (a) and (b) of \ref{g4} (to be a moduli scheme). Assume that $G$ is
unramified at $p$, and let $\Sh_{p}(G,X)=\Sh_{K^{p}K_{p}}(G,X)$ with $K_{p}$
hyperspecial. As explained above, when we write $G=\mathcal{A}{}ut(V,t)$, then
we obtain a model of $\Sh_{p}(G,X)$ over the reflex field $E$ and a
description of it as a moduli scheme for abelian motives with additional
structure. Now assume that $(V,t)$ can be chosen to satisfy the condition
3.2.3 of \cite{kisin2020}. Then the moduli problem extends to schemes over
$\mathcal{O}{}_{w}$, where $\mathcal{O}{}$ is the ring of integers in $E$, and
standard methods show that it has a solution that is a smooth scheme over
$\mathcal{O}{}_{w}$. In this way, we get

\begin{enumerate}
\item an integral canonical model of $\Sh_{p}(G,X)$ over $\mathcal{O}{}_{w}$;

\item a description of the model as a moduli variety for abelian motives with
algebra and level structure (in particular, a categorification of the
conjecture of Langlands and Rapoport);

\item an integral model of the system (\ref{ep6}), extending that of
$\Sh_{p}(G,X)$;

\item an integral theory of automorphic vector bundles on $\Sh_{p}(G,X)$.
\end{enumerate}
\end{plain}

\begin{summary}
In conclusion, our theory of motives allows us to identify the functor of
points of a Shimura variety of abelian type with rational weight (which would
delight Grothendieck), but for those without rational weight we are forced to
apply the trick of Shimura (which would delight Serre).
\end{summary}

\section{Shimura varieties not of abelian type}

\label{SVN}

What about Shimura varieties not of abelian type? It remains an open, and very
interesting question, whether the polarizable Hodge structures arising from
all Shimura varieties are motivic. Absent a proof of that, we are stuck with
the old methods for proving Langlands's conjugacy conjecture and the existence
of canonical models. Concerning these, here is my response (June 14, 2025) to
a query from Richard Taylor (slightly edited for clarity).

\bigskip Rather than revisiting the ad hoc methods I use in \S 6 of my 1983
paper to prove compatibility for different special points, I think one should
instead use the following beautiful result of Borovoi.

\begin{theoremn}
Let $G$ be a simply connected semisimple algebraic group over a totally real
algebraic number field $F$. Assume that $G$ has an anisotropic maximal torus
$T$ that splits over some totally imaginary quadratic extension $K$ of the
field $F$. Let $\Pi$ be a base of the root system $R=R(G_{K},T_{K})$. Then
$G(F^{\mathrm{rc}})$ is generated by the subgroups $G_{\alpha}(F^{\mathrm{rc}%
})$, $\alpha\in\Pi$ (here $F^{\text{rc}}$ is a totally real closure of $F$).
\end{theoremn}

\begin{theoremn}
Under the conditions of Theorem 1, assume that $G$ is a geometrically simple
group of totally hermitian type that is not totally compact. Then
$G(F^{\mathrm{rc}})$ is generated by the subgroups $G_{\alpha}(F^{\mathrm{rc}%
})$, $\alpha\in R^{\mathrm{rtc}}$.
\end{theoremn}

In his 1983/84 paper, Borovoi made use of a stronger statement, which still
hasn't been proved, but later he did prove Theorems 1 and 2 with the help of
his Russian colleagues. See,

\begin{quote}
The group of points of a semisimple group over a totally real closed field
Borovoi, M. V., Selecta Math. Soviet. 9 (1990), no. 4, 331--338.
\end{quote}

In the last two sections of my 1988 Inventiones paper, I gave two proofs of
the compatibility, one with and one without Borovoi's statement.

\bigskip But, as I mentioned briefly at the [Tate 100] conference, I think the
whole business of canonical models (in the general case) needs to be rethought.

At present we

\begin{enumerate}
\item use Kazhdan's theorem that conjugates of arithmetic varieties are
arithmetic varieties (arithmetic variety = bsd/arithmetic group);

\item deduce the conjugation theorem for Shimura varieties (Borovoi--Milne);

\item prove the conjugation theorem for the standard principal bundle (Milne,
1988, Inventiones).
\end{enumerate}

What Nori and Ragunathan (1993) show is that Kazhdan asked the wrong question.
Let $D$ be a bounded symmetric domain and $G$ a real algebraic group such that
$G(\mathbb{R}{})$ acts transitively on $D$ with compact isotropy groups. From
an arithmetic $\Gamma$ we get a system
\[
(Y,P,\nabla,D^{\vee},\gamma),\qquad Y\leftarrow(P,\nabla)\xrightarrow{\gamma}
D^{\vee},
\]
where $Y=\Gamma\backslash D$, $P$ is the principal bundle $\Gamma\backslash
D\times G(\mathbb{C})$, $\nabla$ is a flat connection, $D^{\vee}$ is the
compact dual, and $\gamma\colon P(\Gamma)\rightarrow D^{\vee}$ is defined by
the Borel map. This system is algebraic, and the correct question to ask is
that the conjugate of such a system be again such a system. Nori and
Raghunathan characterize such systems and show that the characterizing
properties are preserved under conjugation. This is much much simpler than Kazhdan.

From a Shimura datum $(G,X)$, we get a similar system
\[
(S,P,\nabla,X^{\vee},\gamma),\qquad S\leftarrow(P,\nabla
)\xrightarrow{\gamma}X^{\vee}
\]
and I think, similarly, that one should work directly with such systems
instead of just the Shimura variety. This should make everything simpler ---
the conjugation conjecture, canonical models, and even integral canonical
models. I talked about this at the Borel conference at Hangzhou in 2004, but
haven't worked out the details.

\bigskip A little history. Langlands made little progress in understanding the
zeta functions of Shimura varieties until Deligne explained to him his axioms
(especially $h$!) and that he should think of them as moduli varieties of
motives. Langlands stated his Corvallis conjecture in order to understand the
factors of the zeta function at infinity, and his conjecture with Rapoport to
understand the factors at finite places. When I asked Langlands how he came up
with the \textquotedblleft cocycle\textquotedblright\ for the conjugation
conjecture, he just said that it was the only thing he could think of. When he
explained it to Deligne, Deligne realized that it conjecturally gave an
explicit description of the Taniyama group, something that he and others had
been searching for.

\section{Mixed Shimura varieties}

\label{MSV}

To be continued.

\section{Appendix: The cohomological approach.}

Tannakian categories and the functors between them are classified by
cohomology sets. In particular, our conjectures predicts the existence of
certain cohomology classes with particular properties. Here, we briefly look
at three examples. The results are positive, but inconclusive, and do not seem
to be helpful in proving the existence of the categories of motives.

\subsection{Tannakian theory and cohomology (0)}

Recall (Theorem \ref{p1}) that we have a commutative diagram of bands%
\[
\begin{tikzcd}
G\arrow{r}&G_{\mathbb{Q}^{\mathrm{al}}}\\
P\arrow{r}\arrow{u}&P_{\mathbb{Q}^{\mathrm{al}}}\arrow{u}.
\end{tikzcd}
\]
A necessary condition for this to extend to a commutative diagram of tannakian
categories%
\[
\begin{tikzcd}
\Mot^{w}(\mathbb{Q}^{\mathrm{al}})\arrow{r}\arrow{d}
&\Mot(\mathbb{Q}^{\mathrm{al}};\mathbb{Q}_{l})\arrow{d}\\
\Mot(\mathbb{F})\arrow{r}&\Mot(\mathbb{F};\mathbb{Q}_{l})
\end{tikzcd}
\]
is that the class of $\Mot(\mathbb{F}{})$ in $H^{2}(\mathbb{Q}{},P)$ map to
the trivial class in $H^{2}(\mathbb{Q}{},G)$. Certainly, it maps to the
trivial class in $H^{2}(\mathbb{Q}{}_{l},G)$ for all $l$ (including
$l=p,\infty$), so the necessary condition holds if the Hasse principle holds
for $H^{2}(\mathbb{Q}{},G)$. There is an exact sequence%
\[
1\rightarrow G^{\mathrm{der}}\rightarrow G\rightarrow S\rightarrow1.
\]
For any CM field $K$, $H^{2}(G,S^{K})$ satisfies the Hasse principle and we
can apply the following theorem of Douai to the algebraic quotients of
$G^{\mathrm{der}}$.

\begin{theorem}
\label{s0} Let $G$ be a semisimple algebraic group over a field $k$. The
gerbes over $\Aff_{k}$ locally bound by $G$ are trivial in each of the
following cases:

\begin{enumerate}
\item $k$ is a local field with finite residue field;

\item $k$ is a number field with no real prime;

\item $k$ is a global field and $G$ is simply connected;

\item $k$ is a number field and $G$ is a semisimple group of type $^{2}%
\!A_{n}$, $B_{n}$, $C_{n}$, $^{1}D_{2n}$, $^{2}D_{2n+1}$, $^{2}E_{6}$, $E_{7}%
$, $E_{8}$, $F_{4}$, or $G_{2}$.
\end{enumerate}
\end{theorem}

\begin{proof}
See \cite{douai1975}.
\end{proof}

\subsection{Tannakian theory and cohomology (1)}

We apply the theory of tannakian categories to the problem of extending the
reduction functor from $\CM(\mathbb{Q}{}^{\mathrm{al}})$ to $\Mot^{w}%
(\mathbb{Q}{}^{\mathrm{al}})$. We begin with a preliminary.

\begin{proposition}
\label{s1}Let $G$ be a semisimple algebraic group over a field $k$ and $T$ a
torus in $G$. If $G$ is simply connected, then so also is the derived group of
$C_{G}(T)$.
\end{proposition}

\begin{proof}
See, for example, \cite{conrad2014}, 6.5.2(iv).
\end{proof}

Let $\mathsf{T}{}$ be a tannakian category over a field $k$. Then%
\[
\Fib(\mathsf{T}{})\overset{\df}{=}\textsc{Hom}(\mathsf{T}{},\Vc_{k})
\]
is a gerbe over $\Aff_{k}$. We call its band $G$ the band of $\mathsf{T}{}$.
For any fibre functor $\omega$ of $\mathsf{T}{}$, $G$ is represented by
$\mathcal{A}{}ut^{\otimes}(\omega)$ equipped with its natural descent datum up
to inner automorphisms. The gerbe $\Fib(\mathsf{T}{})$ defines a class in
$H^{2}(k,G)$, which is trivial if and only if $\mathsf{T}{}$ is neutral.

More generally, let $\mathsf{C}{}$ and $\mathsf{D}$ be tannakian categories
over $k$, and let $u\colon H\rightarrow G$ be a morphism from the band of
$\mathsf{D}{}$ to that of $\mathsf{C}{}$. The morphisms of tannakian
categories (exact tensor functors) $\mathsf{C}{}\rightarrow\mathsf{D}{}$
banded by $u$ form a gerbe%
\[
\textsc{Hom}_{u}(\mathsf{C}{},\mathsf{D}{})
\]
over $\Aff_{k}$. Its band is the centralizer $C_{u}$ of $u$ (apply
\cite{giraud1971}, IV, 2.3.2, to $\Fib(\mathsf{C}{})$ and $\Fib(\mathsf{D}{}%
)$). The gerbe $\textsc{Hom}_{u}(\mathsf{C}{},\mathsf{D}{})$ defines a class
in $H^{2}(k,C_{u})$, which is trivial if and only if there exists a morphism
of tannakian categories $\mathsf{C}{}\rightarrow\mathsf{D}$ banded by $u$.

Now let $\Mot^{\prime}(\mathbb{Q}{}^{\mathrm{al}})$ denote the tannakian
subcategory of $\Mot^{w}(\mathbb{Q}{}^{\mathrm{al}})$ generated by the abelian
varieties in $\Mot^{w}(\mathbb{Q}{}^{\mathrm{al}})$ whose Mumford--Tate group
has simply connected derived group (cf. Theorem \ref{a2}). The band $G$ of
$\Mot^{\prime}(\mathbb{Q}{}^{\mathrm{al}})$ has simply connected derived
group. Let $u\colon P\rightarrow G$ be the morphism of bands underlying the
canonical morphisms $P_{\mathbb{Q}{}_{l}}\rightarrow G_{\mathbb{Q}{}_{l}}$
(see \ref{p1}). Consider the commutative diagram of gerbes over $\Aff_{k}$,%
\[
\begin{tikzcd}
\textsc{Hom}_{u}(\Mot^{\prime}(\mathbb{Q}^{\mathrm{al}}),\Mot(\mathbb{F}))\arrow{r}\arrow{d}
&\textsc{Hom}_{v}(\CM(\mathbb{Q}^{\mathrm{al}}),\Mot(\mathbb{F})\arrow{d}\\
\textsc{Hom}_{u}(\Mot^{\prime}(\mathbb{Q}{}^{\mathrm{al}}),\mathsf{V}{}_{l}(\mathbb{F}))\arrow{r}
&\textsc{Hom}_{v}(\CM(\mathbb{Q}^{\mathrm{al}}),\mathsf{V}_{l}(\mathbb{F}{})).
\end{tikzcd}
\]
Here $v\colon P\rightarrow S$ is the Shimura--Taniyama homomorphism, the
horizontal morphisms are defined by the inclusion $\CM(\mathbb{Q}%
{}^{\mathrm{al}})\hookrightarrow\Mot^{\prime}(\mathbb{Q}^{\mathrm{al}})$, and
the vertical morphisms are defined by the exact tensor functors $\eta
_{l}\colon\Mot(\mathbb{F}{})\rightarrow\mathsf{V}_{l}(\mathbb{F}{})$. We have
an object $R\colon\CM(\mathbb{Q}{}^{\mathrm{al}})\rightarrow\Mot(\mathbb{F}%
{})$ in the gerbe at top-right and an object $\xi_{l}\colon\Mot^{\prime
}(\mathbb{Q}{}^{\mathrm{al}})\rightarrow\mathsf{V}{}_{l}(\mathbb{F}{})$ in the
gerbe at bottom-left, which map to the same object in the gerbe at
bottom-right, and we seek an object in the gerbe at top-left that maps to both
$R$ and $\xi_{l}$ (for all $l$). Let $C=C_{u}$ (centralizer of $u$). Then $C$
has a filtration whose quotients are%
\[
C^{\mathrm{der}},\quad\text{a torus,}\quad S.
\]
According to Proposition \ref{s1}, $C^{\mathrm{der}}$ is simply connected.

\textit{To simplify things. we now pretend that the torus is trivial. Also, we
ignore some of the complexities of nonabelian cohomology. What follows is only
heuristic.}

\begin{proposition}
[?]\label{s4}There exists a morphism of tannakian categories $R^{\prime}%
\colon\Mot^{\prime}(\mathbb{Q}{}^{\mathrm{al}})\rightarrow\Mot(\mathbb{F}{})$
such that

\begin{enumerate}
\item $R^{\prime}|\CM(\mathbb{Q}{}^{\mathrm{al}})=R$,

\item $R^{\prime}\otimes\mathbb{Q}{}_{l}\approx\xi_{l}$ for all $l$.
\end{enumerate}
\end{proposition}

\begin{proof}
Consider the maps%
\[
H^{2}(\mathbb{Q},P)\rightarrow{}H^{2}(\mathbb{Q}{},C)\rightarrow
H^{2}(\mathbb{Q}{},S).
\]
Because $R\colon\CM(\mathbb{Q}^{\mathrm{al}})\rightarrow\Mot(\mathbb{F}{})$
exists, the class of $\Mot(\mathbb{F}{})$ in $H^{2}(\mathbb{Q}{},P)$ maps to
\textquotedblleft zero\textquotedblright\ in $H^{2}(\mathbb{Q}{},S)$, but, by
Douai's theorem, the map ${}H^{2}(\mathbb{Q}{},C)\rightarrow H^{2}%
(\mathbb{Q}{},S)$ is injective, and so it maps to zero in $H^{2}(\mathbb{Q}%
{},C)$. Thus, $R$ extends to a morphism of tannakian categories $\Mot^{\prime
}(\mathbb{Q}{}^{\mathrm{al}})\rightarrow\Mot(\mathbb{F}{})$. Choose one
extension $R^{\prime}$. On comparing $R^{\prime}\otimes\mathbb{Q}{}_{l}$ to
$\xi_{l}$, we get a class $a_{l}\in H^{1}(\mathbb{Q}{}_{l},C)$. The (fake)
exact sequence%
\[
0\rightarrow C^{\mathrm{der}}\rightarrow C\rightarrow S\rightarrow0,
\]
gives rise to an exact commutative diagram%
\[
\begin{tikzcd}
H^{1}(\mathbb{Q},C^{\mathrm{der}})\arrow{r}\arrow{d}{\simeq}
&H^{1}(\mathbb{Q},C)\arrow[two heads]{r}\arrow{d}
&H^{1}(\mathbb{Q},S)\arrow{d}{\simeq}\\
H^{1}(\mathbb{R},C^{\mathrm{der}})\arrow{r}
&\prod_{l}H^{1}(\mathbb{Q}_{l},C)\arrow{r}
&\prod_{l}H^{1}(\mathbb{Q}_{l},S),
\end{tikzcd}
\]
in which the product signs mean restricted product. As $C^{\mathrm{der}}$ is
simply connected, $H^{1}(\mathbb{Q}{}_{l},C^{\mathrm{der}})=0$ for all
$l\neq\infty$ and the first vertical arrow is an isomorphism (e.g.,
\cite{milne2017c}, 25.61, 25.63). Moreover, the map $H^{1}(\mathbb{Q}%
{},C)\rightarrow H^{1}(\mathbb{Q}{},S)$ is surjective by the theorem of Douai,
and the map $H^{1}(\mathbb{Q}{},S)\rightarrow\prod\nolimits_{l}H^{1}%
(\mathbb{Q}{}_{l},S)$ is an isomorphism, at least on the finite level (see
\cite{milne2003}, \S 3). Now a diagram chase starting from $(a_{l})_{l}\in$
$\prod\nolimits_{l}H^{1}(\mathbb{Q}{}_{l},C)$ gives us an element of
$H^{1}(\mathbb{Q}{},C)$ that can be used to modify $R^{\prime}$ to obtain a
morphism with the required properties.
\end{proof}

Note that Proposition \ref{s4} is not what we need. Specifically, we need the
statement with $\Mot^{w}(\mathbb{Q}{}^{\mathrm{al}})$ for $\Mot^{\prime
}(\mathbb{Q}{}^{\mathrm{al}})$ and, more crucially, with $R^{\prime}%
\otimes\mathbb{Q}{}_{l}=\xi_{l}$ (not isomorphic).

\begin{remark}
\label{s5}It would be better to fix a \textquotedblleft
large\textquotedblright\ CM subfield $K$ of $\mathbb{C}{}$, finite and Galois
over $\mathbb{Q}{}$, and consider only abelian varieties over $\mathbb{Q}%
{}^{\mathrm{al}}$ such that $\End^{0}(A)\otimes_{\mathbb{Q}{}}K$ is a product
of matrix algebras over $K$. Also, things simplify when we consider only
abelian varieties with good ordinary reduction.
\end{remark}

\begin{exercise}
\label{s6}Strengthen the above argument

\begin{enumerate}
\item to take account of the torus;

\item to obtain a morphism $\Mot^{w}(\mathbb{Q}{}^{\mathrm{al}})\rightarrow
\Mot(\mathbb{F}{})$ whose image in HOM$_{v}(\CM(\mathbb{Q}{}^{\mathrm{al}%
}),\Mot(\mathbb{F}{}))$ is equal to $R$ and whose image in HOM$_{u}%
(\Mor^{\prime}(\mathbb{Q}{}^{\mathrm{al}}),\mathsf{V}{}_{l}(\mathbb{F}{}))$ is
equal to $\xi_{l}$.
\end{enumerate}
\end{exercise}

\begin{exercise}
\label{s7}We are given a morphism of tannakian categories over $\mathbb{A}{}$,%
\[
\Mot^{w}(\mathbb{Q}{}^{\mathrm{al}})_{(\mathbb{A}{})}\rightarrow
\Mot(\mathbb{F}{})_{(\mathbb{A}{})}\text{.}%
\]
Use a super-version of Lemma \ref{p2} to deduce that this morphism arises from
a morphism of tannakian categories over $\mathbb{Q}{}$,%
\[
\Mot^{w}(\mathbb{Q}{}^{\mathrm{al}})\rightarrow\Mot(\mathbb{F}{}).
\]

\end{exercise}

\begin{note}
\label{s8}Results obtained using cohomology, while comforting --- if
Proposition \ref{s4} were false, then at least one of the Hodge, Tate, or
standard conjectures would be false --- are not strong enough to allow to
construct a motivic paradise because, as in the above proposition, they only
allow us to prove that two objects are isomorphic whereas we need that they
are equal. Nevertheless, they can be useful, for example, they allowed
Langlands and Rapoport to state their conjecture, and perhaps can be combined
with geometric results to allow us to obtain the true results.
\end{note}

\subsection{Tannakian theory and cohomology (2)}

Recall that we have a morphism of bands $P\rightarrow G$, and hence an action
of $P$ on $\Mot^{w}(\mathbb{Q}{}^{\mathrm{al}})$. A quotient morphism
$R\colon\CM(\mathbb{Q}{}^{\mathrm{al}})\rightarrow\Mot(\mathbb{F}{})$
corresponds to a $\mathbb{Q}{}$-valued fibre functor $\omega_{0}$ on
$\Mot(\mathbb{F}{})^{P}$. To extend $R$ to a quotient morphism $\Mot^{w}%
(\mathbb{Q}{}^{\mathrm{al}})\rightarrow\Mot(\mathbb{F}{})$, we need to extend
$\omega_{0}$ to $\Mot^{w}(\mathbb{Q}{}^{\mathrm{al}})^{P}$. We briefly discuss
this problem.

We begin with some preliminaries. For simplicity, we assume that $k$ has
characteristic zero.

\begin{plain}
\label{s9}Let $G$ be a semisimple algebraic group over a field $k$, and let
$G_{1},\ldots,G_{n}$ be the set of almost-simple normal algebraic subgroups of
$G$. The multiplication map%
\[
G_{1}\times\cdots\times G_{n}\rightarrow G
\]
is a central isogeny (\cite{milne2017c}, 24.3). If $G$ is simply connected,
then the map is an isomorphism and each $G_{i}$ is simply connected; moreover,
every connected normal subgroup $N$ of $G$ is a product of some of the $G_{i}%
$, and so both $N$ and $G/N$ are simply connected semisimple algebraic groups.
\end{plain}

Now let $G$ be a reductive group. Let $T$ be a torus in $G$, and let $N$ be
the smallest normal subgroup of $G$ containing $T$ (intersection of the normal
algebraic subgroups of $G$ containing $T$).

\begin{plain}
\label{s10}The algebraic group $N$ is connected. Indeed, $N^{\circ}$ contains
$T$ and is normal in $G$ (because it is a characteristic subgroup of $N$, ibid.~1.52).
\end{plain}

\begin{plain}
\label{s11}The algebraic group $N$ is reductive. We may suppose that $k$ is
algebraically closed. Let $R_{u}(N)$ be the unipotent radical of $N$ (largest
connected normal unipotent algebraic subgroup of $N$). From its definition,
$R_{u}(N)$ is clearly a characteristic subgroup of $N$, and so it is normal in
$G$. Thus $R_{u}(N)\subset R_{u}(G)=e$.
\end{plain}

Now assume that $G^{\mathrm{der}}$ is simply connected.

\begin{plain}
\label{s12}We have $N^{\mathrm{der}}\subset G^{\mathrm{der}}$ (obviously).
Therefore $N^{\mathrm{der}}$ is simply connected, and $G^{\mathrm{der}%
}/N^{\mathrm{der}}$ is simply connected. If $N^{\mathrm{der}}=G^{\mathrm{der}%
}\cap N$, then we have an exact sequece%
\[
1\rightarrow G^{\mathrm{der}}/N^{\mathrm{der}}\rightarrow G/N\rightarrow
S/\pi_{0}(N)\rightarrow1
\]
with $G/N^{\mathrm{der}}=(G/N)^{\mathrm{der}}$ and simply connected.
\end{plain}

\begin{plain}
\label{s13}Is $G^{\mathrm{der}}\cap N$ connected? If so, then, as it is normal
in $G^{\mathrm{der}}$, the quotient $G^{\mathrm{der}}/G^{\mathrm{der}}\cap N$
is simply connected. Now we can use the exact sequence%
\[
1\rightarrow G^{\mathrm{der}}/G^{\mathrm{der}}\cap N\rightarrow G/N\rightarrow
S/\pi_{0}(N)\rightarrow1
\]
in the next section. Alas, this doesn't seem to be always true. Consider, for
example, $G=\GL_{2}$, and $T=\mathbb{G}_{m}$ the centre of $G$. Then $N=T$ and
$G^{\mathrm{der}}\cap N=\mu_{2}$. Hence $G^{\mathrm{der}}/G^{\mathrm{der}}\cap
N=\SL_{2}/\mu_{2}=\PGL_{2}$, which is not simply connected..
\end{plain}

Let $\mathbb{Q}{}^{\mathrm{al}}$ be the algebraic closure of $\mathbb{Q}{}$ in
$\mathbb{C}{}$, and let $G=\mathcal{A}{}ut^{\otimes}(\omega_{B})$, where
$\omega_{B}$ is the Betti fibre functor.

\begin{theorem}
[?]\label{s14}Let $\omega_{0}$ be the $\mathbb{Q}{}$-valued fibre functor on
$\CM(\mathbb{Q}^{\mathrm{al}})^{P}$ defined by the quotient functor
$R\colon\CM(\mathbb{Q}{}^{\mathrm{al}})\rightarrow\Mot(\mathbb{F})$ (see
\ref{t2}). Up to isomorphism, $\omega_{0}$ extends to a $\mathbb{Q}{}$-valued
fibre functor $\omega$ on $\Mot^{w}(\mathbb{Q}{}^{\mathrm{al}})^{P}$ such that%
\[
\omega\otimes\mathbb{Q}{}_{l}\approx\xi_{l}%
\]
for all $l$.
\end{theorem}

\begin{proof}
Let $N$ be the normal closure of $P$ in $G$. Then $\Mot^{w}(\mathbb{Q}%
{}^{\mathrm{al}})^{P}=\Mot^{w}(\mathbb{Q}{}^{\mathrm{al}})^{N}$. Using the
Betti fibre functors, we can identify fibre functors with torsors, which are
classified by their cohomology classes. When we do this, the theorem becomes
the following statement: let $\gamma_{0}$ be class of $\omega_{0}$ in
$H^{1}(\mathbb{Q},S/P)$; then $\gamma_{0}$ lifts to a class in $\gamma$ in
$H^{1}(\mathbb{Q},G/N)$ such that $\gamma_{l}=\gamma_{l}$ for all $l$
(including $p$ and $\infty$). Let $G^{\prime}$ be the derived group of $G$ and
let $N^{\prime}=G^{\prime}\cap N^{\prime}$. From the snake lemma applied to%
\[
\begin{tikzcd}
&N^{\prime}\arrow{r}\arrow{d}
&N\arrow{r}\arrow{d}&P\arrow{r}\arrow[hook]{d}&1\\
1\arrow{r}&G^{\prime}\arrow{r}&G\arrow{r}&S,
\end{tikzcd}
\]
we get an exact sequence%
\[
1\rightarrow G^{\prime}/N^{\prime}\rightarrow G/N\rightarrow S/P\rightarrow1.
\]
\textit{We assume that} $G^{\prime}/N^{\prime}$ \textit{is simply connected}.
Consider the diagram%
\[
\begin{tikzcd}
H^{1}(\mathbb{Q},G^{\prime}/N^{\prime})\arrow{r}\arrow{d}
&H^{1}(\mathbb{Q},G/N)\arrow{r}\arrow{d}
&H^{1}(\mathbb{Q}{},S/P)\arrow{d}\\
H^{1}(\mathbb{R}{},G^{\prime}/N^{\prime})\arrow{r}
&\prod_{l}H^{1}(\mathbb{Q}_{l},G/N)\arrow{r}
&\prod_{l}H^{1}(\mathbb{Q}_{l},S/P).
\end{tikzcd}
\]
As $G^{\prime}/N^{\prime}$ is simply connected, $H^{1}(\mathbb{Q}{}%
_{l},G^{\prime}/N^{\prime})=0$ for all $l\neq\infty$ and the first vertical
arrow is an isomorphism (\cite{milne2017c}, 25.61, 25.63). Moreover, the map
$H^{1}(\mathbb{Q}{},G/N)\rightarrow H^{1}(\mathbb{Q}{},S/P)$ is surjective by
the theorem of Douai. Start with the given element $a$ of $\prod_{l}%
H^{1}(\mathbb{Q}_{l},G/N)$. Its image $b$ in $\prod_{l}H^{1}(\mathbb{Q}%
_{l},S/P)$ is the image of an element $c$ of $H^{1}(\mathbb{Q},S/P)$.
According to Douai's theorem, this lifts to an element $d$ of $H^{1}%
(\mathbb{Q},G/N)$. Let $d^{\prime}$ be the image of $d$ in $\prod_{l}%
H^{1}(\mathbb{Q}_{l},G/N)$. Now $a-d^{\prime}$ is the image of an element $e$
in $H^{1}(\mathbb{R}{},G^{\prime}/N^{\prime})$, which in turn is the image of
an element $f$ in $H^{1}(\mathbb{Q},G^{\prime}/N^{\prime})$. Let $f^{\prime}$
be the image of $f$ in $H^{1}(\mathbb{Q},G^{\prime}/N^{\prime})$. Now
$d+f^{\prime}$ is the element sought.
\end{proof}

Of course, we want to extend $\omega_{0}$ (on the nose) to a fibre functor
$\omega$ on $\Mot^{w}(\mathbb{Q}^{\mathrm{al}})^{P}$ such that $\omega
\otimes\mathbb{Q}_{l}$ and $\xi_{l}$ are equal (not simply isomorphic).

We now drop the assumption that $G/N^{\prime}$ is simply connected. Assume
$G^{\prime}$ is simply connected. We can still sometimes apply Douai's theorem.

\bibliographystyle{cbe}
\bibliography{svp}

\end{document}